\documentclass[11pt]{amsart}
\usepackage{amsmath}
\usepackage{amsthm}
\usepackage{tikz}
\usepackage{amssymb}
\usepackage{graphicx}
\usepackage{amsfonts}
\usepackage{cases}
\usepackage[margin=0.75in]{geometry}
\usepackage{bm}
\usepackage{calc}
\usepackage{hyperref}
\usepackage{cleveref}
\usepackage{lipsum}
\usepackage{color} 
\usepackage{cancel}
\usepackage{pgfplots}
\pgfplotsset{compat=1.15}
\usepackage{mathrsfs}
\usetikzlibrary{arrows}

\hypersetup{
    colorlinks=true, 
    linktoc=all,     
    linkcolor=blue,  
    citecolor = red
}
\newtheorem{Proposition}{Proposition}[subsection]
\newtheorem{Lemma}[Proposition]{Lemma}

\newtheorem{Theorem}[Proposition]{Theorem}

\newtheorem{Corollary}[Proposition]{Corollary}
\newtheorem{Claim}[Proposition]{Claim}
\newtheorem{Remark}[Proposition]{Remark}

\newtheorem{Definition}[Proposition]{Definition}
\numberwithin{Proposition}{subsection}
\newtheorem*{Notation*}{Notation}

\newtheorem{MainTheorem}{Theorem}

\newcommand{\Z}{\ensuremath{\mathbb{Z}}}
\newcommand{\N}{\ensuremath{\mathbb{N}}}
\newcommand{\R}{\ensuremath{\mathbb{R}}}

\newcommand{\LP}{\ensuremath{\mathbb{P}}} 
\newcommand\norm[2]{\left\Vert#1\right\Vert_{#2}} 

\newcommand{\floor}[1]{\left\lfloor #1 \right\rfloor}

\usepackage{mathtools}
\makeatletter
\DeclareRobustCommand\widecheck[1]{{\mathpalette\@widecheck{#1}}}
\def\@widecheck#1#2{%
    \setbox\z@\hbox{\m@th$#1#2$}%
    \setbox\tw@\hbox{\m@th$#1%
       \widehat{%
          \vrule\@width\z@\@height\ht\z@
          \vrule\@height\z@\@width\wd\z@}$}%
    \dp\tw@-\ht\z@
    \@tempdima\ht\z@ \advance\@tempdima2\ht\tw@ \divide\@tempdima\thr@@
    \setbox\tw@\hbox{%
       \raise\@tempdima\hbox{\scalebox{1}[-1]{\lower\@tempdima\box
\tw@}}}%
    {\ooalign{\box\tw@ \cr \box\z@}}}
\makeatother

\definecolor{qqqqff}{rgb}{0,0,1}
\definecolor{ffqqqq}{rgb}{1,0,0}
\definecolor{qqwuqq}{rgb}{0,0.59215686274509803,0}
\definecolor{ffzztt}{rgb}{1,0.6,0.2}
\definecolor{uuuuuu}{rgb}{0.26666666666666666,0.26666666666666666,0.26666666666666666}
\definecolor{ffdxqq}{rgb}{1,0.84313725490191208,0}
\definecolor{xfqqff}{rgb}{0.4980392156862745,0,1}
\definecolor{qqzzff}{rgb}{0,0.6,1}

\title[Quantitative estimates for (NSE) with incorporated forcing and applications]{Quantitative estimates for the forced Navier-Stokes equations and applications}
\author{Tobias Barker, Henry Popkin}
\thanks{TB is supported by an EPSRC New Investigator Award UKRI096 `Dynamics and regularity criteria for nonlinear incompressible partial differential equations'. HP is supported by Raoul \& Catherine Hughes (Alumni funds) and the University Research Studentship award EH-MA1333}
\date{\today}

\begin{document}
 \maketitle
\begin{center}
University of Bath,
Department of Mathematical Sciences,
BA2 7AY\\ \ 
tb2130@bath.ac.uk , hp775@bath.ac.uk \\ \
\end{center}
\hrule
\bigskip
\textsc{Abstract:} \textit{In this paper, we prove a localisation of a slightly supercritical (Orlicz) regularity criterion for the 3D incompressible Navier-Stokes equations. This is a refinement to the recent partial positive answer to Tao's conjecture \cite{Tao21} as given in \cite{Mildcriticality}. The proof requires new quantitative estimates for critically bounded solutions of the forced Navier-Stokes equations, where the forcing is induced by the localisation. A by-product of these new estimates is an application to the Boussinesq equations, where we prove a quantitative blow-up rate for the critical $L^3$ norm of the velocity. We prove these quantitative estimates using Carleman inequalities as in \cite{Tao21}, and subsequently in \cite{Spatialconc}, with an additional forcing term. An obstacle to doing this is that, in the Carleman inequalities, the forcing term is amplified on large scales. Additionally, the low regularity of the forcing requires the addition of Caccioppoli-type estimates to deal with the Carleman inequalities appropriately. 
}
\\
\hrule

\pagenumbering{roman}

\tableofcontents

\pagebreak
\pagenumbering{arabic}
\section{Introduction}
In this paper, we consider the three-dimensional incompressible Navier-Stokes equations over $\R^3$: 
\begin{equation}\label{Navier-Stokes equations}
    \partial_t v -\Delta v +(v \cdot \nabla)v + \nabla p =0, \qquad\qquad  \text{div}(v) = 0. \tag{NSE}
\end{equation}
and the Boussinesq equations:
\begin{align}
    \partial_t v - \Delta v + v\cdot\nabla v + \nabla p &= \theta \mathbf{e}_3, \label{Bous} \tag{Bous.}
    & \partial_t \theta - \Delta\theta + v\cdot\nabla \theta &= 0, &
     \operatorname{div}(v) &= 0.
\end{align}
The focus of this paper is two-fold; the first and main focus is to localise the slightly supercritical Orlicz blow-up criteria of \cite{Mildcriticality}, which gives a partial answer of a conjecture made in \cite{Tao21}. Generally, localisations of blow-up criteria are more tractable for numerics, as discussed in \cite{computerblowup}. For the Navier-Stokes equations, localisations naturally introduce an additional forcing term \cite{TruncationMethod}. We will then require new quantitative estimates for the \textit{forced} Navier-Stokes equations, analogous to those of \cite{Tao21} used subsequently in \cite{Mildcriticality}. This is where the second focus of the paper becomes apparent; using these new quantitative estimates, one may also prove (analogously to \cite{Tao21}) critical blow-up rates for equations that are closely related to \eqref{Navier-Stokes equations}, such as the Boussinesq equations.

\par Blow-up criteria for Leray-Hopf weak solutions have been widely investigated in the literature; we non-exhaustively list some subcritical\footnote{Subcritical/supercritical quantities scale to a positive/negative power of $\lambda$ under the Navier-Stokes rescaling (\ref{NSrescale}). Critical quantities are scale invariant.} \cite{Leray,Sohr01} and critical \cite{Prodi, Serrin62, Ladyzhenskaya, ESS03,Phuc15, ChoeWolfYangLorentz, WangZhangBesovBlow, GKPBesov, AlbBesovBlowup, AlbrBarkerBesov} blow-up criteria. Criticality is defined with respect to the following rescaling property of \eqref{Navier-Stokes equations}: if $(v, p)$ is a solution to \eqref{Navier-Stokes equations} with initial data $u_0$ and forcing $F$, then the rescaled functions \begin{align} \label{NSrescale} v_\lambda(x,t) := \lambda v(\lambda x,\lambda^2t) \qquad \text{and} \qquad  p_\lambda(x,t):= \lambda^2p(\lambda x,\lambda^2t)\end{align}
solve \eqref{Navier-Stokes equations} with initial data $v_{0,\lambda}(x) := \lambda v_0(\lambda x) $ and forcing $F_\lambda(x,t):=\lambda^3 F(\lambda x,\lambda^2t).$ 
\par The seminal work of \cite{ESS03} showed that if $v$ is a Leray-Hopf weak solution and $T^*$ is the first time that $v$ loses smoothness\footnote{This will often be referred to as the ``first blow-up time"; the first time $T^*$ such that there exists a point $x_0\in \R^3$ such that $v \not\in L^\infty(B_r(x_0)\times (T^*-r^2,T^*))$ for all $r>0$.}, then 
\begin{align}
    \limsup\limits_{t\uparrow T^{*}}\norm{v(\cdot,t)}{L^3(\R^3)}=\infty.
\end{align}
The blow-up of the full limit was shown in \cite{Ser2012}. Note that this result is qualitative; this is due to the use of compactness arguments in conjunction with qualitative backwards uniqueness and unique continuation results for parabolic differential inequalities. Later, Tao \cite{Tao21} quantified this result showing that there exists an absolute constant $c>0$ such that
\begin{align}
    \limsup\limits_{t\uparrow T^*}\frac{\norm{v(\cdot,t)}{L^3(\R^3)}}{\left(\log\log\log\left(\frac{1}{T^*-t}\right)\right)^c}=\infty. \label{Tao Quantitative blowup rate}
\end{align}
The key to proving this result was to replace the qualitative arguments of \cite{ESS03} with quantitative counterparts - in particular, one can use Carleman inequalities (see Section \ref{Carleman Section}) in place of the backwards uniqueness and unique continuation results for parabolic differential inequalities; this leads to a quantitative regularity estimate of the form:
\begin{align}
     \norm{v}{L^\infty_t L^3_x(\R^3\times (0,T))}&\leq M \implies \norm{v(\cdot,t)}{L^\infty_{t,x}(\R^3)}\leq t^{-\frac{1}{2}}\exp\exp\exp\left( M^c\right).\label{Tao Quant reg estim}
\end{align}
We mention Tao's results were later refined via spatial concentration techniques in \cite{Spatialconc}, with a localised version of \eqref{Tao Quantitative blowup rate} established in \cite{Localizedblowup}. For further developments, see \cite{Palasek2021ImprovedQuantitative, OzanskiPalasek2022AxisymmetricWeakL3, Palasek2022MinimumCriticalBlowup, spatialsurvey, BarkerMiuraTakahashi2024HeatCriticalBlowUp, HuNguyenNguyenZhang2024EndpointBesov, Barker2025QuantitativeClassification}, for example. 

In \cite{Tao21}, Tao conjectured that if a solution first loses smoothness at time $T^*>0$, then the slightly supercritical Orlicz norm $\norm{v(\cdot,t)}{L^3(\log\log\log(L))^{-c}(\R^3))}$ must blow up as $t$ approaches $T^*.$ \cite{Mildcriticality} provided a partial positive answer to this (with an additional log), by  using an argument inspired by \cite{Bulut} for the non-linear Schrödinger equations. In particular, \cite{Mildcriticality} used a quantitative regularity estimate from \cite{Tao21} to show that there exists a universal constant $\theta\in(0,1)$ so that, if a weak Leray-Hopf solution $v$ first blows-up at time $T^*$, then
\begin{align}
    \limsup_{t\uparrow T^*}\int_{\R^3}\dfrac{|v(x,t)|^3}{\left(\log\log\log\left(\left(\log\left(e^{e^{3e^e}}+|v(x,t)|\right)\right)^\frac{1}{3}\right)\right)^\theta}\;dx = \infty. \label{Original Orlicz}
\end{align} 
We also mention other slightly supercritical regularity criteria for the Navier-Stokes equations. For example, the non-endpoint Prodi-Serrin-Ladyshenskaya (PSL) criteria have been logarithmically improved in \cite{MontgomerySmith,ChanVassLog, ZhouLeiLogProdi,  BjorVassLogProdi}. Most of these papers use energy estimates and the Osgood inequality; these methods are not available for the endpoint case (\ref{Original Orlicz}).

The first purpose of this paper is to localise (\ref{Original Orlicz}) to prove:
\begin{MainTheorem}\label{localorliczblowupnoforce}
   There exists a universal constant $\theta\in (0,1)$ such that the following holds. Suppose that $v:\R^3\times (0,\infty)\to\R^3$ is a suitable Leray-Hopf weak solution of the Navier-Stokes equations on $\R^3\times (0,\infty)$ which first blows up at time $T^*>0,$ and let $(x_*,T^*)$ be a singular point\footnote{We call $(x_*,T^*)$ a singular point of $v$ if $v\notin L^\infty(B_r(x_*)\times (T^*-r^2,T^*))$ for all sufficiently small $r>0$.}. Then for all $\delta>0,$
    \begin{align}
        \limsup_{t\uparrow T^*}\int_{B_\delta(x_*)}\dfrac{|v(x,t)|^3}{\left(\log\log\log\left(\left(\log\left(e^{e^{3e^e}}+|v(x,t)|\right)\right)^\frac{1}{3}\right)\right)^\theta}\;dx = \infty. \label{orlicznormblowuploc}
    \end{align}
\end{MainTheorem}
To prove the global version of this result, the original paper \cite{Mildcriticality} uses a continuity-type argument based on an idea of Bulut \cite{Bulut}; here, the idea is to (iteratively) apply a \textit{global} $L^4(\R^3)$ (subcritical) energy estimate which looks similar\footnote{Technically, the inequality involves the $L^4$ energy defined in Section \ref{mildcritbreakingsection} rather than just the $L^4$ norms.} to
\begin{equation}
        \norm{v(\cdot,t)}{L^4(\R^3)}^4 \lesssim \norm{v(\cdot,t_0)}{L^4(\R^3)}^4+ \norm{v}{L^5(\R^3\times(t_0,t_0+S))}\norm{v}{L^\infty_tL^4_x(\R^3\times (t_0,t_0+S))}^4 \quad \text{for all}\; t\in (t_0,T_0+S).\label{L4energyestim}
    \end{equation}
    Meanwhile, Tao's quantitative regularity estimate (\ref{Tao Quant reg estim}) (adapted to estimate the $L^5$ norm) and Lebesgue interpolation gives:
    \begin{align}
    \norm{v}{L^5_{t,x}}\lesssim \exp\exp\exp(\norm{v}{L^\infty_tL^3_x}^c)\lesssim \exp\exp\exp(\norm{v}{L^\infty_tL^{3-\mu}_x}^{c\theta_\mu}\norm{v}{L^\infty_tL^{4}_x}^{c(1-\theta_\mu)}). \label{TaoL5}
\end{align}
Under the converse of (\ref{Original Orlicz}), one may carefully tune $\mu\in(0,\frac{1}{2}],$ and use (\ref{TaoL5}) to close the iterative energy estimate (\ref{L4energyestim}) to show that regularity holds up to time $T^*$. See \cite{spatialsurvey} for further discussion.
\par This paper uses the same method inspired by \cite{Bulut}, subsequently used by \cite{Mildcriticality}; however, the energy inequality (\ref{L4energyestim}) is a global quantity. Typically, to localise a blow-up criterion, one utilises a Bogovskii truncation procedure (originally seen in \cite{TruncationMethod}). This transforms the \textit{local} problem for the Navier-Stokes equations into a \textit{global} problem for the Navier-Stokes equations with forcing $F\in L^2_tH^1_x\cap L^6_{t,x}$ in our setting. This difficulty faced with the incorporation of the forcing term is that we no longer have access to the quantitative regularity estimate (4) from \cite{Tao21}. Thus, we are required to prove a quantitative regularity estimate for critically bounded solutions of the \textit{forced} Navier-Stokes equations. The associated difficulties and strategy used are discussed in detail in the next section.
\par The second goal of this paper is to prove a slightly supercritical regularity criterion for physical systems that are coupled to the Navier-Stokes equations through a forcing term. For example, we consider the Boussinesq equations:
\begin{align}
    \partial_t v - \Delta v + v\cdot\nabla v + \nabla p &= \theta \mathbf{e}_3, \label{Bous1 - forced NSE}
    & \partial_t \theta - \Delta\theta + v\cdot\nabla \theta &= 0,
    \\ \operatorname{div}(v) &= 0, & (v(\cdot,0),\theta(\cdot,0))&= (v_0,\theta_0).
\end{align}
Here, $v$ is a vector field which represents the velocity field of the fluid, $p$ and $\theta$ are scalar fields representing pressure and temperature (or density in some applications) of the fluid respectively, and $\mathbf{e}_3$ is the standard basis vector $(0,0,1)^T.$ Considering (\ref{Bous1 - forced NSE}) as a forced Navier-Stokes equation with forcing $F=\theta \mathbf{e}_3$, one may see (by a simple $L^p$-energy estimate on $\theta$) that we may bound $F$ in $L^2_tH^1_x\cap L^6_{t,x}$ \textit{independently} of $v$. We may then instantly recover the quantitative estimate (used for Theorem \ref{localorliczblowupnoforce}) for the Boussinesq equations. This then allows one to show a quantitive blow-up rate for the critical norm of $v$ near a singularity, in exactly the same manner as \cite{Tao21} for the Navier-Stokes equations:
\begin{MainTheorem}\label{Qunatitiative critical blow-up for Bous}
    Suppose that $(v,\theta, p)$ is a smooth solution to the Boussinesq equations on $\R^3\times (0,T^*)$ in the energy class, corresponding to initial data $v_0\in L^2\cap L^\infty$ and $\theta_0\in L^2\cap L^\infty$. Moreover, suppose that $v$ blows-up at $T^*.$
    Then we have the following quantitative blow-up rate for the $L^3$ norm of $v$:
    \begin{align}
        \limsup\limits_{t\uparrow T^*}\dfrac{\norm{v(\cdot,t)}{L^3(\R^3)}}{\left(\log\log\log\left(\dfrac{1}{(T^*-t)^\frac{1}{4}}\right)\right)^\frac{1}{614}} = \infty.
    \end{align}
\end{MainTheorem}
\begin{Remark}
    We mention that modifications may be made to the Bogovskii truncation procedure (and minor modifications thereafter) to prove the blow-up of the local Orlicz quantity (\ref{orlicznormblowuploc}) at the first singular time for a Leray-Hopf weak solution to the Navier-Stokes equations with forcing $f\in L^2_tH^1_x\cap L^q_{t,x}$ (where $q>5$). In particular, it instantly follows from these modifications that we have the blow-up of the local Orlicz quantity (\ref{orlicznormblowuploc}) for the Boussinesq equations. 
\end{Remark} 
\subsection{Comparison to previous literature and strategy} We utilise the spatial concentration method of \cite{Spatialconc}, which is a physical-space analogue of \cite{Tao21}. See the recent survey paper \cite{spatialsurvey} for an exposition of these methods. The main difficulty faced in this paper lies in incorporating forcing into the critical quantitative regularity estimate (see Section \ref{MainProof}). In \cite{Localizedblowup}, forcing is also incorporated into the spatial concentration method, but in that setting, the quantitative estimates are carried out far away from the support of the forcing. This is not possible in our setting. By rescaling, the domain of our setting is $\R^3\times (-1,0)$, and we assume that $\norm{v}{L^\infty_t L^3_x(\R^3\times (-1,0))} \leq M$ and $\norm{F}{L^2_tH^1_x\cap L^6_{t,x}(\R^3\times (-1,0))}$ is sufficiently small (as described later). Assuming that the vorticity concentrates at a time $t_0\in(-1,0)$: \begin{align}
    \int_{B(0,{M^{c}(-t_0)^\frac{1}{2})}} |\omega(x,t_0)|^2\;dx\geq \frac{M^{-c}}{\sqrt{-t_0}}, \label{lowerbndomega}
\end{align}
the aim of the method is to use inequality arguments from \cite{Tao21} (with incorporated forcing) to deduce a bound at time $t=0$:
\begin{align}
        \int\limits_{B\left(\exp\left(M^{c'}T^\frac{1}{2}\right)\right)\setminus B(T^\frac{1}{2})}|v(x,0)|^3\;dx \geq \exp\left(-\exp\left(M^{c'}\right)\right) . \label{introL3annuli}
\end{align}
Here, $T$ depends on $t_0$ as described in Proposition \ref{Backwards prop of vort conc}. As in \cite{Tao21} and subsequently \cite{Spatialconc, Localizedblowup}, summing (\ref{introL3annuli}) over suitable scales so that the annuli $B\left(\exp\left(M^{c'}T^\frac{1}{2}\right)\right)\setminus B(T^\frac{1}{2})$ are disjoint, one obtains a quantitative lower bound on $-t_0$. 
As mentioned previously, these methods make use of two Carleman inequalities; these may only be applied on regions where the vorticity satisfies a differential inequality of the form:
\begin{align}
    |\partial_t\omega-\Delta \omega|(x,t) \leq \frac{|\omega(x,t)|}{C_{Carl} T} +\frac{|\nabla \omega(x,t)|}{(C_{Carl} T)^\frac{1}{2}} + |\nabla\times F(x,t)|. \label{maindiffineq}
\end{align}
These inequalities use the vorticity structure for the Navier-Stokes equations.
To demonstrate the difficulty with incorporating forcing into these methods, let us first discuss ``Step 1", the application of the first Carleman inequality corresponding to unique continuation of parabolic operators (see Proposition \ref{FirstCarlemanInequality} for a precise statement). Informally, the first Carleman inequality is most useful when it is applied on large scales $R\gg1$; in which case, rearranging the inequality (without the presence of force) roughly gives:
\begin{align}
    T_1^{-\frac{1}{2}}e^{-\frac{C_MR^2}{T_1}}\lesssim\int_{-T_1}^{-\frac{T_1}{2}}\int_{B_{R}\setminus B_\frac{R}{2}}|\omega(x,s)|^2\;dxds. \label{Gaussian lower bound}
\end{align}
Here, $C_M$ is a large constant depending on $M$ (and $T_1$ is defined with respect to $t_0$ via Proposition \ref{Backwards prop of vort conc}). The bound (\ref{Gaussian lower bound})  is key to obtain (\ref{introL3annuli}). Indeed, one obtains (\ref{introL3annuli}) by combining Step 1 with the second Carleman inequality\footnote{The second Carleman inequality may be applied on regions where the differential inequality of the same form as (\ref{maindiffineq}) holds; we will discuss this more shortly.} corresponding to backwards uniqueness of parabolic operators (see Proposition \ref{SecondCarlemanInequality}), and then finally using the first Carleman inequality again.
When forcing is incorporated, both Carleman inequalities (and hence line (\ref{Gaussian lower bound})) gain an additional term which looks similar to:
\begin{align}
    Te^\frac{C'_MR^2}{T}\norm{\nabla\times F}{L^2(B_{R}\times(-T,0))}^2.\label{problemterm}
\end{align}
Clearly, this term becomes unbounded on arbitrarily large scales $R\gg1$. Thus, to obtain the Gaussian lower bound (\ref{Gaussian lower bound}), the forcing term may not be neglected if $R$ is left unchecked as in \cite{Tao21, Spatialconc} . 
To overcome this obstacle, we make the following observation; if the forcing is small enough\footnote{Indeed, we may assume this via a rescaling argument since $\norm{F}{L^2_t\dot{H}^1_x}$ is subcritical.} (in $L^6_{t,x}$) then the scales on which the second Carleman inequality may be applied where (\ref{maindiffineq}) holds is \textit{not} dependent on the forcing. Moreover, the lower bound (\ref{lowerbndomega}) is also independent of small forcing in $L^6_{t,x}$. Then the second Carleman inequality may be applied where the differential inequality (\ref{maindiffineq}) holds, over an ``annulus of regularity" (see Proposition \ref{ann of reg M version})
\begin{align} \big\lbrace 4\hat{R} < |x| < \frac{M^{1200}\hat{R}}{16} \big\rbrace \times \left(-\frac{T_2}{M^{200}},0\right) \quad \text{where}\; \hat{R}\in [2MT_2^\frac{1}{2},MT_2^\frac{1}{2}\exp{\left( M^{607}\right)}] \label{introanofreg}\end{align}
Thus, (\ref{introanofreg}) gives an upper bound for the scale $R$ where both Carleman inequalities may be applied, hence we know exactly how much the problematic forcing terms (\ref{problemterm}) are amplified in all the Carleman inequalities. We may then use that $\norm{\nabla\times F}{L^2(\R^3\times(-1,0))}$ is subcritical to make\footnote{Via a rescaling argument.} it sufficiently small to counteract this amplification. This then means that the problematic forcing terms (\ref{problemterm}) are negligible compared to the desired Gaussian lower bound; this allows us to use the Carleman inequality with forcing to obtain the desired Gaussian lower bound (\ref{Gaussian lower bound}). Similarly, we utilise the subcritically of the forcing term to counteract the amplification of the forcing terms in subsequent applications of the Carleman inequalities. 
\par Finally, we mention that in our setting, $F$ is only quantitatively controlled in lower regularity spaces. Thus, we only have quantitative control of lower regularity quantities when estimating terms in the Carleman inequalities. This is in contrast to the works of \cite{Tao21, Spatialconc, Localizedblowup} where second derivatives of the velocity are bounded in the region of regularity. We utilise Caccioppoli-type inequalities to circumvent the use of these higher regularity bounds when estimating certain terms in the Carleman inequalities. 

\subsection{Outline of the paper} Section \ref{Tools} contains the tools required for the proof of the main quantitative regularity estimate in Section \ref{MainProof}. This will be used to prove Theorem \ref{localorliczblowupnoforce} in Section \ref{Orlicz Proof}. Finally, we use the main quantitative estimate again in Section \ref{Bouss Aplic} to prove Theorem \ref{Qunatitiative critical blow-up for Bous}.

\subsection{Notation and definitions} \label{Definitions}
\subsubsection{Notation} 
We adopt the convention where $C$ denotes a positive universal constant which may change from line to line unless made otherwise clear; dependence on parameters $a,b,c...$ is denoted via subscript, $C_{a,b,c...}$ . We often use the shorthand $X \lesssim_{a,b,c...} Y$ if there exists a constant $C>0$ depending on parameters $a,b,c...$ such that $X \leq C_{a,b,c...}Y.$ The notation $X \sim_{a,b,c...} Y$ means that both $X \lesssim_{a,b,c...} Y$ and $X \gtrsim _{a,b,c...} Y.$ 
\par We say that a property holds for all $M$ ``sufficiently large'' if there exists a universal constant $M_0>0$ such that the property holds for all $M\geq M_0.$ Similarly, we say that a property holds for a sufficiently small $\varepsilon$ if there exists a universal constant $\varepsilon_0\in (0,1)$ such that the property holds for all $0<\varepsilon\leq\varepsilon_0.$
\par We also adopt Einstein summation convention. In the standard basis $(\textbf{e}_i)$ for $\R^3$, the tensor product of two vectors $a=a_i\textbf{e}_i,$ $b=b_i\textbf{e}_i$ is the matrix $a\otimes b$ with entries $(a\otimes b)_{ij} := a_ib_j.$
\\For $x_0\in \R^3$, $r>0$ we denote the ball
$B_r(x_0):= \left\lbrace x\in \R^3 : |x-x_0|<r \right\rbrace.$
We often write $B_r := B_r(0)$ or $B(0,r):=B_r(0)$ for typesetting purposes.
\par For $(x_0,t_0) \in \R^3 \times \R$, $r>0$ we denote the parabolic cylinder $Q_r(x_0,t_0):= B_r(x_0) \times (t_0-r^2,t_0).$
We often write $Q_r := Q_r(0,0)$.
\par We define $L^2_\sigma(\R^3)$ as the $L^2(\R^3)$-closure of the set of all divergence-free functions belonging to $C^\infty_0(\R^3)$ .
\par For any Banach space $(X, \norm{\cdot}{X})$, $a<b$, $p\in[1,\infty)$ we denote the mixed space $L^p(a,b;X)$ to be the Banach space of strongly measurable functions $f:(a,b)\to X$ such that
$$\norm{f}{L^p(a,b;X)}:= \left( \int_a^b \norm{f(t)}{X}^p\right)^\frac{1}{p} < \infty$$
with the usual modifications for $p=\infty$ or for locally integrable spaces. 
\\For $\Omega \times(a,b) \subset \R^3\times \R$, and $X(\Omega)$ is a Banach space of functions defined on $\Omega$, we will often denote $$L^p_tX_x(\Omega \times (a,b)) := L^p(a,b;X(\Omega))$$
If  $Y(\Omega)$ is also a Banach space of functions defined on $\Omega$, we also denote intersections of mixed spaces $$L^p_tX_x\cap L^q_tY_x(\Omega \times (a,b)):= L^p(a,b;X(\Omega))\cap  L^q(a,b;Y(\Omega)).$$
Here, we equip $L^p_tX_x\cap L^q_tY_x(\Omega \times (a,b))$ with the canonical intersection norm.
\\Let $C([a, b]; X)$ denote the space of continuous $X$ valued functions on $[a, b]$ with the usual norm. In addition, let $C_w([a, b]; X)$ denote the space of $X$ valued functions, which
are continuous from $[a, b]$ to the weak topology of $X$.
\\ The parabolic H\"older space with exponent $\alpha \in (0,1)$ is defined as the set of continuous functions $f:\Omega\times(a,b)\to \R^3$ for which the parabolic H\"older norm is finite:
\begin{align}
    \norm{f}{C^{0,\alpha}_{par}(\Omega\times(a,b))}:= \norm{f}{L^\infty(\Omega\times(a,b))}+ \sup\limits_{(x,t) \neq (y, s) \in \Omega\times(a,b) } \frac{|f(x,t)-f(y,s)|}{|x-y|^\alpha+|t-s|^\frac{\alpha}{2}}.
\end{align}
Finally, we use the following notion of solution throughout the paper:
\begin{Definition}\label{LerayHopfWeak}
  A function $v$ is a \textbf{Leray-Hopf weak solution} to the Navier-Stokes equations with forcing $f\in L^2(\R^3\times(0,\infty))$ on $\R^3\times(0,\infty)$ corresponding to initial data $v_0 \in L^2_\sigma(\R^3)$ if $v$ satisfies the equations in the sense of distributions and
    $$v \in C_w([0,\infty);L^2_\sigma(\R^3)) \cap L_{\textup{loc}}^2([0,\infty);\dot{H}^1(\R^3)), \qquad \text{and} \qquad
\lim_{t\to0^+}\norm{v(t)-v_0}{L^2} = 0. $$
Moreover, $v$ satisfies the global energy inequality for all $t\in(0,\infty)$:
\begin{equation} \label{energy inequality}
\norm{v(\cdot,t)}{L^2}^2+2\int_0^t\norm{\nabla v(\cdot,\tau)}{L^2}^2 d\tau \leq \norm{v_0}{L^2}^2.
\end{equation}

\end{Definition}
\begin{Definition} A distributional solution $(v,p)$ to the forced Navier-Stokes equations with forcing $f\in L^2((a,b);L^2(\Omega))$ is called a \textbf{suitable} weak solution to the Navier-Stokes equations on the domain $\Omega \times (a,b)\subseteq\R^4$ if $$v \in L^{\infty}((a,b);L^2(\Omega)) \cap L^2((a,b);\dot{H}^1(\Omega)), \qquad p\in L^\frac{3}{2}(\Omega \times (a,b)).$$ Moreover, $(v,p)$ satisfies the local energy inequality:  
\begin{multline} \label{LocalEnergyIneq}
\int_\Omega |v(t)|^2\phi(t)\;dx + 2\int_a^t\int_\Omega|\nabla v|^2\phi \;dxds \\ \leq \int_a^t\int_\Omega |v|^2(\partial_t \phi + \Delta \phi)\;dxds + \int_a^t\int_\Omega (|v|^2+2p)v\cdot\nabla\phi\; dxds +2\int_0^t\int_{\R^3}f\cdot v\phi\;dxds
\end{multline}
for all non-negative $\phi \in C_0^\infty(\Omega \times (a,\infty))$ and almost every $t\in(a,b)$.
\end{Definition}
\begin{Definition}\label{singularDefn}
    We say that a weak solution $v$ has a singular point at $(x_0,t_0)\in \R^3\times(0,\infty)$ if $v\notin L^\infty(Q_r(x_0,t_0))$ for all $r>0$ sufficiently small. We also call $t_0$ a blow-up time. 
\end{Definition}

\section{Ingredients for the proof of the main quantitative estimate}\label{Tools}
\subsection{Serrin-type bootstrapping with forcing}
\begin{Lemma}\label{CSYTLemma}
    Let $1\leq s,q<\infty$ and $G,H\in L^{s}_tL^{q}_x(Q_1)$. Assume that $W\in L^s_tL^1_x(Q_1)$ is a distributional solution to the Stokes system in $Q_1$:
    \begin{align}
        \partial_t W-\Delta W+\nabla P = \operatorname{div}(G) + H,
        \qquad \operatorname{div}(W)=0.
    \end{align}
    Then $W$ satisfies the inequality
    \begin{align}
        \norm{\nabla W}{L^s_tL^q_x(Q_\frac{3}{4})}\lesssim_{s,q} \norm{G}{L^s_tL^q_x(Q_1)} + \norm{H}{L^s_tL^q_x(Q_1)} +\norm{W}{L^s_tL^1_x(Q_1)}.
    \end{align}
\end{Lemma}
\begin{proof}
    Use similar arguments to \cite[Appendix A.2]{CSYTlowerbnd};  indeed, the case $H=0$ is contained there.
\end{proof}
\begin{Lemma}\label{Serrin-type bootstrap}
    Suppose that $v \in L^\infty(Q_1)$ is a distributional solution to the Navier-Stokes equations with forcing $F\in L^6(Q_1)$. Then
    \begin{align} \label{holdervorticity serrin}
        \norm{\nabla v}{L^\infty(Q_\frac{1}{2})} +\norm{\omega }{C^{0,\frac{1}{6}}_\text{par}(Q_\frac{1}{2})} \lesssim \left(1+\norm{v}{L^\infty(Q_1)}\right)\left( \norm{v}{L^\infty(Q_1)}^2 + \norm{v}{L^\infty(Q_1)}+\norm{F}{L^6(Q_1)}\right).
    \end{align}    
\end{Lemma}
\begin{proof}
    First, using Lemma \ref{CSYTLemma}, we have
    \begin{align}
        \norm{\nabla v}{L^6(Q_\frac{3}{4})} &\lesssim \norm{F}{L^6(Q_1)}+\norm{v\otimes v}{L^6(Q_1)} + \norm{v}{L^6_tL^1_x(Q_1)}
        &\lesssim \norm{F}{L^6(Q_1)}+\norm{v}{L^\infty(Q_1)}^2 + \norm{v}{L^\infty(Q_1)}. \label{maximalregforstokes}
    \end{align}
    We proceed by considering the vorticity equation:
    \begin{align}\partial_t \omega  -\Delta \omega = \text{div}\left(\omega \otimes v- v \otimes \omega\right) + \nabla \times F.\end{align}
    We use \cite[Proposition A.1]{NecasOnLeray} together with (\ref{maximalregforstokes}) to find:
    \begin{align}
        \norm{\omega}{C^{0,\frac{1}{6}}_{par}(Q_{\frac{2}{3}})}&\lesssim \norm{F}{L^6(Q_\frac{3}{4})}+\norm{\omega}{L^2(Q_\frac{3}{4})}+\norm{|\omega||v|}{L^6(Q_\frac{3}{4})}
        \lesssim \norm{F}{L^6(Q_1)}+\norm{\nabla v}{L^6(Q_\frac{3}{4})}(1+\norm{v}{L^\infty(Q_1)}) \label{holquantest1}
        \\ &\lesssim \left(1+\norm{v}{L^\infty(Q_1)}\right)\left( \norm{v}{L^\infty(Q_1)}^2 + \norm{v}{L^\infty(Q_1)}+\norm{F}{L^6(Q_1)}\right). \label{holquantest2}
    \end{align}
    Then using \cite[Proposition A.2]{NecasOnLeray} gives for $t\in(-\frac{1}{4},0)$
    \begin{align}
       \norm{\nabla v(\cdot,t)}{L^\infty(B_\frac{1}{2})}\lesssim \norm{\nabla v(\cdot,t)}{C^{0,\frac{1}{6}}(B_{\frac{1}{2}})}+\norm{\nabla v(\cdot,t)}{L^1(B_\frac{1}{2})}\lesssim\norm{\omega(\cdot,t)}{C^{0,\frac{1}{6}}(B_{\frac{2}{3}})} + \norm{v(\cdot,t)}{L^2(B_\frac{2}{3})}+\norm{\nabla v(\cdot,t)}{L^1(B_\frac{1}{2})}. \nonumber
    \end{align}
    Thus, from (\ref{maximalregforstokes}) and (\ref{holquantest1})-(\ref{holquantest2}),
    \begin{align}\nonumber
       \norm{\nabla v}{L^\infty(Q_\frac{1}{2})}&\lesssim  \norm{\omega}{C^{0,\frac{1}{6}}_{par}(Q_{\frac{2}{3}})} + \norm{v}{L^\infty(Q_\frac{2}{3})} +\norm{\nabla v}{L^6(Q_\frac{3}{4})}\\&\lesssim \left(1+\norm{v}{L^\infty(Q_1)}\right)\left( \norm{v}{L^\infty(Q_1)}^2 + \norm{v}{L^\infty(Q_1)}+\norm{F}{L^6(Q_1)}\right).
    \end{align}
\end{proof}

\begin{Corollary}\label{CKN higher order derivatives for forcing}
Let $0<R<1$ and suppose that $(v,p)$ is a suitable weak solution to the Navier-Stokes equations in $Q_R$ with forcing $F\in L^6(Q_R)$. There exists a universal constant $\varepsilon_*\in(0,1)$ such that for all $0<\varepsilon\leq \varepsilon_*$ the following holds true.
If:
   \begin{equation}\label{CKN forced assump}                R^{-2}\int_{Q_R}|v|^3+|p|^{3/2}\;dxds+R^{13}\int_{Q_R}|F|^6\;dxdt\leq \varepsilon
    \end{equation}
    then
    for $k=0,1:$
    \begin{align}\norm{\nabla^k v}{L^\infty(Q_{\frac{R}{4}})}&\lesssim_k \frac{\varepsilon^\frac{1}{6}}{R^{k+1}},  & \norm{\omega}{C^{0,\frac{1}{6}}_\text{par}(Q_\frac{R}{4})} &\lesssim \frac{\varepsilon^\frac{1}{6}}{R^{2+\frac{1}{6}}}.\end{align}
\end{Corollary}

\begin{proof}
    By rescaling according to (\ref{NSrescale}), we may assume that $R=1.$ By \cite[pp. 7-9]{CKN82}, provided that $0<\varepsilon\leq\varepsilon_*$ (where $\varepsilon_*$ is a sufficiently small universal constant), we then have that $\norm{v}{L^\infty(Q_\frac{1}{2})}\lesssim\varepsilon^\frac{1}{3}$.
    Now, as $\norm{F}{L^6(Q_1)}\leq \varepsilon^\frac{1}{6}$, we see that applying Lemma \ref{Serrin-type bootstrap} (rescaled) gives: 
    \begin{align}
        \norm{\nabla v}{L^\infty(Q_{\frac{1}{4}})}+\norm{\omega }{C^{0,\frac{1}{6}}_\text{par}(Q_\frac{1}{4})} \lesssim \varepsilon^\frac{1}{6}.
    \end{align}
\end{proof}
We also have the following parabolic Caccioppoli inequality:
\begin{Proposition}\label{CaccioppoliBall}
    Let $B_{(2)}\subset B_{(1)} \subset \R^3$ be two balls and  $I_{(2)}\subset I_{(1)} \subset \R$ be two bounded intervals.
    Suppose that $W, H:B_{(1)}\times I_{(1)}\to \R^3$, $ G:B_{(1)}\times I_{(1)}\to \R^{3\times 3}$ are smooth and satisfy
    \begin{align}
        \partial_tW-\Delta W = \operatorname{div}(G) + H.
    \end{align}
    Suppose also that $W,G,H\in L^2(I_{(1)}\times B_{(1)})$.
    Then W satisfies 
    \begin{align}
        \int_{I_{(2)}}\int_{B_{(2)}} |\nabla W|^2\;dxds &\lesssim\left(\frac{1}{\inf I_{(2)}-\inf I_{(1)}}+ \frac{1}{\mathrm{dist}(\partial B_{(1)},\partial B_{(2)})^2} \right)\left( \int_{I_{(1)}}\int_{B_{(1)}}|W|^2\;dxds\right) \\&\qquad\qquad+ \norm{G}{L^2(I_{(1)}\times B_{(1)})}^2+\norm{H}{L^2(I_{(1)}\times B_{(1)})}\norm{W}{L^2(I_{(1)}\times B_{(1)})}.
    \end{align}
\end{Proposition}
\begin{Remark}\label{CaccioppoliAnnulus}
    We also use the following version for nested annuli $A_{(2)}\subset A_{(1)}$ and $I=[a,b]\subset \R$:
     \begin{align}
        \int_{I}\int_{A_{(2)}} |\nabla W|^2\;dxds &\lesssim \int_{A_{(1)}} \left\vert|W(x,b)|^2-|W(x,a)|^2\right\vert\;dx+\frac{1}{\mathrm{dist}(\partial A_{(1)},\partial A_{(2)})^2} \int_{I}\int_{A_{(1)}}|W|^2\;dxds\\&\qquad\qquad+ \norm{G}{L^2(I\times A_{(1)})}^2+\norm{H}{L^2(I\times A_{(1)})}\norm{W}{L^2(I\times A_{(1)})}
        \\ &\lesssim |b-a|^\frac{\alpha}{2}|A_{(1)}|\cdot\norm{W}{L^\infty(I\times A_{(1)})}\norm{W}{C^{0,\alpha}_{par}(I\times A_{(1)})}+\frac{1}{\mathrm{dist}(\partial A_{(1)},\partial A_{(2)})^2} \int_{I}\int_{A_{(1)}}|W|^2\;dxds\\&\qquad\qquad+ \norm{G}{L^2(I\times A_{(1)})}^2+\norm{H}{L^2(I\times A_{(1)})}\norm{W}{L^2(I\times A_{(1)})}.
    \end{align}
\end{Remark}
\subsection{Local-in-space smoothing with incorporated forcing term}
Local-in-space smoothing without forcing is in the Kang, Miura and Tsai's paper \cite{KMT21}. In the following, we incorporate\footnote{Compare with Theorem 1.1, Remark 1.2 and Theorem 3.1 of \cite{KMT21}.} forcing into \cite{KMT21} and obtain an analogue of \cite[Corollary 1]{Localizedblowup}. For related results with divergence-form forcing, see \cite[Lemma 2.3]{MicropolarPartialRegularity}
\begin{Proposition}\label{KMTwithforce}
     There exists a positive constant $c_1\in (0,1)$ such that the following holds true. Let $M,N\geq 1$ be sufficiently large, let $0<T_1 \leq c_1M^{-18}N^{-52}$ and suppose that $(v,p)$ is a suitable weak solution to the Navier-Stokes equations with forcing $F$ on $B_4\times (0,T_1)$ such that
    \begin{align} \label{KMTintial}
        \lim_{t\to 0^+}\norm{v(\cdot,t)-v_0}{L^2(B_4)}=0 \qquad \text{with} \; v_0\in L^2(B_4),&
    \\
        \norm{v}{L^\infty_tL^2_x\cap L^2_t\dot{H}^1_x(B_4\times(0,T_1))}^2+\norm{p}{L^\frac{3}{2}(B_4\times(0,T_1))} \leq M,    &
        \\       \norm{F}{L^6(B_3\times(0,T_1))} +\norm{v_0}{L^6(B_3)}\leq N.\label{KMTassump}&
    \end{align}    
    Then 
    \begin{equation}
       |\nabla^jv(x,t)| \lesssim_j t^{-\frac{j+1}{2}} \qquad \text{for} \; (x,t)\in B_1\times (0,\,T_1) , \quad j =0,1.  \label{localinspacesmoothing}
    \end{equation}
\end{Proposition}
\begin{proof}
    \underline{\textbf{Step 1: Rescaling:}} Let $x_0\in B_1$ and rescale and translate to the origin; set $\tilde{v}(x,t):= \lambda v(x_0+\lambda x,\lambda^2t)$ for $(x,t)\in B_4\times (0,\lambda^{-2}T_1)$ (with corresponding pressure, initial data, and forcing rescaled as in (\ref{NSrescale})) with $\lambda:= \left(\frac{\varepsilon_0}{2C_{leb}N}\right)^2<\frac{1}{2}$ (for $N$ sufficiently large), where $C_{leb}$ is the constant arising from the continuous embedding $L^6(B_3)\hookrightarrow L^3(B_3),$ and $\varepsilon_0\in(0,1)$ is a small universal constant to be determined later. 
    This rescaling gives
    \begin{align} \norm{\tilde{v}_0}{L^3(B_3)}+\norm{\tilde{F}}{L^6_{t,x}(B_3\times (0,T_0))}&\leq \varepsilon_0,\label{initialsmallnessassump}\\
        \norm{\tilde{v}}{L^\infty_tL^2_x\cap L^2_t\dot{H}^1_x (B_4\times (0,T_0))}^2+\norm{\tilde{p}}{L^\frac{3}{2}_{t,x}(B_4\times (0,T_0))}&\leq 2\lambda^{-\frac{4}{3}}M=:\tilde{M} \label{L3assumptsKMT}
    \end{align}
    where $0<T_0:=\lambda^{-2}T_1\leq c_0\tilde{M}^{-18}$ ($c_0$ is sufficiently small as determined later; $c_1$ is defined by this relation).
    We closely follow the proof and notation of Theorem 3.1 in \cite{KMT21} (with $B_4$ replaced by $B_2$ in the assumptions) with incorporated forcing to obtain
    \begin{align}
        |\nabla^j \tilde{v}(0,t)|\lesssim \frac{1}{t^{-\frac{j+1}{2}}} \quad \forall t\in (0,T_0)\quad \text{and}\quad j=0,1.
    \end{align}
    Undoing the previous rescaling will then yield the claimed result (\ref{localinspacesmoothing}). 
    
    For ease of notation, we will now just write $(v,p,v_0,F)$ instead of $(\tilde{v},\tilde{p},\tilde{v}_0,\tilde{F})$ and $M$ instead of $\tilde{M}.$ For $R>0,$ define $N_R :=  \sup_{R\leq r\leq 1 }\frac{1}{r}\int_{B_r}|v_0|^2\;dx$ and note that \begin{align}
        N_R\leq  \left(\frac{4\pi}{3}\right)^\frac{1}{3}\norm{v_0}{L^3(B_1)}\leq 2\varepsilon_0.
    \end{align}
    For $0<t\leq T_0$ and $0<a<b<1,$ define 
    \begin{align}
        \mathcal{E}_{a,b}(t):= \sup_{a<r\leq b}E_r(t),
    \end{align}
    where 
    \begin{align}
        E_r(t):= \sup_{0<s<t}\frac{1}{r}\int_{B_r}|v(x,s)|^2\;dx + \frac{1}{r}\int_0^t\int_{B_r}|\nabla v|^2\;dxds + \frac{1}{r^2}\int_0^t\int_{B_r}|p|^\frac{3}{2}\;dxds.
    \end{align}
    Let $0<r_1<\frac{1}{3}$ be determined later. Incorporating forcing into the local energy inequality (see \eqref{LocalEnergyIneq}) gives the inequality\footnote{This corresponds to choosing $\delta=10\varepsilon_0$ in line (3.4) of \cite{KMT21}.} $R\leq r\leq r_1:$ 
    \begin{align}
        E_r(t) &\leq 4\varepsilon_0 + \frac{C}{r^3}\int_0^t\int_{B_{2r}}|v|^2 + \frac{C}{r^2}\int_0^t\int_{B_{2r}}|v|^3 + \frac{C}{r^2}\int_0^t\int_{B_{2r}}|p|^\frac{3}{2} +\frac{C}{r}\int_0^t\int_{B_{2r}} F\cdot v \;dxdt
        \\ &\leq 4\varepsilon_0 + I_{lin} + I_{nonlin} +  I_{pr}+I_{F}. \label{KMT3.4}
    \end{align}
    Some terms of (\ref{KMT3.4}) may be treated in exactly the same way as in \cite{KMT21}: for any $R\leq r\leq \frac{r_1}{2}$, we have
    \begin{align}
        \left| I_{lin}\right| &\leq \frac{C}{r^2}\int_0^t\mathcal{E}_{R,r_1}(s)\;ds,
    &\label{nonlinint}
        \left| I_{nonlin}\right| &\leq \varepsilon\mathcal{E}_{R,r_1}(t)+ \frac{C_\varepsilon}{r^2}\int_0^t\mathcal{E}_{R,r_1}^3(s)+\mathcal{E}_{R,r_1}(s)\;ds
    \end{align}
    where $\varepsilon \in (0,1)$ is yet to be determined.  
    With the introduction of forcing, the analysis of the pressure term needs certain adjustment compared to \cite{KMT21}.
    \\ \underline{\textbf{Step 2: Preliminary pressure estimates:}} Similarly to \cite{KMT21}, for $\varrho \in (3r,1]$ we let $\psi\in C^\infty_0(B_{2\rho};[0,1])$ with $\psi=1$ on $B_{\rho}$ and then decompose the pressure $p=\tilde{p}+p_f+p_h$, where\footnote{Where $\mathcal{R}_i$ denotes the Riesz transform in the $i$-th coordinate.} $\tilde{p}=\mathcal{R}_i\mathcal{R}_j(\psi v_iv_j)$ and $p_f:= (\Delta)^{-1}\text{div}(\psi F) - ((\Delta)^{-1}\text{div}(\psi F) )_{B_{\varrho}}$ and $p_h$ is harmonic on $B_{\rho}$.
    Using that $p_h$ is harmonic, we see
    \begin{align}
        \int_{B_{2r}}|p|^\frac{3}{2}\;dx &\leq C\int_{B_{2r}}|\tilde{p}|^\frac{3}{2}\;dx+ C\int_{B_{2r}}|p_h|^\frac{3}{2}\;dx+ C\int_{B_{2r}}|p_f|^\frac{3}{2}\;dx
        \\ &\leq C\int_{B_{\varrho}}|\tilde{p}|^\frac{3}{2}\;dx+ C\left(\frac{r}{\varrho}\right)^3\int_{B_{\varrho}}|p|^\frac{3}{2}\;dx+ C\int_{B_{\varrho}}|p_f|^\frac{3}{2}\;dx.
    \end{align}
    By Poincar\'e's inequality, H\"older's inequality and Calder\'on-Zygmund bounds on the support of $\psi$,  
    \begin{align}
        \int_{B_{\varrho}}|p_f|^\frac{3}{2}\;dx \leq C \varrho^{\frac{3}{2}}\int_{B_{\varrho}}|\nabla (\Delta)^{-1}\text{div}(\psi F)|^\frac{3}{2}\;dx &\leq  C\varrho^{\frac{15}{4}}\left(\int_{B_{2\varrho}}|F|^6\;dx\right)^\frac{1}{4}.
    \end{align}
    By Calder\'on-Zygmund bounds, we also have
    \begin{align}
        \int_{B_{\varrho}}|\tilde{p}|^\frac{3}{2}\;dx \leq C\int_{B_{2\varrho}}|v|^3\;dx.
    \end{align}
    Therefore, (using $r<\varrho$)
    \begin{align}\label{KMTpressureestim}
        \int_{B_{2r}}|p|^\frac{3}{2}\;dx \leq C\int_{B_{2\varrho}}|v|^3\;dx + C'\left(\frac{r}{\varrho}\right)^3\int_{B_\varrho}|p|^\frac{3}{2}\;dx + C\varrho^{\frac{15}{4}}\norm{F}{L^6_x(B_{2\varrho})}^\frac{3}{2}. 
    \end{align}
    Integrating in time, we have
    \begin{align} \label{KMTpressureestim2}
         |I_{pr}| &\leq \frac{C}{r^2}\int_0^t\int_{B_{2\varrho}}|v|^3\;dxds + \frac{C'}{\varrho^2}\frac{r}{\varrho}\int_0^t\int_{B_\varrho}|p|^\frac{3}{2}\;dxds + C\left(\frac{\varrho}{r}\right)^2\varrho^{\frac{7}{4}}\int_0^t\norm{F}{L^6_x(B_{2\varrho})}^\frac{3}{2}\;ds.
    \end{align}
    Let $0\leq r_0\leq \frac{r_1}{3}$ where $r_0$ is yet to be determined. We now estimate $\mathcal{E}_{R,r_1}$ by estimating $\mathcal{E}_{R,r_0}$ and $\mathcal{E}_{r_0,r_1}$ separately as in \cite{KMT21}.
    \\\underline{\textbf{Step 3i: estimate $\mathcal{E}_{r_0,r_1}(t)$}:}
    In this case, for $E_r(t)$, let $R\leq r\leq r_0.$ For $I_{pr}$ and $I_F$, we must adapt \cite{KMT21} due to the forcing term; 
    as in \cite{KMT21}, we set $r_0:= \frac{r_1}{20C'}$ and $\rho = 10C'r$. We substitute this choice into (\ref{KMTpressureestim2}) and use (\ref{nonlinint}) to see  
    \begin{align}
         |I_{pr}| &\leq \varepsilon\mathcal{E}_{R,r_1}(t)+ \frac{C_\varepsilon}{r^2}\int_0^t\mathcal{E}_{R,r_1}^3(s)+\mathcal{E}_{R,r_1}(s)\;ds + \frac{1}{10}\mathcal{E}_{R,r_1}(t) +Cr_1^{\frac{7}{4}}T_0^{\frac{3}{4}}\norm{F}{L^6_{t,x}(B_3\times(0,T_0))}^\frac{3}{2}
         \\ &\leq \varepsilon\mathcal{E}_{R,r_1}(t)+ \frac{C_\varepsilon}{r^2}\int_0^t\mathcal{E}_{R,r_1}^3(s)+\mathcal{E}_{R,r_1}(s)\;ds + \frac{1}{10}\mathcal{E}_{R,r_1}(t) +\frac{\varepsilon_0}{10},
    \end{align} 
    using that $T_0\leq c_0M^{-18}$ with $c_0$ sufficiently small and $r\leq r_0 \leq r_1\leq \frac{1}{3}$.
    Finally, using (\ref{initialsmallnessassump}) and H\"older's inequality, the definition of $\mathcal{E}$ and Young's inequality (and that $r<r_0$, $t<T_0=c_0M^{-18}$ here) gives us  
    \begin{align}
        \left| I_{F}\right| &\leq  Cr_1^\frac{1}{2}\mathcal{E}^\frac{1}{2}_{R,r_1}(t) \norm{F}{L^6(B_{2r}\times(0,T_0))}T_0^\frac{5}{6}\leq \frac{1}{100}\mathcal{E}_{R,r_1}(t) +\frac{\varepsilon_0}{10},
    \end{align}
    provided that $M$ is sufficiently large.
    Substituting all these estimates into (\ref{KMT3.4}), we see
    \begin{equation}
        \mathcal{E}_{R,r_0}(t) \leq \frac{21\varepsilon_0}{5} + \left( 2\varepsilon +\frac{1}{10} +\frac{1}{100} \right) \mathcal{E}_{R,r_1}(t) + \frac{C_\varepsilon}{R^2}\int_0^t\mathcal{E}_{R,r_1}^3(s)+\mathcal{E}_{R,r_1}(s)\;ds.
    \end{equation}
    Choosing $\varepsilon=\frac{1}{40}$ gives
    \begin{equation}
        \mathcal{E}_{R,r_0}(t) \leq \frac{21\varepsilon_0}{5} + \frac{4}{25} \mathcal{E}_{R,r_1}(t) + \frac{C}{R^2}\int_0^t\mathcal{E}_{R,r_1}^3(s)+\mathcal{E}_{R,r_1}(s)\;ds.
    \end{equation}
    \\\underline{\textbf{Step 3ii: estimate $\mathcal{E}_{r_0,r_1}(t)$}:} In this case, for $E_r(t)$, let $r_0\leq r\leq r_1$. Again, some terms may be treated exactly as in \cite{KMT21}, taking $r_1 := c_1M^{-\frac{3}{2}}$ for some sufficiently small constant $c_1>0$ and using that $T_0\leq c_0M^{-18}$ for some sufficiently small constant $c_0>0$, we see that for $M$ sufficiently large: 
    \begin{align}
         \left| I_{lin}\right| &\leq \frac{CT_0M}{r_0^3}\leq \frac{\varepsilon_0}{100} & &\text{and} &
    \label{KMTnonlinMestim}
        \left| I_{nonlin}\right| &\leq \frac{CT_0^\frac{1}{4}M^\frac{3}{2}}{r_0^2}+\frac{CT_0M^\frac{3}{2}}{r_0^2} \leq \frac{\varepsilon_0}{100}.
    \end{align}
    Setting $\varrho=1$ and substituting (\ref{KMTnonlinMestim}) into (\ref{KMTpressureestim2}),using H\"older's inequality and using the definition of $r_0:=\frac{r_1}{20C'}$ and $T_0$ then yields:
    \begin{equation}
        \left| I_{pr}\right| \leq \frac{CT_0^\frac{1}{4}M^\frac{3}{2}}{r_0^2}+\frac{CT_0M^\frac{3}{2}}{r_0^2} + Cr_1M^\frac{3}{2}+ \frac{C}{r_0^2}T_0^{\frac{3}{4}}\norm{F}{L^6(B_3\times(0,T_0))}^\frac{3}{2} \leq \frac{\varepsilon_0}{100}.
    \end{equation}
    We then estimate $|I_F|$ by using H\"older's inequality and (\ref{KMTnonlinMestim}):
    \begin{align}
         \left| I_{F}\right| &\leq \frac{C}{r} \left(\int_0^t\int_{B_{2r}}|v|^2 \;dxds\right)^\frac{1}{2}\norm{F}{L^6(B_{3}\times(0,T_0))}rt^\frac{1}{3} \leq C |I_{lin}|^\frac{1}{2}\varepsilon_0r_1^\frac{3}{2}T_0^\frac{1}{3} \leq \frac{\varepsilon_0}{100}. \label{widerforcedestim}
    \end{align}
    Combining (\ref{KMTnonlinMestim})-(\ref{widerforcedestim}) then gives:  \begin{equation}
        \mathcal{E}_{r_0,r_1}(t) \leq \frac{21\varepsilon_0}{5}.
    \end{equation}
    \underline{\textbf{Step 4: combining cases and concluding:}}
    Combining Steps 3i and 3ii then yields the inequality
    \begin{equation}
        \mathcal{E}_{R,r_1}(t) \leq \frac{21\varepsilon_0}{5} + \frac{4}{25}\mathcal{E}_{R,r_1}(t) + \frac{C}{R^2}\int_0^t\mathcal{E}_{R,r_1}^3(s)+\mathcal{E}_{R,r_1}(s)\;ds.
    \end{equation}
    Thus, rearranging gives 
    \begin{equation}
        \mathcal{E}_{R,r_1}(t) \leq 5\varepsilon_0 + \frac{C}{R^2}\int_0^t\mathcal{E}^3_{R,r_1}(s)+\mathcal{E}_{R,r_1}(s)\;ds \qquad \text{for} \; 0<t\leq T_0.
    \end{equation}
    From here, we may use the same reasoning as \cite{KMT21}; using a continuation-type inequality (see \cite[Lemma 2.4]{KMT21}), one may then deduce that
    \begin{equation}
        \mathcal{E}_{r,r_1}(t) \leq 10\varepsilon_0 \quad \text{ for }\; t\in [0,\min(T_0,\lambda r^2)],\; \text{where}\; r\in [R,r_1]\;\; \text{and} \;\lambda:= \frac{c'}{1+100\varepsilon_0^2}.\label{smallnessofLocEn}
    \end{equation}
    where $c'$ is a sufficiently small universal constant so that $\lambda<1.$
    Note that $r_1:=c_1M^{-\frac{3}{2}}$, so setting $r:=r_2:=\sqrt\frac{t}{\lambda}$ is admissible\footnote{As we set $T_0=c_0M^{-18}$ and $r_1=c_1M^{-\frac{3}{2}}$, we have $ \min\left(r_1, \sqrt{\frac{t}{\lambda}}\right) =  \sqrt{\frac{t}{\lambda}} $ for any $0<t<T_0$, in comparison with \cite{KMT21}.} in (\ref{smallnessofLocEn}) for any $0<t<T_0$. From this, Lebesgue interpolation on $v$ (with the Sobolev embedding theorem), and using $\varepsilon_0\in(0,1)$, we have
    \begin{align}
        \frac{1}{r_2^2} \int_0^{\lambda r_2^2}\int_{B_{r_2}} |v|^3+|p|^\frac{3}{2}\;dxds \lesssim \varepsilon_0 
    \end{align}
   In particular, we see that,
   \begin{align}
       \frac{1}{t} \int_0^{t}\int_{B_{\sqrt{t}}} |v|^3+|p|^\frac{3}{2}\;dxds+ t^{\frac{13}{2}}\int_{Q_{\sqrt{t}}(0,t)}|F|^6\;dxds\lesssim \frac{\varepsilon_0}{\lambda} + \frac{\varepsilon_{0}}{2} \lesssim\varepsilon_0. \label{KMTCKN}
   \end{align}
    Ensuring that $\varepsilon_0$ is small enough, we may apply Corollary \ref{CKN higher order derivatives for forcing} for $t\in (0,T_0)$, (over the parabolic cylinder $(0,t)\times B_{\sqrt{t}}(0)$): 
    \begin{align}
        |\nabla^j v(0,t)|\lesssim \frac{1}{t^{-\frac{j+1}{2}}} \quad \text{for}\; j=0,1.
    \end{align}
    This is as required, in light of the rescaling and translation at the beginning of the proof.
\end{proof}

\subsection{Backwards propagation of vorticity concentration}
\begin{Notation*}
    In this section, for $r>0$, we define 
    \begin{align}
        A_r &= \frac{1}{r}\norm{v}{L^\infty_t L^2_x(Q_r)}^2, & E_r&=\frac{1}{r}\norm{\nabla v}{L^2(Q_r)}^2, & C_r &=\frac{1}{r^2}\norm{v}{L^3(Q_r)}^3,
        \\ D_r&=\frac{1}{r^2}\norm{p}{L^\frac{3}{2}(Q_r)}^\frac{3}{2} ,&\mathcal{E}_r&= A_r+E_r+D_r.
    \end{align}
\end{Notation*}
Compare with Proposition 2 and Proposition 3 of \cite{Localizedblowup}. 
\begin{Proposition}\label{Backwards prop of vort conc}\label{vorticity doesn't cocentrate into quantative regularity}
    If $M$ is sufficiently large, the following holds true. Suppose that $(v,p)$ is a smooth solution to the Navier-Stokes equations with smooth forcing $F$ on $\R^3\times (-4,0)$ such that 
    \begin{align}
        \norm{v}{L^\infty_t L^3_x(\R^3\times (-4,0))} +\norm{F}{L^6_{t,x}(\R^3\times (-4,0))}&\leq M. \label{backpropassump}
    \end{align}
    Suppose that  $0<-t_1<-t_2<M^{-192}$ 
    such that \begin{equation}\label{timeswellsep1}
        -t_1 \leq M^{-200}(-t_2)
    \end{equation} and
    \begin{equation}
    \int\limits_{B(0,M^{96}(-t_2)^\frac{1}{2})}  |\omega(x,t_2)|^2 \;dx \leq \frac{M^{-88}}{(-t_2)^\frac{1}{2}}. \label{vorticityconc1}
    \end{equation}
    Then the above assumptions imply that
    \begin{equation} \label{vortconcinitial}
    \int\limits_{B(0,M^{96}(-t_1)^\frac{1}{2})}  |\omega(x,t_1)|^2 \;dx \leq \frac{M^{-88}}{(-t_1)^\frac{1}{2}},
    \end{equation}
    and for $j=0,1,$
    \begin{align}
        |\nabla^jv(x,s)| \lesssim  (-t_2)^{-\frac{j+1}{2}} \qquad \text{for} \qquad (x,s) \in B(M^{92}(-t_2)^\frac{1}{2}) \times \left(\frac{t_2}{2}, 0\right]. \label{smoothingunderprop}
    \end{align}
\end{Proposition}
\begin{proof}
   The proof closely follows \cite{Localizedblowup}, incorporating forcing by utilising Proposition \ref{KMTwithforce}. As in \cite[Lemma 2]{Localizedblowup}, we first show that \begin{align} \sup_{0<r\leq1}\mathcal{E}_r\lesssim M^3. \end{align}
    By \eqref{backpropassump} and the local energy inequality \eqref{LocalEnergyIneq}, it's enough to show that
    \begin{align} \sup_{0<r\leq2}D_r  \lesssim M^3. \label{locpressreq}\end{align}
    Indeed, by \cite[Lemma 3.1]{Seregin01}, with $\gamma=-\frac{1}{2}$, we have that for $0<\theta<1$ and $0<r\leq 2$,
    \begin{align}
        D_{\theta r}\lesssim\theta D_r+\theta^{-2}(C_r+\norm{F}{L^2(Q_2)}^\frac{3}{2})\lesssim \theta D_r+\theta^{-2}M^3. \label{Press Iterator}
    \end{align}
    By decomposing the pressure  $p_{v\otimes v} := \mathcal{R}_i\mathcal{R}_j(v_iv_j)$, $p_F := (\Delta)^{-1}\operatorname{div}F$, we may estimate $D_2$ by using \eqref{backpropassump}, H\"older, Sobolev and Calderón-Zygmund inequalities: \begin{align}
        D_2\lesssim \norm{p_{v\otimes v}}{L^\frac{3}{2}(Q_2)}^\frac{3}{2}+\norm{ p_F}{L^6(Q_2)}^\frac{3}{2}\lesssim \norm{p_{v\otimes v}}{L^\frac{3}{2}(\R^3\times(-4,0))}^\frac{3}{2}+\norm{ \nabla p_F}{L^2(\R^3\times(-4,0))}^\frac{3}{2}\lesssim  M^3.
    \end{align}
   Then a standard iteration of (\ref{Press Iterator}) gives line (\ref{locpressreq}) as required. Just as in \cite[Lemma 2]{Localizedblowup}, this implies that 
   \begin{equation}
        \sup_{0<r\leq 1}\left[ \frac{1}{r}\norm{v}{L^\infty_tL^2_x\cap L^2_t\dot{H}^1_x(Q_r)}^2 + \frac{1}{r^\frac{4}{3}}\norm{p}{L^\frac{3}{2}(Q_r)}  \right] +  \norm{F}{L^6(Q_1)} \lesssim  M^3. \label{morreyassumpt1}
    \end{equation}
   \\ We now show that line (\ref{morreyassumpt1}) along with our assumptions on the vorticity allow us to apply the local-in-space smoothing Proposition \ref{KMTwithforce}. 
    Observe that (\ref{morreyassumpt1}) implies that, for $M$ large enough ($``r=4M^{92}(-t_2)^\frac{1}{2} \leq 4M^{92}M^{-96} \leq 1"$)
    \begin{equation}
         \frac{1}{4M^{92}(-t_2)^\frac{1}{2}}\norm{v}{L^\infty_tL^2_x\cap L^2_t\dot{H}^1_x(Q(0,4M^{92}(-t_2)^\frac{1}{2})}^2 + \frac{1}{\left(4M^{92}(-t_2)^\frac{1}{2}\right)^\frac{4}{3}}\norm{p}{L^\frac{3}{2}(Q(0,4M^{92}(-t_2)^\frac{1}{2}))}  \lesssim M^3 .\label{morreyassumpt1a}
    \end{equation}
    Now let $\lambda:= M^{92}(-t_2)^\frac{1}{2} \leq \frac{1}{4}.$ 
    We consider the rescaled and translated functions which are a suitable weak solution to the forced Navier stokes equations on $B_{\lambda^{-1}}\times (0,M^{-184})$
    \begin{align}
        (v_\lambda(x,t),  p_\lambda(x,t), F_\lambda(x,t)) &:= (\lambda v(\lambda x,\lambda^2t+t_2),  \lambda^2p(\lambda x,\lambda^2t+t_2), \lambda^3F(\lambda x,\lambda^2t+t_2)).
    \end{align}
    Observe that $B_{\lambda^{-1}}\times (0,M^{-184})\supseteq B_4 \times (0,M^{-184}).$
    Then for $M$ sufficiently large, (\ref{morreyassumpt1a}) becomes 
    \begin{align}
        \norm{v_\lambda}{L^\infty_tL^2_x\cap L^2_t\dot{H}^1_x(B_4\times (0,M^{-184}))}^2 + \norm{p_\lambda}{L^\frac{3}{2}(B_4\times (0,M^{-184}))}  \lesssim M^3, \label{morreyassumpt1b}
    \end{align}
    and (\ref{vorticityconc1}) becomes
    \begin{align}
        \int\limits_{B(0,\frac{7}{2})}|\omega_\lambda(y,0)|^2\;dy \leq M^{92}(-t_2)^\frac{1}{2}\int\limits_{B(0,M^{96}(-t_2)^\frac{1}{2})}  |\omega(x,t_2)|^2 \;dx  \leq M^{4}.\label{recaledvortconineq1}
    \end{align}
    We also see that by (\ref{backpropassump}) and $0<-t_2<M^{-192}$
    \begin{align}
        \norm{F_\lambda}{L^6(B_4\times (0,M^{-184}))} \leq \left(M^{92}(-t_2)^\frac{1}{2}\right)^\frac{13}{6}\norm{F}{L^6(Q_{1})}
         \leq M^{2}.    
    \end{align}
    From here, the proof proceeds as in \cite[Lemma 2]{Localizedblowup}. We utilise the the Sobolev embedding $H^1\hookrightarrow L^6$, the Biot-Savart law ($-\Delta v_{0,\lambda} = \nabla\times\omega_{0,\lambda}$), elliptic regularity theory and \eqref{morreyassumpt1b}-(\ref{recaledvortconineq1}) to find:
    \begin{align}
        \norm{v_{0,\lambda}}{L^6(B_3)}&\lesssim \norm{\nabla v_{0,\lambda}}{L^2(B_3)}+\norm{v_{0,\lambda}}{L^2(B_3)} \lesssim \norm{\omega_{0,\lambda}}{L^2(B_\frac{7}{2})}+\norm{v_{0,\lambda}}{L^2(B_\frac{7}{2})}  
        \lesssim M^2 .
    \end{align}
    We may then apply Proposition \ref{KMTwithforce} (with $M$ replaced with $M^4$ and $N$ replaced with $C_{univ}M^2$) to find that for $M$ sufficiently large,
    \begin{equation}
       |\nabla^jv_\lambda(x,t)| \lesssim t^{-\frac{j+1}{2}} \qquad \text{for} \; (x,t)\in B_1\times (0,\,M^{-184}) , \quad j =0,1.  
    \end{equation}
    Hence, undoing the previous rescaling gives:
    \begin{equation}
       |\nabla^jv(x,t)| \lesssim (t-t_2)^{-\frac{j+1}{2}} \qquad \text{for} \; (x,t)\in B(0,M^{92}(-t_2)^\frac{1}{2})\times (t_2,\,0] , \quad j =0,1.  
    \end{equation}
    In particular,
    \begin{equation}
       |\nabla^jv(x,t)| \lesssim (-t_2)^{-\frac{j+1}{2}} \qquad \text{for} \; (x,t)\in B(0,M^{92}(-t_2)^\frac{1}{2})\times \left(\frac{t_2}{2},\,0\right] , \quad j =0,1.  
    \end{equation}
    Now, by (\ref{timeswellsep1}), observe that 
    \begin{align}
        B(0,M^{96}(-t_1)^\frac{1}{2})\times\lbrace t_1\rbrace \subset B(0,M^{92}(-t_2)^\frac{1}{2})\times \left(\frac{t_2}{2},\,0\right],
    \end{align}
    hence, 
    \begin{equation} 
    \int\limits_{B(0,M^{96}(-t_1)^\frac{1}{2})}  |\omega(x,t_1)|^2 \;dx \lesssim \frac{\left(M^{96}(-t_1)^\frac{1}{2}\right)^3}{(-t_2)^2} = \frac{M^{288}}{(-t_1)^\frac{1}{2}}\left(\frac{-t_1}{-t_2}\right)^2 \leq \frac{M^{-112}}{(-t_1)^\frac{1}{2}} \leq \frac{M^{-88}}{(-t_1)^\frac{1}{2}} .
    \end{equation}
\end{proof}

\subsection{Epoch of Regularity}
\begin{Proposition}[Epoch of regularity] \label{Epoch of regularity}
     Let $0<T_1\leq 1$ and define $I := (-T_1,-\frac{T_1}{2})$. Suppose that $(v,p)$ is a smooth Leray-Hopf weak solution to the Navier-Stokes equations with smooth forcing $F$ on $\R^3\times (-T_1,0)$ such that 
\begin{align}
\norm{v}{L^\infty_t L^3_x(\R^3\times (-T_1,0))} &\leq M,
    &
    \norm{F}{L^2_{t,x}\cap L^6_{t,x}(\R^3\times (-T_1,0))} &\leq \frac{M^{-1}}{T_1^\frac{1}{4}} . \label{epochassumpt}
\end{align} Then there exists a positive universal constant $M_0\geq 1$ such that for all $M\geq M_0$, there exists an ``epoch of regularity", a closed interval $I' \subset I$ such that
\begin{equation}
    |I'| = M^{-28}|I|,
\end{equation}
and for $k=0,1,$
\begin{equation}
    \norm{\nabla^{(k)}v}{L^\infty(\R^3\times I')} \lesssim \frac{1}{M}|I'|^{-\frac{k+1}{2}}.
\end{equation}
Moreover, for $R\geq T_1^\frac{1}{2}$, we have that 
\begin{align}          
    \sup_{x_0\in\R^3}\int_{I'}\int_{B_R(x_0)} \frac{|\omega|^2}{T_1}+ |\nabla\omega|^2\;dxds \lesssim M^{19}RT_1^{-1} .\label{EpochCacciop}
\end{align}

\end{Proposition}
\begin{Remark}\label{Epoch Remark}
   Compare with \cite[Proposition 5]{Localizedblowup}. Due to the regularity of the forcing in the quantitative assumption \eqref{epochassumpt}, one would be unable to obtain a bound on $\norm{\nabla\omega}{L^\infty(\R^3\times I')}$ as in \cite[Proposition 5]{Localizedblowup}. However, for the purposes of estimating terms appearing in the application of quantitative Carleman inequalities in Section \ref{MainProof}, we find that the bound \eqref{EpochCacciop} is sufficient. 
\end{Remark}
\begin{proof}
    The proof is directly analogous to \cite[Proposition 5]{Localizedblowup}, but we prove lower regularity estimates on the gradient of the vorticity \eqref{EpochCacciop}. We provide the estimates required to obtain this, referencing the original proof where details are similar. 
    \\By a rescaling argument, replacing $v$ with $v_\lambda(x,t):=\lambda v(\lambda x,\lambda^2 t)$ with $\lambda:= T_1^\frac{1}{2},$ (with corresponding pressure and forcing rescaled as in (\ref{NSrescale})) we may instead assume that $I=(-1,-\frac{1}{2})$. Just as in \cite[Proposition 5]{Localizedblowup}, since $v$ is a smooth suitable weak Leray-Hopf solution on $\R^3\times [-1,0)$ , we can decompose $v(x,t) = L_F(x,t) + W(x,t)$ where 
    \begin{align*}
        L_F(x,t):= e^{(t+1)\Delta}v(x,-1) + \int_{-1}^te^{(t-s)\Delta}\LP F(x,s)\;ds,\qquad
    W(x,t):= -\int_{-1}^te^{(t-s)\Delta}\LP \nabla\cdot(v\otimes v )(x,s)\;ds.
\end{align*}
As in \cite[Proposition 5]{Localizedblowup}, using standard heat semigroup estimates gives that for $3\leq q \leq \infty$ and $t\in[-\frac{1}{2},0)$:
\begin{align}
    \norm{L_F(\cdot,t)}{L^q(\R^3)} &\lesssim_{q} \frac{\norm{v(\cdot,-1)}{L^3(\R^3)}}{(t+1)^{\frac{1}{2}\left(\frac{1}{3}-\frac{1}{q}\right)}} + \int_{-1}^t\frac{1}{(t-s)^{\frac{1}{2}\left(\frac{1}{3}-\frac{1}{q}\right)}}\norm{F(\cdot,s)}{L^3(\R^3)}\;ds .
\end{align}
Then by \eqref{epochassumpt}, H\"older's inequality, and Lebesgue interpolation ($F\in L^2\cap L^6 \subset L^3$)
\begin{align} 
     \norm{L_F}{L^\infty_tL^q_x(\R^3\times[-\frac{1}{2},0))}&\lesssim_{q} M + \norm{\frac{1}{s^{\frac{1}{2}\left(\frac{1}{3}-\frac{1}{q}\right)}}}{L_s^\frac{3}{2}((0,1))}\norm{F}{L^3(\R^3\times [-1,0))} \lesssim_{q} M + \norm{F}{L^3(\R^3\times [-1,0))}  \lesssim_q M. \label{L_F key estim}
\end{align}
Next, we estimate $W$ in the energy space. Just as in \cite[Proposition 5]{Localizedblowup}, utilising \eqref{epochassumpt}, Young's convolution inequality and the Oseen kernel estimate\footnote{It is known (for example, from \cite[pp.234-235]{SolonnikovRussian}) that the operator $-e^{t\Delta}\LP\nabla\cdot$ is a convolution operator with the Oseen tensor $K$ as its kernel and satisfies the bound
\begin{equation}
\left\vert K(x,t)\right\vert \lesssim \frac{1}{\left( |x|^2+t\right)^{2}} \quad \text{for all} \;\; (x,t)\in \R^3\times (0,\infty).
\end{equation}
}, we have that:
\begin{align}
    \sup_{t\in[-1,0)}\norm{W(\cdot,t)}{L^2(\R^3)} \lesssim\int_{-2}^0 \frac{\norm{v(\cdot,s)}{L^3(\R^3)}^2}{(t-s)^\frac{3}{4}} \;ds \lesssim M^2.\label{Westimepoch2}
\end{align}

$W$ is smooth and divergence-free on $\R^3\times[-1,0)$ and satisfies the equation
\begin{equation}
    \partial_tW-\Delta W+W\cdot\nabla W +W\cdot\nabla L_F +L_F\cdot\nabla W +L_F\cdot\nabla L_F + \nabla\Pi = 0
\end{equation}
and $W(\cdot,-1)=0.$
Since $v$ is a smooth suitable weak Leray-Hopf solution, performing the standard energy estimate as in \cite[Proposition 5 ]{Localizedblowup}, using (\ref{L_F key estim}) with $q=\infty$ and $q=4$ and \eqref{Westimepoch2} gives that for $t\in [-\frac{1}{2},0)$: 
\begin{align}
     &\norm{W(\cdot,t)}{L^2(\R^3)}^2 + \norm{\nabla W}{L^2(\R^3\times (-\frac{1}{2},0))}^2 
     \lesssim M^6. \label{W energy estim epoch}
\end{align}
Now, let $x\in \R^3$, $r:=M^{-\frac{13}{2}}$, $t\in [a+\frac{3}{2}M^{-13}, a+2M^{-13}].$
Then observe that for $M$ sufficiently large, $(t-r^2,t) \subset (a,a+C_{univ}M^{-12})$. By exactly the same arguments as in \cite[Proposition 5 (138)-(149)]{Localizedblowup}, we have that  
\begin{equation}
    \frac{1}{r^2}\int_{t-r^2}^t\int_{B_r(x)} |v|^3+|p-(p)_{B_r(x)}|^\frac{3}{2} \;dyds \lesssim M^{-\frac{3}{4}}.
\end{equation}
Then, given our assumption on the forcing \eqref{epochassumpt}, it follows from Corollary \ref{CKN higher order derivatives for forcing} that for $k=0,1:$ 
\begin{equation}
    |\nabla^{(k)}v(x,t)| \lesssim r^{-(k+1)} =M^{\frac{13}{2}(k+1)} \label{nabla velo epoch}
\end{equation}
for $x\in \R^3$ and $t\in [a+\frac{3}{2}M^{-13}, a+2M^{-13}] =: [\alpha,\beta] =: J.$
\\Next, we prove the lower regularity estimates on the gradient of the vorticity \eqref{EpochCacciop}. Let $R\geq 1.$
By using \eqref{L_F key estim}, the parabolic Caccioppoli inequality (Proposition \ref{CaccioppoliBall}), Young's inequality and H\"older's inequality, we have
\begin{align}
    \int_{-\frac{1}{2}}^0\int_{B_{2R}}|\nabla L_F|^2 \;dxd\tau &\lesssim \left(1+\frac{1}{R^2}\right)\int_{-1}^0\int_{B_{3R}}|L_F|^2 \;dxd\tau + \norm{\mathbb{P}F}{L^2_{t,x}(B_{3R}\times[-1,0)])}\norm{L_F}{L^2_{t,x}(B_{3R}\times[-1,0)])}
    \\ &\lesssim \left(1+\frac{1}{R^2}\right)\int_{-1}^0\left(\int_{B_{3R}}|L_F|^3 \;dx\right)^\frac{2}{3}d\tau + \norm{\mathbb{P}F}{L^2_{t,x}(B_{3R}\times[-1,0)])}^2
    \\ &\lesssim \norm{L_F}{L^\infty_t L^3_x(\R^3\times[-1,0)])}^2\left(R+\frac{1}{R}\right)+ M^{-2}
    \lesssim RM^2. \label{nabla L_F key bound}
\end{align}
Using estimates (\ref{nabla L_F key bound}) and (\ref{W energy estim epoch}) gives
\begin{align}
    \norm{\omega}{L^2(B_{2R}\times J)}^2 &\lesssim \norm{\nabla L_F}{L^2(B_{2R}\times (-\frac{1}{2},0) )}^2  + \norm{\nabla W}{L^2( B_{2R}\times(-\frac{1}{2},0))}^2 \lesssim RM^2 + M^6. \label{nabla vorticity epoch}
\end{align}
Denote $J' = [a+\frac{7}{4}M^{-13} ,a+2M^{-13}]$. We use the parabolic Caccioppoli inequality (Proposition \ref{CaccioppoliBall}) to estimate the $L^2$ norm of $\nabla\omega$ here:
\begin{align}
    \int_{J'}\int_{B_{R}}|\nabla\omega|^2\;dxds &\lesssim \left(M^{13}+\frac{1}{R^2}\right)\int_{J}\int_{B_{2R}}|\omega|^2\;dxds + \norm{|v||\omega|}{L^2(B_{2R}\times J')}^2 +\norm{\omega\otimes v + F}{L^2(B_{2R}\times J')}^2
    \\ &\lesssim \left(M^{13}+\frac{1}{R^2}+ \norm{v}{L^\infty(B_{2R}\times J')}^2\right) \int_{J}\int_{B_{2R}} |\omega|^2\;dxds +\norm{F}{L^2(B_{2R}\times J')}^2
    \\ &\lesssim \left(M^{13}+\frac{1}{R^2}\right)\left(RM^2+M^6\right)+ 1 \lesssim RM^{19}.
\end{align}
In the last line, we used Young's inequality and $\norm{F}{L^2}\lesssim 1$. Finally, we note that we may take $I'\subset J'$ such that $|I'| = \frac{1}{2}M^{-28} = |I|M^{-28}.$ 
Then \eqref{nabla velo epoch} gives the following estimate for $k=0,1$:
\begin{equation}
    \norm{\nabla^{(k)}v}{L^\infty(\R^3\times I')} \leq M^{13} \lesssim \frac{1}{M}|I'|^{-\frac{k+1}{2}} .
\end{equation}
\end{proof}

\subsection{Annulus of Regularity}\label{Ann of Reg Section}
This is based on \cite[Proposition 7]{Localizedblowup} (also see \cite[Proposition 19]{Spatialconc}). Here, as well as incorporating forcing, we also provide the additional Caccioppoli-type estimate \eqref{smalltimeestim1} for $\omega$, which we will use to estimate certain terms of the quantitative Carleman inequalities applied in Section \ref{MainProof}.  
\begin{Proposition}[General version with forcing]  \label{ann of reg M version}
Let $0<T\leq 1$ and $R\geq T^\frac{1}{2}$, $\mu>0$ and suppose that $(v,p)$ is a suitable weak solution to the Navier-Stokes equations with forcing $F\in L^2_{t,x}\cap L^6_{t,x}(\R^3 \times [-T,0))$ on $\R^3 \times [-T,0) .$
There exists $M_0(\mu)>1$ such that for all $M\geq M_0(\mu)$ the following statement holds:
If
\begin{align}
    \norm{v}{L^\infty_t L^3_x(\R^3\times[-T,0)])}\leq M \quad\text{and}\quad\norm{F}{L^2_{t,x}\cap L^6_{t,x}(\R^3\times [-T,0))} \leq M^{-2\mu},
    \end{align}
then there exists $\hat{R}=\hat{R}(M,\mu,R) \in [2R,2R\exp{\left(C_{univ}\mu M^{3\mu+6}\right)}]$ such that if $A := \left\lbrace \hat{R} < |x| <\frac{M^{3\mu}\hat{R}}{4} \right\rbrace$, then for $k=0,1:$
\begin{align}
        \norm{\nabla^kv}{L^\infty\left(A\times (-\frac{T}{16},0)\right)} &\lesssim_k M^{-\frac{\mu}{2}}T^{-\frac{k+1}{2}}, &
    \norm{\omega}{C^{0,\frac{1}{6}}_\text{par}\left(A\times (-\frac{T}{16},0)\right)} &\lesssim M^{-\frac{\mu}{2}}T^{-\frac{13}{12}}.    
\end{align} 
Moreover, let $0<t<\frac{T}{16}$ and let $r\geq T^\frac{1}{2}$ be such $A_{(2)}:=B_r\setminus B_\frac{r}{2}\subset A_{(1)}:=B_{2r}\setminus B_\frac{r}{4}\subseteq A$. Then
 \begin{align}
        \int_{-t}^0\int_{A_{(2)}} \frac{|\omega|^2}{T}+|\nabla\omega|^2\;dxds &\lesssim M^{-\mu}\left(\frac{t}{T}\right)^\frac{1}{12}\left(\frac{r^2}{T}\right)^\frac{3}{2}T^{-\frac{1}{2}}+ T^{\frac{13}{6}}t^\frac{2}{3}r^2 \norm{F}{L^6(\R^3\times(-T,0))}^2. \label{smalltimeestim1}
    \end{align}
\end{Proposition}
\begin{proof}
    Up to the Cacciopoli-type estimate \eqref{caciiobegin}, the proof is almost exactly the same as \cite[Proposition 7]{Localizedblowup}. By a rescaling argument, replacing $v$ with $\tilde{v}(x,t):=T^\frac{1}{2} v(T^\frac{1}{2} x,Tt)$ (with corresponding pressure, initial data, and forcing rescaled as in (\ref{NSrescale})), we may instead assume that $T=1$. Let us fix $R \geq 1$, $\mu>0$ and $M\geq M_0(\mu)>1$ where $M_0(\mu)$ is yet to be determined. First, let us decompose the pressure $p=p_{v\otimes v}+p_f$ where
    \begin{equation}
        p_{v\otimes v} := \mathcal{R}_i\mathcal{R}_j(v_iv_j), \qquad \qquad p_F := (\Delta)^{-1}\text{div}F.
    \end{equation}
    Similarly to \cite[Proposition 7]{Localizedblowup}, set $R_k := M^{3\mu k}R. $ Observe that by Calderón-Zygmund estimates, 
    \begin{align}    \sum\limits_{k=0}^{\infty}\int\limits_{-1}^0\int\limits_{R_k<|x|<R_{k+1}}|v|^3 + |p_{v\otimes v}|^\frac{3}{2} \;dxdt &\leq \norm{v}{L^\infty_tL^3(\R^3\times(-1,0))}^3+\norm{p_{v\otimes v}}{L^\infty_tL^\frac{3}{2}(\R^3\times(-1,0))}^\frac{3}{2}\lesssim 
    M^3.
    \end{align}
    By the pigeonhole principle, there exists $k_0 \in \left\lbrace 0,1,2,\cdots, \floor{C_{univ}M^{3\mu+3}} \right\rbrace$ such that
    \begin{equation}    \int\limits_{-1}^0\int\limits_{R_{k_0}<|x|<R_{k_0+1}}|v|^3 + |p_{v\otimes v}|^\frac{3}{2} \;dxdt \leq M^{-3\mu}.\label{annularCKNint}
    \end{equation}
    We observe that $R\leq R_{k_0}\leq R\exp{\left(C_{univ}\mu M^{3\mu+6} \right)}$. Next, define $\tilde{A}:= \left\lbrace R_{k_0}+1<|x|<M^{3\mu} R_{k_0}-1 \right\rbrace$, and let $M_0(\mu)\geq 4^\frac{1}{3\mu}$ (so that $\tilde{A}$ is non-empty). Then by \eqref{annularCKNint}, 
    \begin{equation}
        \sup_{x_0\in \tilde{A}}\int_{-1}^0\int_{B_{1}(x_0)} |v|^3+|p_{v\otimes v}-(p_{v\otimes v})_{B_1(x_0)}|^\frac{3}{2} \;dxdt \leq C_{univ}M^{-3\mu} \leq \frac{\varepsilon_{*}}{3},
    \end{equation}
    where $M\geq M_0(\mu)\geq \max\left(4^\frac{1}{3\mu},\left(\frac{\varepsilon_{*}}{3C_{univ}}\right)^{-\frac{1}{3\mu}}\right)$ and $\varepsilon_*$ is from Corollary \ref{CKN higher order derivatives for forcing}. 
    \\Additionally, using H\"older's inequality, Poincar\'e's inequality and Calderón-Zygmund estimates, 
    \begin{align}
        \sup_{x_0\in \tilde{A}}\int_{-1}^0\int_{B_{1}(x_0)} |p_{F}-(p_{F})_{B_1(x_0)}|^\frac{3}{2} \;dxdt&\lesssim \sup_{x_0\in \tilde{A}}\left(\int_{-1}^0\int_{B_{1}(x_0)} |p_{F}-(p_{F})_{B_1(x_0)}|^2 \;dxdt\right)^\frac{3}{4}  \\&\lesssim\sup_{x_0\in \tilde{A}}\left(\int_{-1}^0\norm{\nabla p_F(\cdot,t)}{L^2(B_1(x_0))}^2 \;dxdt\right)^\frac{3}{4}
        \\&\lesssim \norm{F}{L^2(\R^3\times(-1,0))}^\frac{3}{2}\lesssim M^{-3\mu}\leq \frac{\varepsilon_*}{3}.
    \end{align}
    We also have $\int_{-1}^0\int_{\R^3} |F|^6\;dxdt \leq  M^{-12\mu}\leq \frac{\varepsilon_*}{3}. $ 
    Thus, we may apply Corollary \ref{CKN higher order derivatives for forcing} to find that for all $x_0\in \tilde{A},$
    we have for $k=0,1$ 
    \begin{align}
        \norm{\nabla^kv}{L^\infty\left(Q_\frac{1}{4}(x_0,0)\right)} \lesssim M^{-\frac{\mu}{2}}, \quad
    \text{and}  \quad\norm{\omega}{C^{0,\frac{1}{6}}_\text{par}\left(Q_\frac{1}{4}(x_0,0)\right)} &\lesssim M^{-\frac{\mu}{2}}.   \label{annbounds1} 
    \end{align} 
    Consider $\hat{R}:=2R_{k_0}$. Observe that, as $R_{k_0}\geq R \geq 1$ and $M>4^\frac{1}{3\mu}$:
    \begin{align}
        M^{3\mu} R_{k_0} - 1> \frac{M^{3\mu}}{2}R_{k_0} > 2R_{k_0} > R_{k_0}+1.
    \end{align}
    Then we see that \begin{equation}
    A:=\left\lbrace \hat{R} <|x|<\frac{M^{3\mu}}{4}\hat{R}\right\rbrace \subset \tilde{A}.
\end{equation}
Finally, we show the bound on the gradient of the vorticity. 
We apply the parabolic Caccioppoli inequality on the annuli $A_{(2)}:=B_r\setminus B_\frac{r}{2}\subset A_{(1)}:=B_{2r}\setminus B_\frac{r}{4}\subseteq A\subset \tilde{A}$ where $r\geq 1$. Indeed, applying the statement contained in Remark \ref{CaccioppoliAnnulus} and using \eqref{annbounds1} shows that for $0\leq t\leq\frac{1}{16}$ :
    \begin{align}
        \int_{-t}^0\int_{A_{(2)}}|\omega|^2+|\nabla\omega|^2\;dxds&\lesssim\norm{\nabla v}{L^\infty(A_{(1)}\times (-t,0))}\norm{\omega}{C^{0,\frac{1}{6}}_\text{par}(A_{(1)}\times (-t,0))}t^{\frac{1}{12}}r^3+\left(1+\frac{1}{r^2}\right)\int_{-t}^0\int_{A_{(1)}} |\omega|^2\;dxds\label{caciiobegin}\\&\quad+ \norm{\omega\otimes v}{L^2(A_{(1)}\times(-t,0))}^2 + \norm{F}{L^2(A_{(1)}\times(-t,0))}^2  \nonumber
        \\ &\lesssim \norm{\nabla v}{L^\infty(A_{(1)}\times (-\frac{1}{16},0))}\norm{\omega}{C^{0,\frac{1}{6}}_\text{par}(A_{(1)}\times (-\frac{1}{16},0))}t^{\frac{1}{12}}r^3+\norm{F}{L^2(A_{(1)}\times (-1,0))}^2\\ &\qquad +\left(1+\frac{1}{r^2}+\norm{v}{L^\infty(A_{(1)}\times (-\frac{1}{16},0))}^2\right)\int_{-t}^0\int_{A_{(1)}} |\omega|^2\;dxds  \nonumber 
        \\ &\lesssim M^{-\mu}t^\frac{1}{12}r^3+t^\frac{2}{3}r^2 \norm{F}{L^6(A_{(1)}\times (-1,0))}^2 + (1+\frac{1}{r^2}+M^{-\mu})tr^3M^{-\mu}
        \\ &\lesssim M^{-\mu}t^\frac{1}{12}r^3+t^\frac{2}{3}r^2 \norm{F}{L^6(A_{(1)}\times (-1,0))}^2.
        \end{align}
\end{proof}

\subsection{Carleman inequalities with forcing} \label{Carleman Section}
Below is the ``first Carleman inequality" with incorporated forcing, which corresponds to backwards uniqueness in the unforced setting (cf. \cite[Proposition 4.2]{Tao21}):
\begin{Proposition}\label{FirstCarlemanInequality}
    Let $T>0$, $C_{Carl} \geq 1$, $0<r<R$ and denote the space-time annulus:
    $$\mathcal{A}:= \left\lbrace (x,t)\in \R^3\times[0,T] : r\leq |x|\leq R \right\rbrace.$$
    Suppose that $w:\mathcal{A}\to\R^3$ is smooth and satisfies the differential inequality
    \begin{equation}\label{diffineq1}
        |\partial_tw + \Delta w| \leq \frac{|w|}{C_{Carl}T} + \frac{|\nabla w|}{(C_{Carl}T)^\frac{1}{2}} + |G|.
    \end{equation}
    and $r^2\geq 4C_{Carl}T.$
    Then we have:
    \begin{equation}
        Z \lesssim C_{Carl}^3e^{-\frac{rR}{4C_{Carl}T}}\left( X + e^\frac{2R^2}{C_{Carl}T}Y + Te^\frac{3R^2}{2C_{Carl}T}\iint_\mathcal{A}|G|^2\;dxdt\right),
    \end{equation}
    where 
    \begin{align}
        X &:= \iint_\mathcal{A} e^\frac{2|x|^2}{C_{Carl}T}\left( \frac{|w|^2}{T} + |\nabla w|^2 \right) \;dxdt ,
        & Y &:= \int\limits_{r\leq |x|\leq R} |w(x,0)|^2\;dx,
        \\Z &:= \int_0^\frac{T}{4}\int\limits_{10r\leq |x| \leq \frac{R}{2}} \frac{|w|^2}{T} + |\nabla w|^2 \;dxdt.
    \end{align}
\end{Proposition}
\begin{proof}
    We closely follow Tao's Proposition 4.2 \cite{Tao21}. We provide sufficient details so that the reader can verify the estimates of the additional term due to the forcing $G$. As in \cite{Tao21}, we find that there exists a time $T_0\in[\frac{T}{2},T)$ such that
    \begin{align}
        \int_{r\leq|x|\leq R} e^\frac{2|x|^2}{C_{Carl}T}\left(\frac{|w(x,T_0)|^2}{T}+ |\nabla w(x,T_0)|^2\right)\;dx \lesssim \frac{X}{T}.
    \end{align}
    As in \cite{Tao21}, define \begin{align}
        \alpha&:=\frac{R}{2C_{Carl}T^2}, & g(x,t)&:= \alpha(T_0-t)|x|+\frac{|x|^2}{C_{Carl}T}, & F&:= \partial_tg-\Delta g - |\nabla g|^2.
    \end{align}
    Let $\psi\in C^\infty_c(\left\lbrace r<|x|<R
    \right\rbrace; [0,1])$ be such that $\psi = 1$ on $\left\lbrace 2r<|x|<\frac{R}{2}
    \right\rbrace$
    and $|\nabla^j\psi(x)| \lesssim |x|^{-j}$ for $j=0,1,2$ and $r<|x|<R$.
    By the same arguments as \cite{Tao21}, we get
    \begin{multline}\label{Taopg30}
        \int_0^{T_0}\int_{2r<|x|<\frac{R}{2}}\left(\frac{28|w|^2}{C_{Carl}^2T^2}+\frac{4|\nabla w|^2}{C_{Carl}T}\right)e^g\;dxdt\leq \frac{1}{2}\int_0^{T_0}\int_{\R^3}|L(\psi w)|^2e^g\;dxdt + \int_{\R^3}|\nabla(\psi w)(x,T_0)|^2e^{g(x,T_0)}\;dx\\+\frac{1}{2}\int_{\R^3}|F(x,0)||(\psi w)(x,0)|^2e^{g(x,0)}\;dx.
    \end{multline}
    Now, for $2r<|x|<\frac{R}{2}$, we use the differential inequality (\ref{diffineq1}) to see that:
    \begin{align}
        |L(\psi w)|^2= |Lu|^2 \leq 4\left(\frac{|w|^2}{C_{Carl}^2T^2} + \frac{|\nabla w|^2}{C_{Carl}T}\right)+2|G|^2.
    \end{align}
    Therefore,
    \begin{multline}
        \frac{1}{2}\int_0^{T_0}\int_{\R^3}|L(\psi w)|^2e^g\;dxdt \leq \frac{1}{2}\int_0^{T_0}\int_{|x|\in [r,2r]\cup[\frac{R}{2},R]}|L(\psi w)|^2e^g\;dxdt \\+\int_0^{T_0}\int_{2r<|x|<\frac{R}{2}}\left(\frac{2|w|^2}{C_{Carl}^2T^2} + \frac{2|\nabla w|^2}{C_{Carl}T}\right) e^g\;dxdt + \int_0^{T_0}\int_{2r<|x|<\frac{R}{2}} |G|^2e^g\;dxdt.\label{taolineplusg}
    \end{multline}
    Following \cite{Tao21}, we may absorb terms on the right side of \eqref{taolineplusg} into the left hand side of \eqref{Taopg30} to get:
    \begin{align}
        \int_0^{T_0}\int_{2r<|x|<\frac{R}{2}}\left(\frac{|w|^2}{C_{Carl}^2T^2}+\frac{|\nabla w|^2}{C_{Carl}T}\right)e^g\;dxdt\lesssim  \int_0^{T_0}\int_{|x|\in [r,2r]\cup[\frac{R}{2},R]}|L(\psi w)|^2e^g\;dxdt \\+\int_0^{T_0}\int_{2r<|x|<\frac{R}{2}} |G|^2e^g\;dxdt.      + \int_{\R^3}|\nabla(\psi w)(x,T_0)|^2e^{g(x,T_0)}\;dx\\+\int_{\R^3}|F(x,0)||(\psi w)(x,0)|^2e^{g(x,0)}\;dx.
    \end{align}
    Next, for $|x|\in [r,2r]\cup[\frac{R}{2},R]$, we use the properties of $\psi$ to see that
    \begin{align}
        |L(\psi w)|^2\lesssim \frac{|w|^2}{C_{Carl}^2T^2}+\frac{|\nabla w|^2}{C_{Carl}T} + |G|^2.
    \end{align}
    Substituting this into the previous inequality gives
    \begin{align}
        \int_0^{T_0}\int_{2r<|x|<\frac{R}{2}}\left(\frac{|w|^2}{C_{Carl}^2T^2}+\frac{|\nabla w|^2}{C_{Carl}T}\right)e^g\;dxdt\lesssim  \int_0^{T_0}\int_{|x|\in [r,2r]\cup[\frac{R}{2},R]}\left(\frac{|w|^2}{C_{Carl}^2T^2}+\frac{|\nabla w|^2}{C_{Carl}T}\right)e^g\;dxdt \\+\int_0^{T_0}\int_{r<|x|<R} |G|^2e^g\;dxdt     + \int_{\R^3}|\nabla(\psi w)(x,T_0)|^2e^{g(x,T_0)}\;dx\\+\int_{\R^3}|F(x,0)||(\psi w)(x,0)|^2e^{g(x,0)}\;dx.
    \end{align}
    All terms without $G$ may be treated as in \cite{Tao21}. This gives
    \begin{align}
        e^{\frac{5\alpha Tr}{2}}\int_0^{\frac{T_0}{4}}\int_{10r<|x|<\frac{R}{2}}\left(\frac{|w|^2}{C_{Carl}^2T^2}+\frac{|\nabla w|^2}{C_{Carl}T}\right)\;dxdt\lesssim  \frac{X}{T}e^{2\alpha Tr}+C_{Carl}e^\frac{2R^2}{C_{Carl}T}\int_{\R^3}\frac{|w(x,0)|^2}{T}\;dx\\+\int_0^{T_0}\int_{r<|x|<R} |G|^2e^g\;dxdt . 
    \end{align}
    Now, to estimate the $G$ term, we see that on $\left\lbrace r<|x|<R \right\rbrace\times[0,T_0] $:
    \begin{equation}
         g(x,t) = \alpha (T_0-t)|x| + \frac{|x|^2}{C_{Carl}T} \leq \alpha TR +\frac{R^2}{C_{Carl}T} = \frac{3R^2}{2C_{Carl}T}.
     \end{equation}
     Hence
     \begin{equation}
         \int_0^{T_0}\int_{r<|x|<R} |G|^2e^g\;dxdt \leq e^{\frac{3R^2}{2C_{Carl}T}}\int_0^{T_0}\int_{r<|x|<R} |G|^2\;dxdt.
     \end{equation}
     Substituting this into the previous inequality yields
     \begin{align}
        e^{\frac{5\alpha Tr}{2}}\int_0^{T_0}\int_{10r<|x|<\frac{R}{2}}\left(\frac{|w|^2}{C_{Carl}^2T^2}+\frac{|\nabla w|^2}{C_{Carl}T}\right)\;dxdt\lesssim  \frac{X}{T}e^{2\alpha Tr}+C_{Carl}e^\frac{2R^2}{C_{Carl}T}\int_{\R^3}\frac{|w(x,0)|^2}{T}\;dx\\+e^{\frac{3R^2}{2C_{Carl}T}}\int_0^{T_0}\int_{r<|x|<R} |G|^2\;dxdt. 
    \end{align}
    Rearranging, using that $C_{Carl}\geq 1$ and the definition of $\alpha$ gives the claimed result.
\end{proof}

Next, we have the second Carleman inequality with incorporated forcing, corresponding to quantitative unique continuation (cf. \cite[Proposition 4.3]{Tao21}): \begin{Proposition} \label{SecondCarlemanInequality}
    Let $T,r>0$, $C_{Carl}\geq 1,$ and denote the space-time cylinder 
    $$\mathcal{C}:= \left\lbrace (x,t)\in \R^3\times\R : t\in[0,T], |x|\leq r \right\rbrace.$$
    Suppose that $w:\mathcal{C}\to\R^3$ is smooth and satisfies the differential inequality
    \begin{equation}
        |\partial_tw + \Delta w| \leq \frac{|w|}{C_{Carl}T} + \frac{|\nabla w|}{(C_{Carl}T)^\frac{1}{2}} + |G|.
    \end{equation}
    Assume that \begin{equation}
        r^2 \geq 4000T. 
    \end{equation}
    Then for all $0<t_1\leq t_0<\frac{T}{10000},$ we have the inequality:
    \begin{equation}
        Z \lesssim e^{-\frac{r^2}{500t_0}}X + t_0^\frac{3}{2}\left( \frac{6et_0}{t_1}\right)^\frac{r^2}{400t_0}Y +  \left(\frac{t_0}{t_1}\right)^\frac{3}{2}\left(\frac{6et_0}{t_1}\right)^{\frac{r^2}{400t_0}}\int_0^T\int_{|x|\leq r}(t+t_1)|G|^2e^{-\frac{|x|^2}{4(t+t_1)}}\;dxdt
    \end{equation}
    where 
    \begin{equation}
        X := \int_0^\frac{T}{4}\int_{ |x| \leq r}\frac{|w|^2}{T} +|\nabla w|^2  \;dxdt,\qquad Y:= \int_{|x|\leq r}|w(x,0)|^2t_1^{-\frac{3}{2}}e^{-\frac{|x|^2}{4t_1}}\;dx,
    \end{equation}
    \begin{equation}
        Z:= \int_{t_0}^{2t_0}\int_{|x|\leq \frac{r}{2}}\left(\frac{|w|^2}{T} +|\nabla w|^2 \right)e^{-\frac{|x|^2}{4t}}\;dxdt.
    \end{equation}
\end{Proposition}
\begin{proof}
    We closely follow the proof given by Tao \cite{Tao21}, providing sufficient details so that the reader can verify the estimates of the additional term due to the forcing $G$. As in \cite{Tao21}, observe that there exists $T_0\in (\frac{T}{200}, \frac{T}{100})$ such that \begin{equation}
        \int_{|x|\leq r} \frac{|w(x,T_0)|^2}{T}+|\nabla w(x,T_0)|^2\;dx \lesssim \frac{X}{T}.
    \end{equation}
    As in \cite{Tao21}, define \begin{equation}
        \alpha := \frac{r^2}{400t_0} \geq 10
    \quad \text{and} \quad
        g:= -\frac{|x|^2}{4(t+t_1)} -\frac{3}{2}\log(t+t_1) -\alpha\log(\frac{t+t_1}{T_0+t_1}) + \alpha\frac{t+t_1}{T_0+t_1}.
    \end{equation}
    Let $\psi \in C^\infty_c(B_r;[0,1])$ with $\psi = 1$ on $B_\frac{r}{2}$ and $|\nabla^j\psi|\lesssim r^{-j}$ on $B_r\setminus \bar{B}_\frac{r}{2}$ for $j=0,1,2$.
    Following \cite{Tao21}, we obtain the inequality:
    \begin{align}
        &\int_0^{T_0}\int_{|x|\leq \frac{r}{2}} \left(\frac{t+t_1}{10(T_0+t_1)}|\nabla w|^2 +\frac{\alpha}{10(T_0+t_1)}| w|^2 \right)e^g\;dxdt 
        \\ &\leq \int_0^{T_0}\int_{\R^3} \left( (t+t_1)|L(\psi w)|^2 \right)e^g\;dxdt \\ &\qquad + O\left(T_0\int_{|x|\leq r}|\nabla(\psi w)(x,T_0)|^2e^{g(x,T_0)}\;dx  + \alpha\int_{|x|\leq r} |w(x,0)|^2e^{g(x,0)}\;dx \right).  \label{mainCarle ineq1}
    \end{align}
    We estimate an upper bound for $\int_0^{T_0}\int_{\R^3}  (t+t_1)|L(\psi w)|^2 e^g\;dxdt$ by splitting the integral into the cases $|x|\leq \frac{r}{2}$ and $\frac{r}{2}\leq |x|\leq r.$
    \\For $|x| \leq \frac{r}{2}$, we have $\psi = 1$ and the differential inequality: 
    \begin{align}
        |L(\psi w)|^2  =  |Lw|^2 \leq \frac{4|w|^2}{(C_{Carl}T)^2} + \frac{4|\nabla w|^2}{(C_{Carl}T)} + 4|G|^2.
    \end{align}
    Therefore, 
    \begin{multline}
        \int_0^{T_0}\int_{|x| \leq \frac{r}{2}} \left( (t+t_1)|L(\psi w)|^2 \right)e^g\;dxdt \\ \leq  \int_0^{T_0}\int_{|x| \leq \frac{r}{2}} \left( \frac{4(t+t_1)|w|^2}{(C_{Carl}T)^2} + \frac{4(t+t_1)|\nabla w|^2}{(C_{Carl}T)} + 4(t+t_1)|G|^2 \right)e^g\;dxdt.
    \end{multline}
    As $C_{Carl}\geq1$, $T>100T_0,$ $t_1 <\frac{T_0}{50},$ 
    we have
    \begin{multline}
        \int_0^{T_0}\int_{|x| \leq \frac{r}{2}} \left( (t+t_1)|L(\psi w)|^2 \right)e^g\;dxdt \\\leq  \frac{4}{5}\int_0^{T_0}\int_{|x|\leq \frac{r}{2}} \left(\frac{\alpha}{10(T_0+t_1)}|w|^2+\frac{t+t_1}{10(T_0+t_1)}|\nabla w|^2 +4(t+t_1)|G|^2 \right)e^g\;dxdt. \label{blob1}  
    \end{multline}
    For $\frac{r}{2} \leq |x| \leq r$, we use the product rule, the differential inequality, and the properties of the bump function, and $r^2\geq 4000T$ to see:
    \begin{align}
        |L(\psi w)|^2  &\lesssim \frac{|w|^2}{T^2} + \frac{|\nabla w|^2}{T} +|G|^2 + \frac{|\nabla w|^2}{r^2} + \frac{|w|^2}{r^4}
        \lesssim \frac{|w|^2}{T^2} + \frac{|\nabla w|^2}{T} +|G|^2.\label{Lphiannul}
    \end{align}
    Noting that $L(\psi w) = 0$ for $|x|>r,$ we then obtain from (\ref{mainCarle ineq1}), (\ref{blob1}) and (\ref{Lphiannul}),
    \begin{align}
        &\int_0^{T_0}\int_{|x|\leq \frac{r}{2}} \left(\frac{t+t_1}{10(T_0+t_1)}|\nabla w|^2 +\frac{\alpha}{10(T_0+t_1)}| w|^2 \right)e^g\;dxdt 
        \\ &\leq \frac{4}{5}\int_0^{T_0}\int_{|x|\leq \frac{r}{2}} \left(\frac{\alpha}{10(T_0+t_1)}|w|^2+\frac{t+t_1}{10(T_0+t_1)}|\nabla w|^2 +4(t+t_1)|G|^2 \right)e^g\;dxdt \\ &\qquad+O\Bigg( \int_0^{T_0}\int_{\frac{r}{2}\leq|x|\leq r} \left( (t+t_1)\left(\frac{|w|^2}{T^2} + \frac{|\nabla w|^2}{T} +|G|^2 \right)\right)e^g\;dxdt \\ &\qquad \qquad\quad+ T_0\int_{|x|\leq r}|\nabla(\psi w)(x,T_0)|^2e^{g((x,T_0))}\;dx  + \alpha\int_{|x|\leq r} |w(x,0)|^2e^{g(x,0)}\;dx \Bigg).
    \end{align}
    By also shrinking the time domain on the left-hand side and also using that $t_1<t_0<T_0$, we obtain
    \begin{align}
        &\int_{t_0}^{2t_0}\int_{|x|\leq \frac{r}{2}} \left(\frac{t_0}{T_0}|\nabla w|^2 +\frac{\alpha}{T_0}| w|^2 \right)e^g\;dxdt \label{LHSCarlem}
        \\ &\lesssim  \int_0^{T_0}\int_{\frac{r}{2}\leq|x|\leq r} \left( (t+t_1)\left(\frac{|w|^2}{T^2} + \frac{|\nabla w|^2}{T} \right)\right)e^g\;dxdt \\ &\qquad + T_0\int_{|x|\leq r}|\nabla(\psi w)(x,T_0)|^2e^{g(x,T_0)}\;dx  + \alpha\int_{|x|\leq r} |w(x,0)|^2e^{g(x,0)}\;dx \\ &\qquad+ \int_0^{T_0}\int_{|x|\leq r} (t+t_1)|G|^2 e^g \;dxdt 
        =: I_X + I_{T_0} +I_0 + I_G.\label{RHSCarlem}
    \end{align}
    We now estimate the left-hand side of \eqref{LHSCarlem} from below, and then estimate the right-hand side \eqref{RHSCarlem} from above. As noted in \cite{Tao21}, in the region $t_0<t<2t_0$, $|x|\leq \frac{r}{2}$, we have
    \begin{align}
        g \geq -\frac{|x|^2}{4t}-\frac{3}{2}\log(3t_0)-\alpha\log\left( \frac{3t_0}{T_0+t_1}\right)
    \end{align}
    so
    \begin{align}
        e^g \gtrsim e^{ -\frac{|x|^2}{4t}}t_0^{-\frac{3}{2}}\exp\left( -\alpha\log\left( \frac{3t_0}{T_0+t_1}\right)\right).
    \end{align}
    Combining these together, we see that
    \begin{align}
        \int_{t_0}^{2t_0} \int_{|x|\leq \frac{r}{2}} \left( \frac{t_0}{T_0}|\nabla w|^2 + \frac{\alpha}{T_0}|w|^2\right)e^{ -\frac{|x|^2}{4t}} \;dxdt &\lesssim t_0^{\frac{3}{2}}\exp\left( \alpha\log\left( \frac{3t_0}{T_0+t_1}\right)\right)(I_X+I_{T_0}+I_0+I_G). \label{maincarleineqbase}
    \end{align}
    The terms $I_X,I_{T_0},I_0$ may all be treated in the same way as in \cite{Tao21}. This gives
    \begin{align}
        t_0^{\frac{3}{2}}\exp\left( \alpha\log\left( \frac{3t_0}{T_0+t_1}\right)\right)(I_X+I_{T_0}) &\lesssim e^{-\alpha}X , \label{xt0estim}
    \end{align} and
    \begin{align}
        t_0^{\frac{3}{2}}\exp\left( \alpha\log\left( \frac{3t_0}{T_0+t_1}\right)\right)I_0 &\lesssim  \alpha t_0^{\frac{3}{2}}\left( \frac{3t_0e}{t_1}\right)^\alpha Y. \label{I0estim}
    \end{align}
    Notice that $\frac{t_0}{T_0}\gtrsim \frac{t_0}{r^2} \gtrsim \frac{1}{\alpha}$ and $\frac{\alpha}{T_0} = \frac{r^2}{400t_0T_0}\gtrsim \frac{T}{400t_0T_0} \gtrsim \frac{1}{T}. $
    Using this, $\alpha\geq1,$ and \eqref{maincarleineqbase}-\eqref{I0estim} above,
    \begin{align}
        Z &= \int_{t_0}^{2t_0}\int_{|x|\leq \frac{r}{2}} \left(\frac{|w|^2}{T}+|\nabla w|^2 \right)e^{-\frac{|x|^2}{4t}}\;dxdt \lesssim  \int_{t_0}^{2t_0}\int_{|x|\leq \frac{r}{2}} \left(\frac{\alpha|w|^2}{T_0}+\frac{\alpha t_0}{T_0}|\nabla w|^2 \right)e^{-\frac{|x|^2}{4t}}\;dxdt
        \\ &\lesssim \alpha\int_{t_0}^{2t_0} \int_{|x|\leq \frac{r}{2}} \left( \frac{t_0}{T_0}|\nabla w|^2 + \frac{\alpha}{T_0}|w|^2\right)e^{ -\frac{|x|^2}{4t}} \;dxdt
        \\ &\lesssim \alpha e^{-\alpha}X +\alpha^2 t_0^{\frac{3}{2}}\left( \frac{3t_0e}{t_1}\right)^\alpha Y + \alpha  t_0^{\frac{3}{2}}\exp\left( \alpha\log\left( \frac{3t_0}{T_0+t_1}\right)\right)I_G. 
    \end{align}
    Note that as $\alpha=\frac{r^2}{400t_0}\geq 10$, we have\footnote{This is because $s\to s^\frac{1}{s}$ is decreasing on $[e,\infty)$.} $\alpha^\frac{1}{\alpha}\leq 10^\frac{1}{10}< \sqrt{2}$, so \begin{align}\alpha^2\left(\frac{3et_0}{t_1}\right)^\alpha \leq \left(\frac{6et_0}{t_1}\right)^\alpha  = \left(\frac{6et_0}{t_1}\right)^{\frac{r^2}{400t_0}} .\label{absorbalpha}\end{align}
    \\Also, as $\alpha\geq 10$, $\alpha\leq e^\frac{\alpha}{5},$ hence $\alpha e^{-\alpha} \lesssim  e^{-\frac{4\alpha}{5}} = e^{-\frac{r^2}{500t_0}}.$
    This yields
    \begin{align}
        Z \lesssim e^{-\frac{r^2}{500t_0}}X + t_0^\frac{3}{2}\left(\frac{6t_0e}{t_1}\right)^{\frac{r^2}{400t_0}}Y + \alpha  t_0^{\frac{3}{2}}\exp\left( \alpha\log\left( \frac{3t_0}{T_0+t_1}\right)\right)I_G. \label{carlemanbase2}
    \end{align}
    Finally, we estimate the $I_G$ term.
    Notice that, for $|x|\leq r$, $t\in(0,T_0)$,
    \begin{align}
        g(x,t) &= -\frac{|x|^2}{4(t+t_1)} - \frac{3}{2}\log(t+t_1) +\alpha\log\left(\frac{T_0+t_1}{t+t_1} \right) +\alpha
        \\ &\leq -\frac{|x|^2}{4(t+t_1)} +\log(t_1^{- \frac{3}{2}}) +\alpha\log\left(\frac{T_0+t_1}{t_1} \right) +\alpha
        \\ &\leq  -\frac{|x|^2}{4(t+t_1)} +\log(t_1^{- \frac{3}{2}}) +\alpha\log\left(\frac{e(T_0+t_1)}{t_1}\right).
    \end{align}
    Therefore,
    \begin{align}
        e^{g(x,t)} \lesssim t_1^{-\frac{3}{2}}e^{ -\frac{|x|^2}{4(t+t_1)}}\left(\frac{e(T_0+t_1)}{t_1}\right)^\alpha.
    \end{align}
    This gives
    \begin{align}
        \alpha  t_0^{\frac{3}{2}}\exp\left( \alpha\log\left( \frac{3t_0}{T_0+t_1}\right)\right)I_G &\lesssim  \alpha\left(\frac{t_0}{t_1}\right)^\frac{3}{2}\left(\frac{3et_0}{t_1}\right)^\alpha\int_0^\frac{T}{100}\int_{|x|\leq r}(t+t_1)|G|^2e^{-\frac{|x|^2}{4(t+t_1)}}\;dxdt.
    \end{align}
    Using the definition of $\alpha$ and the same reasoning as (\ref{absorbalpha}) gives
    \begin{align}
        \alpha  t_0^{\frac{3}{2}}\exp\left( \alpha\log\left( \frac{3t_0}{T_0+t_1}\right)\right)I_G &\lesssim \left(\frac{t_0}{t_1}\right)^\frac{3}{2}\left(\frac{6et_0}{t_1}\right)^{\frac{r^2}{400t_0}}\int_0^\frac{T}{100}\int_{|x|\leq r}(t+t_1)|G|^2e^{-\frac{|x|^2}{4(t+t_1)}}\;dxdt.
    \end{align}
    Substituting this into line (\ref{carlemanbase2}) gives the claim.
\end{proof}
\section{Proof of the main quantitative estimate via the spatial concentration method}\label{MainProof}
\begin{Proposition}
     Let $M\geq 1$ be sufficiently large. Suppose that $(v,p)$ is a smooth Leray-Hopf weak solution to the Navier-Stokes equations with forcing $F$ on $\R^3\times (-1,0)$ such that 
    \begin{align}
        \norm{v}{L^\infty_t L^3_x(\R^3\times(-1,0))} &\leq M,
        & \norm{F}{L^2_tH^1_x\cap L^6_{t,x}(\R^3\times(-1,0))} &\leq e^{-e^{M^{609}}}.\label{Quant Assumps}
    \end{align}
    Then we have that
    \begin{align}
        \norm{v}{L^\infty\left(\R^3\times (-\exp\left(-\exp(\exp\left(M^{612} \right)\right)\right),0))} \lesssim  \exp(\exp(\exp(M^{612}))). \label{Quant Bound}
    \end{align}
\end{Proposition}
This proof follows Proposition 2 of \cite{Spatialconc}, which is a physical space analogue of \cite{Tao21}, with adaptations made due to forcing. In particular, care is needed to show that the size of the force in \eqref{Quant Assumps} ensures that the additional forcing terms in the Carleman inequality are negligible. Moreover, we must use Caccioppoli estimates to estimate terms in the Carleman inequalities involving $\norm{\nabla\omega}{L^2}$ (c.f. Remark \ref{Epoch Remark}). 
\begin{proof}
    Assume that $0< -t_0 < M^{-192}$ and that the vorticity concentrates at $t_0$: 
    \begin{align}
        \int\limits_{B(0,M^{96}(-t_0)^\frac{1}{2})}|\omega(x,t_0)|^2\;dx > \frac{M^{-88}}{(-t_0)^\frac{1}{2}}. \label{vorticity concentrates}
    \end{align}
    Then (the contrapositive of) Proposition \ref{Backwards prop of vort conc} tells us that we have backwards propagation of vorticity:
    \begin{align} \label{backwardprop2}
        t_0' \in [-M^{-192},t_0], \qquad -t_0 \leq M^{-200}(-t_0') \qquad \implies \int\limits_{B(0,M^{96}(-t_0')^\frac{1}{2})}|\omega(x,t_0')|^2\;dx > \frac{M^{-88}}{(-t_0')^\frac{1}{2}}.
    \end{align}
    Using this, we aim to find an explicit upper bound on $t_0$, in which case we may select an explicit time where the contrapositive of (\ref{vorticity concentrates}) holds, hence we may apply Proposition \ref{Backwards prop of vort conc} to gain the claimed quantitative estimate.
    \\ \textbf{Step 1: Quantitative unique continuation.}
    We prove the following claim:
    \begin{Claim} \label{Step 1 claim}
        For all \  $2M^{200}(-t_0) < T_1 < M^{-192}$ and $M^{100}\left(\frac{T_1}{2}\right)^\frac{1}{2}\leq R \leq e^{M^{608}}T_1^\frac{1}{2}$, we have:
        \begin{align}
            (T_1)^\frac{1}{2}e^{\frac{-CM^{129}R^2}{T_1}} \lesssim  \int\limits_{-T_1}^{-\frac{T_1}{2}}\int_{B_{2R}(0)\setminus B_{\frac{R}{2}}(0)}|\omega(x,s)|^2\;dxds.
        \end{align}
    \end{Claim}
    \begin{proof}
        Step 1 closely follows ``Step 1: quantitative unique continuation" of \cite{Spatialconc}, with the additional forcing terms carefully handled as described below \eqref{Quant Bound}. Denote the interval $I_1:= [-T_1, -\frac{T_1}{2}] \subset [-1,0].$ By the assumption on $T_1$, observe that (\ref{backwardprop2}) is satisfied for all $s\in I_1.$

        Now, let us apply the epoch of regularity Proposition \ref{Epoch of regularity} to $I_1$: there exists a closed interval $I_1' \subset I_1$
        such that $$I_1' = [a-T_1',a] \quad \text{with} \quad T_1' = M^{-28}\frac{T_1}{2},$$ where we have the estimates:
        \begin{align} \label{regularity for Epoch}
            \norm{\nabla^k v}{L^\infty(\R^3\times I')} \lesssim \frac{1}{M}(T_1')^{-\frac{k+1}{2}} \quad \text{for} \quad k=0,1,
        \end{align}
        and for $r\geq T_1^\frac{1}{2}$, $x_0\in \R^3,$
        \begin{align} \label{vorticity estim for epoch}
            \int_{I_1'}\int_{B_r(x_0)} \frac{|\omega|^2}{T_1'}+ |\nabla\omega|^2\;dxds \lesssim M^{19}r (T_1')^{-1}.
        \end{align}
        Let $T_1'' := \frac{3}{4}T_1'$ and $s_1\in [a-\frac{T_1'}{4},a].$
        Observe that $[s_1-T_1'',s_1] \subset I'.$
        Now let $x_1 \in \R^3$ such that 
        \begin{align}
            M^{100}\left(\frac{T_1}{2}\right)^\frac{1}{2} \leq |x_1| \leq e^{M^{608}}T_1^\frac{1}{2}, \label{x_1range}
        \end{align}
        and set $r_1 := M^{50}|x_1|.$ Observe that by this and our assumption on $T_1,$ 
        \begin{align}
            r_1 \geq M^{150}\left(\frac{T_1}{2}\right)^\frac{1}{2} \geq M^{250}\left(-t_0\right)^\frac{1}{2},
        \end{align}
        and for $M$ sufficiently large,
        \begin{align}
            r_1^2 \geq M^{300}\left(\frac{T_1}{2}\right) = M^{328}T_1' = M^{328}\left(\frac{4T_1''}{3}\right) \geq 4000T_1''.
        \end{align}
        Consider the function $w(x,t):= \omega(x_1+x,s_1-t)$ on the space-time cylinder $C_1:=\left\lbrace t\in[0,T_1''], |x|<r_1\right\rbrace.$ Using the vorticity equation satisfied by $\omega $ and (\ref{regularity for Epoch}), we observe that $w$ satisfies the differential inequality for $M$ sufficiently large on $C_1$:
        \begin{align}
            |\partial_tw+\Delta w|(x,t)\leq \frac{|w(x,t)|}{T_1''} +\frac{|\nabla w(x,t)|}{(T_1'')^\frac{1}{2}} +|\nabla\times F(x_1+x,s_1-t)|.
        \end{align}
        We may then apply the second Carleman inequality (Proposition \ref{SecondCarlemanInequality}), with $C_{Carl} = 1,$ $r=r_1$, $T=T_1''$, $t_0= \frac{T_1''}{20000}$, $t_1  = M^{-150}T_1''.$ Then:
        \begin{align}
            Z_1 &\lesssim e^{-\frac{r_1^2}{500t_0}}X_1 + t_0^\frac{3}{2}\left(\frac{6et_0}{t_1} \right)^{\frac{r_1^2}{400t_0}}Y_1 +T_1''\left(\frac{t_0}{t_1}\right)^\frac{3}{2}\left(\frac{6et_0}{t_1} \right)^{\frac{r_1^2}{400t_0}}\int_0^{T_1''}\int_{|x|\leq r_1} |\nabla \times F(x_1+x,s_1-s)|^2 dxds ,\label{unsimplifiedcarl1}
        \end{align}
        where
        \begin{align}
            &X_1 := \int\limits_{s_1-\frac{T_1''}{4}}^{s_1}\int_{B_{r_1}(x_1)} \frac{|\omega(x,s)|^2}{T_1''} + |\nabla\omega(x,s)|^2 \;dxds,  \qquad
            Y_1 := \int_{B_{r_1}(x_1)}|\omega(x,s_1)|^2(t_1)^{-\frac{3}{2}}e^{-\frac{|x_1-x|^2}{4t_1}}\;dx,
    \\
            &Z_1 := \int\limits_{s_1-\frac{T_1''}{10000}}^{s_1-\frac{T_1''}{20000}}\int_{B_{\frac{r_1}{2}}(x_1)} \left(\frac{|\omega(x,s)|^2}{T_1''} + |\nabla\omega(x,s)|^2  \right)e^{-\frac{|x_1-x|^2}{4(s_1-s)}}\;dxds.
        \end{align}
        First, by using the same arguments as \cite[pp 743-744]{Spatialconc} (with different indices), we find a lower bound of $Z_1$. Indeed, using $B_{M^{96}(-s)^\frac{1}{2}}(0) \subset B_{\frac{r_1}{2}}(x_1)$ and that the backward propagation of vorticity \eqref{backwardprop2} is satisfied for all $s\in I_1'$, we obtain
        \begin{align}
            Z_1 \gtrsim  \frac{M^{-102}}{(T_1'')^\frac{1}{2}}e^{-\frac{C|x_1|^2}{T_1''}}. \label{Ztermcarl1}
        \end{align}
        Next, using the same arguments as \cite[pp 744-745]{Spatialconc} (with different indices), we find an upper bound of $Y_1$, using (\ref{regularity for Epoch}) and the properties of $x_1$ and $r_1$ (in particular, $r_1^3(T_1'')^{-\frac{3}{2}} \lesssim e^\frac{r_1^2}{T_1''} = e^\frac{M^{100}|x_1|^2}{T_1''}$). We obtain that for $M$ sufficiently large,
        \begin{align}
            t_0^\frac{3}{2}\left(\frac{6et_0}{t_1} \right)^{\frac{r_1^2}{400t_0}}Y_1  &\lesssim M^{225}e^{\frac{CM^{101}|x_1|^2}{T_1''}}\int_{B_{\frac{|x_1|}{2}}(x_1)}|\omega(x,s_1)|^2\;dx + M^{225}\frac{1}{(T_1'')^\frac{1}{2}}e^{-\frac{CM^{150}|x_1|^2}{T_1''}}.  \label{Ytermcarl1}
        \end{align}       
        Next, we use the Caccioppoli estimate (\ref{vorticity estim for epoch}) on the vorticity in the epoch of regularity to gain an upper bound on the $X_1$ term (we also use $r_1(T_1'')^{-\frac{1}{2}} \lesssim e^\frac{r_1^2}{T_1''} = e^\frac{M^{100}|x_1|^2}{T_1''}$):
        \begin{align}
            e^{-\frac{r_1^2}{500t_0}}X_1 &\lesssim e^{-\frac{40r_1^2}{T_1''}}r_1M^{19}(T_1'')^{-1}
            \lesssim e^{-\frac{40r_1^2}{T_1''}}M^{19}e^{\frac{r_1^2}{T_1''}}(T_1'')^{-\frac{1}{2}} \lesssim e^{-\frac{CM^{100}|x_1|^2}{T_1''}} (T_1'')^{-\frac{1}{2}}. \label{Xtermcarl1}
        \end{align}
        Finally, we estimate the forcing term from above. Crucially, we show that it is a lower order term for the upper bound on $|x_1|$ in \eqref{x_1range}. In particular:
        \begin{align}
            &T_1''\left(\frac{t_0}{t_1}\right)^\frac{3}{2}\left(\frac{6et_0}{t_1} \right)^{\frac{r_1^2}{400t_0}}\int_0^{T_1''}\int_{|x|\leq r_1} |\nabla \times F(x_1+x,s_1-s)|^2 dxds 
            \\&\lesssim T_1''M^{225}\left(\frac{6eM^{150}}{20000} \right)^{\frac{50M^{100}|x_1|^2}{T_1''}}\norm{F}{L^2_t\dot{H}^{1}_x(\R^3\times(-1,0))}^2
            \lesssim (T_1'')^{-\frac{1}{2}}e^{-\frac{CM^{101}|x_1|^2}{T_1''}}. \label{forcingtermcarl1}
        \end{align}
        In \eqref{forcingtermcarl1}, we used that $|x_1| \leq e^{M^{608}}T_1^\frac{1}{2}$ hence $M^{225}T_1''e^{-2e^{M^{609}}}\lesssim e^{-\frac{CM^{101|x_1|^2}}{T_1''}} \left(\frac{6eM^{150}}{20000} \right)^{-\frac{50M^{100}|x_1|^2}{T_1''}}(T_1'')^{-\frac{1}{2}}$.
        \\Substituting (\ref{Ztermcarl1})-\eqref{Ytermcarl1} and (\ref{forcingtermcarl1}) into the Carleman inequality (\ref{unsimplifiedcarl1}) yields: 
        \begin{align}
            \frac{M^{-102}}{(T_1'')^\frac{1}{2}}e^{-\frac{C|x_1|^2}{T_1''}} &\lesssim e^{-\frac{CM^{100}|x_1|^2}{T_1''}}(T_1'')^{-\frac{1}{2}} +  M^{225}e^{\frac{CM^{101}|x_1|^2}{T_1''}}\int_{B_{\frac{|x_1|}{2}}(x_1)}|\omega(x,s_1)|^2\;dx \\&\quad+ M^{225}\frac{1}{(T_1'')^\frac{1}{2}}e^{-\frac{CM^{150}|x_1|^2}{T_1''}} + e^{-\frac{CM^{101}|x_1|^2}{T_1''}}(T_1'')^{-\frac{1}{2}}.
        \end{align}
        Hence, using \eqref{x_1range} we may absorb all but the second term into the left-hand side. Writing $R:=|x_1|$ using $B_{\frac{R}{2}}(x_1) \subset B_{2R}(0)\setminus B_{\frac{R}{2}}(0) $ and subsequently integrating over $[a-\frac{T_1'}{4},a]$ gives: 
        \begin{align}
            T_1^\frac{1}{2}e^{\frac{-CM^{129}R^2}{T_1}} \lesssim  \int\limits_{a-\frac{T_1'}{4}}^{a}\int_{B_{2R}(0)\setminus B_{\frac{R}{2}}(0)}|\omega(x,s)|^2\;dxds.
        \end{align}
        This is what was claimed, observing that $[a-\frac{T_1'}{4},a] \subset [-T_1, -\frac{T_1}{2}].$
    \end{proof}
    \begin{Claim} \label{Step 2 and 3 claim}
        For all $8M^{400}(-t_0)<T_2 \leq 1$,
        we have \begin{equation}
            T_2^{-\frac{1}{2}}\exp\left(-\exp\left(M^{608}\right)\right) \lesssim \int\limits_{B_{\frac{3M^{600}\hat{R}}{16}}\setminus B_{2\hat{R}}}|\omega(x,0)|^2\;dx
        \end{equation}
        where $\hat{R}$ is defined below in Step 2.
    \end{Claim}
    \begin{proof}
    \textbf{Step 2: Quantitative backward uniqueness.} \quad Step 2 follows ``Step 2: quantitative backwards uniqueness" of \cite{Spatialconc}, with the additional forcing terms carefully handled as described below \eqref{Quant Bound}. We apply the existence of the annulus of regularity (Proposition \ref{ann of reg M version}) with the first Carleman inequality (Proposition \ref{FirstCarlemanInequality}). 
    \\Let $T_2' := \frac{T_2}{M^{200}}$ and let $R_2 := MT_2^\frac{1}{2}$, where $M$ is sufficiently large. By Proposition \ref{ann of reg M version} (with $\mu=200$), there exists an annulus of regularity \begin{align} A_2 = \left\lbrace \hat{R} < |x| <\frac{M^{600}\hat{R}}{4} \right\rbrace \quad\text{where}\quad 
        \hat{R}=\hat{R}(v,p,F,R) \in [2R_2,R_2\exp{\left( M^{607}\right)}]\end{align}
    such that, 
    for $k=0,1:$
    \begin{align}\label{annulus of reg step 2}         \norm{\nabla^kv}{L^\infty\left(A_2\times (-\frac{T_2}{16},0)\right)} &\lesssim M^{-200}T_2^{-\frac{k+1}{2}}, &        \norm{\omega}{C^{0,\frac{1}{6}}_\text{par}\left(A_2\times (-\frac{T_2}{16},0)\right)} &\lesssim M^{-200}T_2^{-\frac{13}{12}}   .
    \end{align} 
    Moreover, if $0<t<\frac{T_2}{16}$ and $r\geq T_2^\frac{1}{2}$ such $B_r\setminus B_\frac{r}{2}\subset B_{2r}\setminus B_\frac{r}{4}\subseteq A$, then
 \begin{align}
        \int_{-t}^0\int_{B_r\setminus B_\frac{r}{2}} \frac{|\omega|^2}{T_2}+|\nabla\omega|^2\;dxds &\lesssim M^{-400}\left(\frac{t}{T_2}\right)^\frac{1}{12}\left(\frac{r^2}{T_2}\right)^\frac{3}{2}T_2^{-\frac{1}{2}}+ t^\frac{2}{3}r^2 \norm{F}{L^6(\R^3\times(-T_2,0))}^2. \label{annofregL2gradvort}
    \end{align}
    Now, denote \begin{align}\mathcal{A}_2:= \big\lbrace 4\hat{R} < |x| < \frac{M^{600}\hat{R}}{16} \big\rbrace \times (-T_2',0) \subset A_2\times (-\frac{T_2}{16}, 0).\end{align}
    \\ We wish to apply the first Carleman inequality to $w(x,t):= \omega(x,-t)$, with $``r=4\hat{R},\; R=\frac{M^{600}}{16}\hat{R},\; T=T_2',\; C_{carl}= M^{200}."$ Indeed, we note that $r^2 = 16 \hat{R}^2 \geq 64R_2^2 \geq 64T_2 = 64M^{200} T_2' \geq 4C_{carl}T_2'.$
    Using the vorticity equation and the estimates 
    (\ref{annulus of reg step 2}), we also observe that the differential inequality holds on $\mathcal{A}_2$:
    \begin{align}
            |\partial_tw+\Delta w|(x,t) \leq \frac{|w(x,t)|}{M^{200} T_2'} +\frac{|\nabla w(x,t)|}{(M^{200} T_2')^\frac{1}{2}} + |\nabla\times F(x,-t)|.
        \end{align}
    Applying the first Carleman inequality (Proposition \ref{FirstCarlemanInequality}) then gives: 
    \begin{equation}
        \int\limits^0_{-\frac{T_2}{4M^{200}}}\int\limits_{40\hat{R}\leq |x|\leq \frac{M^{600}\hat{R}}{32}} \frac{|\omega|^2}{T_2'} \;dxdt \lesssim M^{600}e^{-\frac{M^{600}\hat{R}^2}{16 T_2}}\left( X_2  + e^\frac{M^{1200}\hat{R}^2}{128T_2}Y_2  + T_2e^\frac{3M^{1200}\hat{R}^2}{512T_2}\norm{\nabla\times F}{L^2(\R^3\times(-1,0))}^2\right), \label{unsimplifiedCarl2}
    \end{equation}
    where 
    \begin{align}
        X_2 &:= \int_{-\frac{T_2}{M^{200}}}^0\int_{4\hat{R}\leq |x|\leq \frac{M^{600}\hat{R}}{16}} e^\frac{2|x|^2}{T_2}\left( \frac{|\omega|^2}{T_2'} + |\nabla \omega|^2 \right) \;dxdt ,
        & Y_2 &:= \int\limits_{4\hat{R}\leq |x|\leq \frac{M^{600}\hat{R}}{16}} |\omega(x,0)|^2\;dx.
    \end{align}
    Similarly to \cite{Spatialconc}, we now apply Claim \ref{Step 1 claim} from Step 1 to gain a lower bound on the left-hand-side of (\ref{unsimplifiedCarl2}).
    Indeed, let $T_1:= \frac{T_2'}{4} = \frac{M^{-200}T_2}{4}.$
    By our assumption on $T_2,$ we have $2M^{200}(-t_0)< T_1 < M^{-192}.$ We also let $R^* := 80\hat{R}$ so that, for $M$ sufficiently large, we have
    \begin{align}
        R^*\geq 160R_2 \geq 160T_2^\frac{1}{2} = 320M^{100}T_1^\frac{1}{2} \geq M^{100}\left(\frac{T_1}{2}\right)^\frac{1}{2},
    \\
        \text{and}\quad R^* \leq 2M^{101}T_1^\frac{1}{2}e^{
        M^{607}} \leq T_1^\frac{1}{2}e^{M^{608}}.
    \end{align}
    These conditions allow us to apply Claim \ref{Step 1 claim} from Step 1 - this gives:
    \begin{align}
            M^{-100}T_2^\frac{1}{2}e^{\frac{-CM^{329}(R^*)^2}{T_2}} \lesssim  \int\limits_{-\frac{T_2}{4M^{200}}}^{-\frac{T_2}{8M^{200}}}\int_{B_{2R^*}(0)\setminus B_{\frac{R^*}{2}}(0)}|\omega(x,s)|^2\;dxds.
    \end{align}
    Similarly to \cite[Step 2]{Spatialconc}, using the Carleman inequality (\ref{unsimplifiedCarl2}), and using that $\hat{R}\geq 2MT_2^\frac{1}{2}$, we see that for $M$ sufficiently large 
    \begin{align}
            T_2^{-\frac{1}{2}}e^{\frac{CM^{600}\hat{R}^2}{T_2}}  &\lesssim  X_2  + e^\frac{M^{1200}\hat{R}^2}{128T_2}Y_2  + T_2e^\frac{3M^{1200}\hat{R}^2}{512T_2}\norm{F}{L^2_t\dot{H}^1_x(\R^3\times(-1,0))}^2.\label{calemaforceabsorb}
    \end{align}
    Now, noticing that 
    \begin{align} \norm{F}{L^2_t\dot{H}^1_x(\R^3\times(-1,0))}^2\leq e^{-e^{2M^{609}}}\leq e^{\frac{-CM^{1201}\hat{R}^2}{T_2}}(T_2)^{-\frac{3}{2}},\end{align}
    we may absorb the forcing term into the left-hand side of \eqref{calemaforceabsorb}. Therefore, at least one of the two following bounds occurs:
    \begin{align}
         &(a)\;X_2 \gtrsim T_2^{-\frac{1}{2}} e^{\frac{CM^{600}\hat{R}^2}{T_2}}, & &(b)\;  Y_2 \gtrsim e^\frac{-CM^{1200}\hat{R}^2}{T_2}T_2^{-\frac{1}{2}}.
            \nonumber
    \end{align}
    In the case of (b), we can use the estimate directly - simply rearrange and use the upper bound of $\hat{R}$ to obtain the claimed estimate.
    Finally, case (a) requires more work to complete the proof of this claim:
    \textbf{Step 3: Apply quantitative unique continuation again.}\quad Step 3 follows ``Step 3: a final application of quantitative unique continuation" in \cite{Spatialconc}, with the additional forcing terms carefully handled as described below \eqref{Quant Bound}.
    For this step, assume that we are in the case of (a).
    By the pigeonhole principle\footnote{Else, the opposite holds for all radii in this range; set $R^{(0)}=8\hat{R}$, $R^{(k)}=2R^{(k-1)}$ for $k\leq\floor{C\log{M}}$. Summing over these radii would contradict (a).}, there exists $8\hat{R}\leq R_3\leq \frac{M^{600}\hat{R}}{8}$ such that
    \begin{align}
        T_2^{-\frac{1}{2}}\lesssim \int_{-\frac{T_2}{M^{200}}}^0\int_{B_{R_3}\setminus B_{\frac{R_3}{2}}} e^{\frac{2|x|^2}{T_2}}\left(\frac{|\omega|^2}{T_2} + |\nabla\omega|^2\right)\;dxds.
    \end{align}
    Therefore:
    \begin{align} \label{subannulus by pigeonhole}
        T_2^{-\frac{1}{2}}\exp\left(-\frac{2R_3^2}{T_2} \right) \lesssim \int_{-\frac{T_2}{M^{200}}}^0\int_{B_{R_3}\setminus B_{\frac{R_3}{2}}} \frac{|\omega|^2}{T_2} + |\nabla\omega|^2\;dxds.
    \end{align}
    Now, observe that $B_{R_3}\setminus B_{\frac{R_3}{2}}\subset B_{2R_3}\setminus B_{\frac{R_3}{4}} \subset A_2,$ the annulus of regularity. Therefore, estimates (\ref{annulus of reg step 2}) and (\ref{annofregL2gradvort}) hold here; (\ref{annofregL2gradvort}) gives 
    \begin{align}
        \int_{-\exp{\left(-\frac{60R_3^2}{T_2}\right)}T_2}^0\int_{B_{R_3}\setminus B_{\frac{R_3}{2}}} \frac{|\omega|^2}{T_2}+|\nabla\omega|^2\;dxds &\lesssim \exp{\left(-\frac{5R_3^2}{T_2}\right)}\left(\frac{R_3^2}{T_2}\right)^\frac{3}{2}T_2^{-\frac{1}{2}}+ \exp{\left(-\frac{40R_3^2}{T_2}\right)}T_2^\frac{2}{3}R_3^2 \norm{F}{L^6(\R^3\times(-1,0))}^2.\label{smalltimeestimstep3}
    \end{align}
    Now, for $M$ sufficiently large, as $\norm{F}{L^6(\R^3\times(-1,0))}^2\lesssim e^{-2e^{M^{609}}} \lesssim \frac{T_2}{R_3^2}\lesssim\frac{1}{R_3^2}$, we have
    \begin{align}
        \int_{-\exp{\left(-\frac{60R_3^2}{T_2}\right)}T_2}^0\int_{B_{R_3}\setminus B_{\frac{R_3}{2}}} \frac{|\omega|^2}{T_2}+|\nabla\omega|^2\;dxds &\lesssim T_2^{-\frac{1}{2}}\exp\left(-\frac{3R_3^2}{T_2} \right).
    \end{align}
    Combining this observation with line (\ref{subannulus by pigeonhole}), we have 
    \begin{align}
        T_2^{-\frac{1}{2}}\exp\left(-\frac{2R_3^2}{T_2} \right) \lesssim \int_{-\frac{T_2}{M^{200}}}^{-\exp{\left(-\frac{60R_3^2}{T_2}\right)}T_2}\int_{B_{R_3}\setminus B_{\frac{R_3}{2}}} \frac{|\omega|^2}{T_2} + |\nabla\omega|^2\;dxds.
    \end{align}
    Next, by applying the pigeonhole principle twice again\footnote{For the spatial pigeonhole, consider covering $B_{R_3}\setminus B_{\frac{R_3}{2}}$ with $C\frac{R_3^3}{(-t_3)^\frac{3}{2}}\leq \exp(\frac{91R_3^2}{T_2})$ balls of radius $(-t_3)^\frac{1}{2}$.}, there exists a time \begin{align}\label{t3range}
        \frac{\exp{\left(-\frac{60R_3^2}{T_2}\right)}T_2}{2}<-t_3 < \frac{T_2}{M^{200}},
    \end{align} and there exists a point $x_3 \in B_{R_3}\setminus B_{\frac{R_3}{2}}$ such that
    \begin{align}\label{pigeonholestimstep3}
        T_2^{-\frac{1}{2}}\exp\left(-\frac{100R_3^2}{T_2}\right) \lesssim \int_{2t_3}^{t_3}\int_{B_{\sqrt{-t_3}}(x_3)}\frac{|\omega|^2}{T_2} + |\nabla\omega|^2\;dxds.
    \end{align}
    This puts us in a position to apply the second Carleman inequality again, to the function $w:\R^3\times [0,-20000t_3]\to\R^3$ defined by $w(x,t):=\omega(x+x_3,-t).$
    We also take $T_3=-20000t_3$, $r_3:=1000R_3\left(\frac{-t_3}{T_2}\right)^\frac{1}{2},$ $\widehat{t}_3 = \widecheck{t}_3 = -t_3 .$
    We then check that for $M$ sufficiently large
    \begin{align}
     \label{r31estim}
        &\qquad r_3^2 = 10^6R_3^2\frac{(-t_3)}{T_2}\geq 256\cdot10^6M^2(-t_3)\geq 4000T_3,
    \\ \label{r32estim}
        \text{and}&\qquad \frac{r_3}{2}\geq 4000\hat{R}\left(\frac{-t_3}{T_2}\right)^\frac{1}{2}\geq 8000M(-t_3)^\frac{1}{2} > (-t_3)^\frac{1}{2},  
    \\ \label{x3 estim}
        \text{and}& \qquad \frac{|x_3|}{2}\geq \frac{R_3}{4}= \frac{r_3}{4000}\left(\frac{T_2}{-t_3}\right)^\frac{1}{2}\geq \frac{r_3M^{100}}{4000}\geq 2r_3 \geq r_3.  
    \end{align}
    Hence, by (\ref{r32estim})-(\ref{x3 estim}),
    \begin{align}\label{balnestingCarl3}
        B_{(-t_3)^\frac{1}{2}}(x_3) \subseteq B_{\frac{r_3}{2}}(x_3)  \subseteq B_{2r_3}(x_3) \subseteq B_{\frac{|x_3|}{2}}(x_3) \subseteq \Big\lbrace \frac{R_3}{4}<|y|<\frac{3R_3}{2}  \Big\rbrace \subseteq \Big\lbrace 8\hat{R}<|y|<\frac{M^{600}\hat{R}}{8}  \Big\rbrace.
    \end{align}
    We also have that $0\leq \widehat{t}_3 = \widecheck{t}_3=-t_3 \leq -2t_3 = \frac{T_3}{10^4}.$
    Also, for $M$ sufficiently large, 
       $T_3 = -20000t_3 \leq \frac{20000T_2}{M^{200}} \leq \frac{T_2}{16}, \label{T_2vsT_3}$ hence
    the annulus of regularity results (\ref{annulus of reg step 2})-(\ref{annofregL2gradvort}) hold within $B_{r_3}\times [0, T_3].$
    Finally, we observe that the following inequality holds in $B_{r_3}\times [0, T_3]$ with $C_{carl}=1$ for $M$ sufficiently large:
    \begin{align}
            |\partial_tw+\Delta w|(x,t) &\leq  \frac{|w(x,t)|}{ T_3} +\frac{|\nabla w(x,t)|}{(T_3)^\frac{1}{2}} + |\nabla\times F(x,-t)|.
        \end{align}
    With these conditions checked, we may apply the second Carleman inequality. Also using (\ref{pigeonholestimstep3}) (along with the fact that $T_2\geq T_3$) and the ball nesting (\ref{balnestingCarl3}), we obtain:
    \begin{align}\label{Carleman3rd}
        T_2^{-\frac{1}{2}}\exp\left(-\frac{100R_3^2}{T_2}\right)  \lesssim e^{\frac{r_3^2}{500t_3}}X_3+ (-t_3)^\frac{3}{2}(6e)^{-\frac{Cr_3^2}{t_3}}Y_3 +T_3\left(6e \right)^{-\frac{r_3^2}{400t_3}}\int^0_{-T_3}\int_{|x|\leq r_3} |\nabla \times F(x,s)|^2 dxds,
    \end{align}
    where:
    \begin{align}
        X_3 &= \int_{-\frac{T_3}{4}}^0\int_{B_{r_3}(x_3)}\frac{|\omega|^2}{T_3}+|\nabla\omega|^2\;dxds,
        & Y_3&= \int_{B_{r_3}(x_3)}|\omega(x,0)|^2(-t_3)^{-\frac{3}{2}}e^{\frac{|x-x_3|^2}{4t_3}}\;dx.
    \end{align}
    Next, we estimate the $X_3$ term from above using the Caccioppoli inequality (Proposition \ref{CaccioppoliBall}) using observation (\ref{balnestingCarl3}), (\ref{annulus of reg step 2}),  $T_3 \lesssim r_3^2$ by (\ref{r31estim}) and $T_3\leq T_2$ by (\ref{T_2vsT_3}). This gives:
    \begin{align}
        e^{\frac{r_3^2}{500t_3}}X_3 &\lesssim  e^{\frac{r_3^2}{500t_3}}\left(\left(\frac{1}{T_3}+\frac{1}{r_3^2}+\norm{v}{L^\infty(A_2\times (-\frac{T_2}{16},0))}^2\right)\int_{-2T_3}^{0}\int_{B_{2r_3}(x_3)}|\omega|^2\;dxds+\norm{F}{L^2((\R^3\times(-1,0)))}^2\right)
        \\ &\lesssim  e^{\frac{r_3^2}{500t_3}}\left( T_3^{-2}r_3^3+\norm{F}{L^2((\R^3\times(-1,0)))}^2 \right).\label{x3fromabove1}
    \end{align}
    We may then bound the first term above as follows, using the definition of $T_3=-20000t_3$ and the range of $t_3$ (\ref{t3range}):
    \begin{align}
        e^{\frac{r_3^2}{500t_3}}T_3^{-2}r_3^3 &\lesssim e^{\frac{r_3^2}{500t_3}}(-t_3)^{-2}r_3^3\lesssim e^{\frac{r_3^2}{500t_3}}\left(\frac{r_3^2}{-t_3}\right)^\frac{3}{2}(-t_3)^{-\frac{1}{2}}\lesssim e^{\frac{r_3^2}{1000t_3}} (-t_3)^{-\frac{1}{2}}
        \\ &\lesssim e^{\frac{r_3^2}{1000t_3}} T_2^{-\frac{1}{2}}e^{\frac{30R_3^2}{T_2}}\lesssim T_2^{-\frac{1}{2}}e^{\frac{-970R_3^2}{T_2}}\lesssim  T_2^{-\frac{1}{2}}e^{\frac{-100R_3^2}{T_2}}e^{\frac{-870R_3^2}{T_2}}
        \lesssim T_2^{-\frac{1}{2}}e^{\frac{-100R_3^2}{T_2}}e^{-870\cdot (16M)^2}. \label{x3fromabove2}
    \end{align} 
    Then for $M$ sufficiently large, combining lines (\ref{x3fromabove1})-(\ref{x3fromabove2}) shows that the term $e^{\frac{r_3^2}{500t_3}}T_3^{-2}r_3^3 $ may be absorbed into the left-hand side of line (\ref{Carleman3rd}). 
    Furthermore, we may also absorb the second term arising from the $X_3$ term into the left-hand side of line (\ref{Carleman3rd}). Indeed, since $R_3<\frac{M^{600}\hat{R}}{8}$ and $\hat{R}\leq MT^\frac{1}{2}e^{M^{607}}$, we have that for $M$ sufficiently large:
    \begin{align}
        \exp\left(-\frac{100R_3^2}{T_2}\right) \geq Me^{-e^{M^{609}}}.
    \end{align}
    It follows that
    \begin{align}
        e^{\frac{r_3^2}{500t_3}}\norm{F}{L^2((\R^3\times(-1,0)))}^2 &\lesssim \hat{C}_{2}e^{-2e^{M^{609}}} 
        \lesssim T_2^{-\frac{1}{2}}M^{-1}\exp\left(-\frac{100R_3^2}{T_2}\right). \label{x3fromabove F}
    \end{align}
    Similarly, if $M$ is sufficiently large, the forcing term from the Carleman inequality may also  be absorbed into the left hand side of (\ref{Carleman3rd}):
    \begin{align}
        T_3\left(6e \right)^{-\frac{r_3^2}{400t_3}}\int^0_{-T_3}\int_{|x|\leq r_3} |\nabla \times F(x,s)|^2 dxds &\leq (6e)^{\frac{2500R_3^2}{T_2}}T_3\norm{ F}{L^2\dot{H}^1(\R^3\times(-1,0))}^2 
        \\ &\leq T_3e^{e^{M^{608}}} e^{-2e^{M^{609}}} 
        \leq M^{-1}T_2^{-\frac{1}{2}}\exp\left(-\frac{100R_3^2}{T_2}\right). \label{F3fromabove}
    \end{align}
    Combining lines (\ref{Carleman3rd})-(\ref{F3fromabove}) then gives us 
    \begin{align}
        T_2^{-\frac{1}{2}}\exp\left(-\frac{100R_3^2}{T_2}\right) \lesssim (6e)^{-\frac{Cr_3^2}{t_3}}\int_{B_{r_3}(x_3)}|\omega(x,0)|^2\;dx 
        \lesssim \exp{\left(\frac{CR_3^2}{T_2}\right)}\int_{B_{r_3}(x_3)}|\omega(x,0)|^2\;dx.
    \end{align}
    Substituting the upper bound for $R_3$, and using the containment of $B_{r_3}(x_3)$ in the annulus of regularity $B_{\frac{3M^{600}\hat{R}}{16}}\setminus B_{2\hat{R}}$ gives the Claim \ref{Step 2 and 3 claim} in the case (a). 
    \end{proof}
    \textbf{Step 4 - Summing of scales and conclusion:}
    \quad As in Step 4 of \cite{Spatialconc} and \cite{Tao21}, we transfer Claim \ref{Step 2 and 3 claim} to a lower bound on the $L^3$ norm of the velocity and then sum disjoint scales to conclude. In our case with forcing with low-regularity bounds, we cannot use the Lipschitz norm of the vorticity to do this as in the unforced case. Instead, we work with appropriate H\"older norms of the vorticity for this purpose.  
    
    Let $8M^{400}(-t_0)<T_2 \leq 1$ and set $r_4:=T_2^{\frac{1}{2}}\exp\left(-6\exp\left(M^{609}\right)\right)<\hat{R}$. By using Claim \ref{Step 2 and 3 claim} and applying the pigeonhole principle, there exists $x_4\in B_{\frac{3M^{600}\hat{R}}{16}}\setminus B_{2\hat{R}}$ and $i\in \lbrace 1,2,3 \rbrace$ such that
    \begin{align}
        |\omega_i(x_4,0)| \geq 2T_2^{-1}\exp\left(-\exp\left(M^{609}\right)\right). \label{bnd below of vort at t=0}
    \end{align}
    In particular, observe that $B_{r_4}(x_4)$ is contained in $A_2$. Thus, using the H\"older continuity (\ref{annulus of reg step 2}) of the vorticity in the annulus $A_2$, we see that for all $x\in B_{r_4}(x_4),$ \begin{align}
        |\omega_i(x,0)-\omega_i(x_4,0)|\leq \norm{\omega(\cdot,0)}{C^{0,\frac{1}{6}}(A_2)}r_4^\frac{1}{6}
        &\leq
        T_2^{-1}\exp\left(-\exp\left(M^{609}\right)\right).
    \end{align}
    Hence, using (\ref{bnd below of vort at t=0}), 
    \begin{align}
         |\omega_i(x,0)| &\geq |\omega_i(x_4,0)|-|\omega_i(x,0)-\omega_i(x_4,0)| \\&\geq 2T_2^{-1}\exp\left(-\exp\left(M^{609}\right)\right)-T_2^{-1}\exp\left(-\exp\left(M^{609}\right)\right) = T_2^{-1}\exp\left(-\exp\left(M^{609}\right)\right).
    \end{align}
    It then also follows that $\omega_i$ has constant sign within $B_{r_4}(x_4)$. 
    \\We may now follow in verbatim the rest of Step 4 of \cite[pp. 751-752]{Spatialconc}, with adjusted indices, as this part of the proof is not affected by forcing. 
    Indeed, using these observations, integration by parts, H\"older's inequality and the definition of $\hat{R}$, we have that for all $ 8M^{400}(-t_0)\leq T_2\leq 1$,
    \begin{align}
        \int\limits_{B\left(\exp\left(M^{608}\right)T_2^\frac{1}{2}\right)\setminus B(T_2^\frac{1}{2})}|v(x,0)|^3\;dx \geq \exp\left(-\exp\left(M^{610}\right)\right). \label{Key L3 estim at t=0}
    \end{align}
    Then by summing (\ref{Key L3 estim at t=0}) over $K:= \floor{M^{-1216}\log\left(\frac{1}{8M^{400}(-t_0)}\right)}$ scales\footnote{$K\geq 1$ if $0 < -t_0 \leq \frac{1}{8}M^{-400}\exp\left(-M^{608}\right). $ Else, we may just take $T_2=1$ in \eqref{Key L3 estim at t=0} to obtain (\ref{t_0lowerbnd}) instantly.} (for $k=0,1,...\,,\, K,$ set $T_2^{(0)}:= 8M^{400}(-t_0)$ and $T_2^{(k)} := \exp\left(kM^{1216}\right)8M^{400}(-t_0)$) and rearranging, we have
    \begin{align}
        -t_0 \geq \frac{1}{8}M^{-400}\exp\left(-\exp(\exp\left(M^{611} \right)) \int\limits_{\R^3}|v(x,0)|^3\;dx\right) .\label{t_0lowerbnd}
    \end{align}    
    This allows us to define $-s_0 := \frac{1}{16}M^{-400}\exp\left(-\exp(\exp\left(M^{611} \right)) M^3\right)$ so that the contrapositive of (\ref{vorticity concentrates}) holds. But then by Proposition \ref{vorticity doesn't cocentrate into quantative regularity}, we see that, for $M$ sufficiently large,
    \begin{align}
        \norm{v}{L^\infty\left(B\left(M^{92}(-s_0)^\frac{1}{2}\right)\times (\frac{s_0}{2},0)\right)} \lesssim (-s_0)^{-\frac{1}{2}}\leq \exp(\exp(\exp(M^{612}))).
    \end{align}
    As the bounds are spatially translation invariant, we infer that 
    \begin{align}
        \norm{v}{L^\infty(\R^3\times(\exp(-\exp(\exp(M^{612}))),0))}\lesssim \exp(\exp(\exp(M^{612}))).
    \end{align}
\end{proof}
Rescaling what we have proven, we have:
\begin{Proposition}\label{QuantestimRescaled}
    Let $M\geq 1$ be sufficiently large. Suppose that $(v,p)$ is a smooth Leray-Hopf weak solution to the Navier-Stokes equations with forcing $F$ on $\R^3\times (-1,0)$ such that 
    \begin{align}
        \norm{v}{L^\infty_t L^3_x(\R^3\times(-1,0))}+\norm{F}{L^2_tH^1_x\cap L^6_{t,x}(\R^3\times(-1,0))} \leq M.
    \end{align}
    Then we have that
    \begin{align}
        \norm{v}{L^\infty\left(\R^3\times (-\exp\left(-\exp(\exp\left(M^{613} \right))\right),0)\right)} \lesssim  \exp(\exp(\exp(M^{613}))).
    \end{align}
\end{Proposition}
\begin{proof}
    Rescale; let $\lambda:= e^{-e^{e^M}}$ and let $v_\lambda, F_\lambda:\R^3\times (-1,0)\to\R^3$ be defined by 
    \begin{align}
        v_\lambda(x,s) := \lambda v(\lambda x, \lambda^2s)
        \qquad F_\lambda(x,s) := \lambda^3 F(\lambda x, \lambda^2s).
    \end{align}
    Then using that $\lambda<1,$ we have
    $\norm{v_\lambda}{L^\infty_t L^3_x(\R^3\times(-1,0))}\leq M$
    and
    $\norm{F_\lambda}{L^2_tH^1_x\cap L^6_{t,x}(\R^3\times(-1,0))} \leq \lambda^\frac{1}{2}M \leq e^{-e^{M^{609}}}.$
    Then the preceding Proposition tells us that
    \begin{align}
        \norm{v_\lambda}{L^\infty(\R^3\times (-\exp\left(-\exp(\exp\left(M^{612} \right))\right),0))}\lesssim \exp(\exp(\exp(M^{612}))).
    \end{align}
    Rescaling back to $v$ and ensuring that $M$ is sufficiently large, we get the claimed result.
\end{proof}
Rescaling again:
\begin{Proposition}\label{Rescaledquantestimtime}
     Let $M\geq 1$ be sufficiently large and let $t_0<t_1<t_0+1\in\R$. Suppose that $(v,p)$ is a smooth suitable Leray-Hopf weak solution to the Navier-Stokes equations with forcing $F$ on $\R^3\times (t_0,t_1).$ Assume that 
    \begin{align}
        \norm{v}{L^\infty_t L^3_x(\R^3\times (t_0,t_1))}&+\norm{F}{L^2_tH^1_x\cap L^6_{t,x}(\R^3\times(t_0,t_1))} \leq M.
    \end{align}
    Then we have
    \begin{align}
        \norm{v(\cdot,t)}{L^\infty\left(\R^3\right)} \lesssim  \frac{1}{(t-t_0)^\frac{1}{2}}\exp(\exp(\exp(M^{613}))) \qquad \text{for all } t\in (t_0,t_1).
    \end{align}
\end{Proposition}
\begin{proof}
    Rescale; let $\hat{v}, \hat{F}:\R^3\times (-1,0)\to\R^3$ be defined by 
    \begin{align}
        \hat{v}(x,s) &:= (t-t_0)^\frac{1}{2}v((t-t_0)^\frac{1}{2}x, (t-t_0)s+t),
        &\hat{F}(x,s) &:= (t-t_0)^\frac{3}{2}F((t-t_0)^\frac{1}{2}x, (t-t_0)s+t).
    \end{align}
\end{proof}
\section{Application of the quantitative estimate I: localized Orlicz-type quantity blow-up}\label{Orlicz Proof}
\subsection{Global mild criticality breaking for the forced Navier-Stokes equations}\label{mildcritbreakingsection}
For $t_0\in\R$, $S>0$, define the $L^4$ energy as
\begin{equation}
\begin{split}\nonumber
    \mathcal{E}_{4,t_0}(S):= &\sup_{t_0<s<t_0+S}\norm{v(\cdot,s)}{L^4_x(\R^3)}^4 + 12\int_{t_0}^{t_0+S}\int_{\R^3}|\nabla v|^2|v|^{2}\;dxds
\end{split}
\end{equation}
\begin{Lemma} \label{L^p energy estimate}
    Let $t_0\in \R$ and $s>0$. There exists a universal constant $C>0$ such that the following holds: Let $v$ be a smooth Leray-Hopf weak solution to the Navier-Stokes equations on $\R^3\times [t_0,t_0+S]$ with smooth forcing $f\in L^1_tL^4_x(\R^3\times [t_0,t_0+S])$ and initial data $v(\cdot,t_0)\in L^2(\R^3)\cap L^4(\R^3).$
    Then we have the following estimate:
    \begin{equation}
        \mathcal{E}_{4,t_0}(S) \leq 4\norm{v(\cdot,t_0)}{L^4(\R^3)}^4+ C\norm{v}{L^5(\R^3\times(t_0,t_0+S))}\mathcal{E}_{4,t_0}(S) + C\norm{f}{L^1_tL^4_x(\R^3\times(t_0,t_0+S))}.\label{L^p specific energy est}
    \end{equation}
\end{Lemma}
This is in the same spirit as \cite[Proposition 4]{Mildcriticality}. Here, we incorporate forcing; in particular, one must deal with the part of the pressure arising from the forcing term.
\begin{proof}
     We decompose the pressure into $p=p_{v\otimes v}+p_f$ where $p_{v\otimes v}:= \mathcal{R}_i\mathcal{R}_j(v_iv_j)$ and $p_f:= \Delta^{-1}\operatorname{div}(f).$ Then, using H\"older's inequality and Calder\'on-Zygmund estimates, we have
    \begin{align}
        \left\vert \int_{t_0}^t\int_{\R^3}\nabla p\cdot |v|^{2}v\;dxds\right\vert \nonumber &\leq \int_{t_0}^{t_0+S}\int_{\R^3}|\mathcal{R}_i\mathcal{R}_j(v_iv_j)\cdot v_kv_l\partial_kv_l|\;dxds +\int_{t_0}^{t_0+S}\int_{\R^3}|\nabla p_f\cdot |v|^{2}v|\;dxds
        \\ &\lesssim I_{v\otimes v}+ \int_{t_0}^{t_0+S}\norm{\nabla p_f}{L^4(\R^3)}\norm{v}{L^4(\R^3)}^{3}\;ds
        \\&\lesssim I_{v\otimes v}+\int_{t_0}^{t_0+S}\norm{f}{L^4(\R^3)}\norm{v}{L^4(\R^3)}^{3}\;ds \lesssim I_{v\otimes v}+(\mathcal{E}_{4,t_0}(S))^\frac{3}{4}\norm{f}{L^1_tL^4_x(\R^3\times(t_0,t_0+S))}.\label{prssenergydecomp}
    \end{align}
    Here, $I_{v\otimes v}:=\int_{t_0}^t\int_{\R^3}|\mathcal{R}_i\mathcal{R}_j(v_iv_j)\cdot v_kv_l\partial_kv_l|\;dxds$. 
    Testing the Navier-Stokes equations against $v|v|^{2}$ and using \eqref{prssenergydecomp} then gives:
    \begin{align}
        &\mathcal{E}_{4,t_0}(S) \leq \int_{\R^3}|v(x,t_0)|^4\;dx +C\left(I_{v\otimes v}+(\mathcal{E}_{4,t_0}(S))^\frac{3}{4}\norm{f}{L^1_tL^4_x(\R^3\times(t_0,t_0+S))}\right)+ 4\int_{t_0}^{t_0+S}\int_{\R^3}f\cdot v|v|^{2}dxds
    \end{align}
    We estimate the final forcing term using H\"older's inequality:
    \begin{align}\nonumber\left| \int_{t_0}^{t_0+S}\int_{\R^3}f\cdot v|v|^{2}dxds\right| \leq \int_{t_0}^{t_0+S}\norm{f}{L^4(\R^3)}\norm{v}{L^4(\R^3)}^{3}\;ds &\leq  \left(\sup_{t_0<s<t_0+S}\norm{v(\cdot,s)}{L^4(\R^3)}^{3}\right)\int_{t_0}^{t_0+S}\norm{f}{L^4(\R^3)}\;ds
    \\&\leq (\mathcal{E}_{4,t_0}(S))^\frac{3}{4}\norm{f}{L^1_tL^4_x(\R^3\times(t_0,t_0+S))}. 
    \end{align}
    Finally, $I_{v\otimes v}$ may be treated exactly as in \cite{Spatialconc}.
    This gives, for some universal $C>0$,
    \begin{align}
        \mathcal{E}_{4,t_0}(S) &\leq 
        \norm{v(\cdot,t_0)}{L^4(\R^3)}^4+ C\norm{v}{L^5(\R^3\times(t_0,t_0+S))}\mathcal{E}_{4,t_0}(S) + C(\mathcal{E}_{4,t_0}(S))^\frac{3}{4}\norm{f}{L^1_tL^4_x(\R^3\times(t_0,t_0+S))}.
    \end{align}
    Using Young's inequality, 
    \begin{align}
        \mathcal{E}_{4,t_0}(S) &\leq \norm{v(\cdot,t_0)}{L^4(\R^3)}^4+ C\norm{v}{L^5(\R^3\times(t_0,t_0+S))}\mathcal{E}_{4,t_0}(S) + \frac{3\mathcal{E}_{4,t_0}(S)}{4}+C\norm{f}{L^1_tL^4_x(\R^3\times(t_0,t_0+S))}^4.
    \end{align}
    Rearranging this yields the claimed result.
\end{proof}
In the following, we show a slightly supercritical regularity criteria for the forced Navier-Stokes equations, using the same method as \cite[Theorem 1]{Mildcriticality} (which was inspired by \cite{Bulut}).  
\begin{Proposition}[Mild criticality breaking for the forced Navier-Stokes equations]
    \label{General forced subcritical information transfer}
    For $A$ sufficiently large and $N\geq A$, there exists $\mu_0:=\mu_0(A,N)\in(0,\frac{1}{2}]$ such that the following holds.  Let $0<T<1$ and let $v$ be a smooth Leray-Hopf weak solution to the forced Navier-Stokes equations with smooth forcing $f$ on $\R^3\times(0,T)$. 
    If $0<\mu<\mu_0$ and \begin{align}\label{mildcritassump}
        \norm{v_0}{L^{4}(\R^3)} + \norm{f}{L^2_tH^1_x\cap L^6_{t,x}(\R^3\times(0,T))}  &\leq A, 
        &\norm{v}{L^\infty_tL^{3-\mu}_x(\R^3\times(0,T))} &\leq N,
    \end{align} then $v\in L^\infty_t L^4_x(\R^3\times(0,T))$.

\end{Proposition}
\begin{proof}
    Let $A,N$ be sufficiently large and let $0<\mu<\mu_0$, where $\mu_0\in(0,\tfrac{1}{2}]$ is to be determined later. First, we wish to apply $L^4$ local-in-time existence theory, for which we need to estimate $\norm{f}{L^1_tL^4_x(\R^3\times(0,T))}.$ Using Lebesgue interpolation ($L^4$ between $L^2$ and $L^6$) and H\"older's inequality, 
    \begin{align}
        \norm{f}{L^1_tL^4_x(\R^3\times(0,T))} \leq T^\frac{3}{4}\norm{f}{L^2_{t,x}(\R^3\times(0,T))}^\frac{1}{4}\norm{f}{L^6_{t,x}(\R^3\times(0,T))}^\frac{3}{4} \leq A.\label{forcingl1l4}
    \end{align}
    We can apply\footnote{This gives existence \cite[Theorem 3.2 and 4.5]{FJR72} of a local-in-time solution $\bar{v}$ with the property \eqref{initinalintialestim}, which coincides with our smooth solution via strong-strong uniqueness \cite{Leray}.} \cite{FJR72} with data $       
   g(\cdot,s):=e^{s\Delta}v(\cdot,0)+\int_{0}^s e^{(s-\tau)\Delta}\LP fd\tau$ to see that there exists a time $0<t'\leq\min \left(T,C_{univ}A^{-8}\right)<1$ such that
   \begin{align} 
        \norm{v}{L^\infty_tL^4_x(\R^3\times(0,t'))} &\leq 2\norm{g}{L^\infty_t L^4_x(\R^3\times(0,t'))}
        \leq 2\norm{v_0}{L^4(\R^3)} + C_{univ}\norm{f}{L^1_tL^4_x(\R^3\times(0,T))}\leq A^2.\label{initinalintialestim}
   \end{align} 
   If $T<C_{univ}A^{-8},$ then we are done. Else, we have shown that there exists a time $t':=C_{univ}A^{-8}$ so that, provided that we choose $K(A,N) \geq A^2$,
    \begin{align}\norm{v}{L^\infty_tL^4_x(\R^3\times(0,t'))}\leq  K(A,N). \label{initialmildL4estim}\end{align}
   It also follows (for example, applying \cite[Theorem 3.2]{FJR72}; also see \cite[p.196]{Giga})  from \eqref{mildcritassump} and \eqref{forcingl1l4} that
    \begin{align}
        \norm{v}{L^5_{t,x}(\R^3\times (0,t'))}\leq (t')^\frac{1}{8}\norm{v}{L^\frac{40}{3}_tL^5_x(\R^3\times (0,t'))} \leq 2(t')^\frac{1}{8}\norm{g}{L^\frac{40}{3}_tL^5_x(\R^3\times (0,t'))}\leq C_{univ} .\label{InitiallyL5BND}
    \end{align}
    We may use this to begin the continuity-type argument originally seen in \cite{Bulut}, and subsequently in \cite{Mildcriticality}. Our goal is to show that, under the assumptions of Proposition \ref{General forced subcritical information transfer} and $K(A,N)\geq A^2$ to be determined, if $0<t'<T$ is a time for which 
    \begin{align}
       \norm{v}{L^\infty_tL^4_x(\R^3\times(0,t'))}\leq K,
    \end{align}
    then under assumption \eqref{mildcritassump} there exists a time $t'<t''<T$ so that 
    \begin{align}
        \norm{v}{L^\infty_tL^4_x(\R^3\times(0,t''))}\leq K. \label{continuityarg}
    \end{align}
    This then implies that $v\in L^\infty_tL^4_x(\R^3\times(0,T))$ as required. 
    
    As $\norm{v(\cdot, t')}{L^4}\leq K(A,N)$, applying local-in-time existence theory again tells us that there exists $t'<t''<T$ so that 
    \begin{align} 
        \norm{v}{L^\infty_tL^4_x(\R^3\times(t',t''))} 
        \leq 2\norm{v(\cdot,t')}{L^4(\R^3)} + C_{univ}\norm{f}{L^1_tL^4_x(\R^3\times(0,T))}
        \leq 2K(A,N) + C_{univ}A \leq 4K(A,N). \label{intialmildestimforcritbreak}
    \end{align} 
    This implies that $\norm{v}{L^\infty_tL^4_x(\R^3\times(0,t''))} \leq 4K(A,N). $ 
    Let $\mu\in(0,\frac{1}{2}]$ be determined later. By Lebesgue interpolation, see that
    \begin{align}\norm{v}{L^{\infty}_tL^3_x(\R^3\times(0,t''))} \leq \norm{v}{L^{\infty}_tL^{3-\mu}_x(\R^3\times(0,t''))}^{\frac{3-\mu}{3+3\mu}}\norm{v}{L^\infty_tL^4_x(\R^3\times(0,t''))}^\frac{4\mu}{3+3\mu}\leq N^\frac{3-\mu}{3+3\mu}(4K)^\frac{4\mu}{3+3\mu}=:M.\end{align}   
    Using \eqref{mildcritassump} and $N\geq A$, we may apply Proposition \ref{Rescaledquantestimtime} (with $t_0=0$ and $t_1=t''$) to see that for $t\in(0,t'')$: 
    \begin{equation} \label{L5 in terms of E,B}
        \norm{v(\cdot,t)}{L^\infty(\R^3)} \leq \frac{1}{t^\frac{1}{2}}\exp{\left(\exp{\left(\exp{\left((2M)^{613}\right)}\right)}\right)}.
    \end{equation}
    Interpolating this estimate with the $L^3$ norm for $v$ gives for $t\in(0,t'')$
    \begin{align}
        \norm{v(\cdot,t)}{L^5(\R^3)}\lesssim\frac{1}{t^\frac{1}{5}}\exp{\left(\exp{\left(\exp{\left(M^{614}\right)}\right)}\right)}.
    \end{align}
    Integrating in time over $(t',t'')$, (and using $M, N\geq A$) we have\footnote{Since $0<t''-t'<1$.}
     \begin{align}
        \norm{v}{L^5_{t,x}(\R^3\times (t',t''))}&\lesssim\left(\log(t'')-\log(t')\right)^\frac{1}{5}\exp{\left(\exp{\left(\exp{\left(M^{614}\right)}\right)}\right)}\nonumber
        \\ &\lesssim \left(-\log(C_{univ}A^{-8})\right)^\frac{1}{5}\exp{\left(\exp{\left(\exp{\left(M^{614}\right)}\right)}\right)}\leq \exp{\left(\exp{\left(\exp{\left(M^{615}\right)}\right)}\right)}.
    \end{align}
    Combining with the initial estimate (\ref{InitiallyL5BND}), we see that
     \begin{align}\label{L5 M estim}
        \norm{v}{L^5_{t,x}(\R^3\times (0,t''))} &\leq \exp{\left(\exp{\left(\exp{\left(M^{615}\right)}\right)}\right)}.
    \end{align}
    Now, let $\varepsilon:=\min\left(\frac{1}{2C},1\right)$ where $C$ is the constant in line (\ref{L^p specific energy est}). Similarly to \cite{Bulut}, \cite{Mildcriticality}, we slice the time interval $(0,t'')$ into $m+1$ successive disjoint intervals $I_j=(t_j,t_{j+1})=(t_j,t_j+S_j)$ such that
    \begin{equation}
        (0,t'') = \bigcup\limits_{j\in \lbrace 0, \cdots , m \rbrace} I_j, \qquad \text{and} \qquad I_j\cap I_k = \emptyset \qquad \forall j \neq k,
    \end{equation}
    (where equality is up to the set of interval endpoints which has measure zero) and 
    \begin{equation}
        \norm{v}{L^5(I_j;L^5(\R^3))} =\varepsilon \qquad \forall j\in \lbrace 0, \cdots , m-1 \rbrace, \qquad \text{and} \qquad \norm{v}{L^5(I_{m};L^5(\R^3))} \leq \varepsilon.
    \end{equation}
    This time-slicing in conjunction with the estimate (\ref{L5 M estim}) provides an upper bound on $m$; indeed, see that
    \begin{align}
        m\varepsilon^5 &= \sum_{j=0}^{m-1}\norm{v}{L^5(I_j;L^5(\R^3))}^5 \leq \norm{v}{L^5(0,t;L^5(\R^3))}^5 \leq \exp{\left(\exp{\left(\exp{\left(\left(N^\frac{3-\mu}{3+3\mu}(4K)^\frac{4\mu}{3+3\mu}\right)^{616}\right)}\right)}\right)}.\label{estimate for m}
    \end{align}
    Just as in \cite{Mildcriticality}, we may iteratively apply our $L^4$ energy estimate (Lemma \ref{L^p energy estimate}) on the intervals $I_j$ for $0\leq j \leq m$. With the choice of $\varepsilon = \min\left(\frac{1}{2C},1\right)$ (where $C$ is from Lemma \ref{L^p energy estimate}), and using our estimate for $m$ (\ref{estimate for m}), we obtain
     \begin{align}
        \norm{v}{L^\infty(0,t'';L^4(\R^3))}^4 \leq \exp{\left(\frac{1}{\varepsilon^5}\exp{\exp{\left(\exp{\left(N^{617}(4K)^\frac{4\mu \cdot 617}{3+3\mu}\right)}\right)}}\right)},
    \end{align}
    where we have used the observation that $\mu\in(0,\frac{1}{2}]$ implies $\frac{3-\mu}{3+3\mu}\leq 1$.
    Next, we judiciously choose $K(A,N)$ and $\mu_0(A,N)$; we define:
    \begin{align}
         K(A,N) := &\exp{\left(\frac{1}{4\varepsilon^5}\exp{\exp{\left(\exp{\left(eN^{617}\right)}\right)}}\right)}, \label{K definition}
    \end{align}
   and
   \begin{align}
       0<\mu<\mu_0(A,N) := \min\left( \frac{3}{4\cdot617\log(4K(A,N))-3}, \frac{1}{2}\right).
   \end{align}
   This choice of $\mu_0$ ensures that,  
   $$N^{617}(4K)^\frac{4\mu \cdot 617}{3+3\mu} \leq eN^{617}.$$
   Then our choice of $K(A,N)$ gives us that
   \begin{align}
       \norm{v}{L^\infty(0,t'';L^4(\R^3))} \leq K(A,N).
   \end{align}
   This concludes the continuity argument, so (\ref{continuityarg}) holds for all $0<t''<T,$ so we have that $\norm{v}{L^\infty(0,T;L^4(\R^3))} \leq K(A,N).$ 
\end{proof}
\begin{Remark}\label{Small mu remark}
    (Compare with \cite[Remark 6]{Mildcriticality}).
    Observe that, in particular, the proposition holds if $N\geq A$ with $A$ sufficiently large and 
    \begin{align}
        0< \mu \leq \frac{1}{\exp\exp\exp(N^{618})}.
    \end{align}
\end{Remark}
\subsection{Proof of local slightly supercritical Orlicz blow-up}\label{Localisation Procedure}\label{Localized Orlicz-type blow-up}
We use the following Lemma, which is based on the work of \cite{TruncationMethod}.
\begin{Lemma}[Truncation procedure]\label{Truncation procedure}
        Let $(v,p)$ be a smooth Leray-Hopf weak solution to (\ref{Navier-Stokes equations}) on $\R^3\times (-1,0)$. There exist radii $0<r_0<r_1<1,$ and functions $(V,\Pi)$ which are a smooth Leray-Hopf weak solution to (\ref{Navier-Stokes equations}) with smooth forcing $F$ (defined below) on $\R^3\times (-1,0)$ such that:
        \begin{enumerate}
            \item $V= \Phi v+w$, $\Pi=\Phi p$ where $\Phi\in C^\infty_c(B_{r_1};[0,1])$ with $\Phi=1$ on $B_{r_0}$ and $w$ and $F$ are supported in $B_{r_1}\setminus B_{r_0}.$
            \item $w\in L^\infty_tW^{k,\infty}_x((-\tfrac{1}{2},0)\times\R^3) \label{w estim}$ for $k\in\N_0$.
            \item $F\in L^6_tW^{k,\infty}_x((-\tfrac{1}{2},0)\times\R^3)$ for $k\in\N_0$. 
        \end{enumerate}   
    \end{Lemma}
        See, for example, \cite[Proposition 2.2]{curvedbndry}. 
        As $v$ is a suitable weak Leray-Hopf solution, as shown in \cite{TruncationMethod}, there exists $0<r_0'<r_1'<1$ such that \begin{align}v\in L^\infty(-\tfrac{1}{2},0;W^{k,\infty}(B_{r_1'}\setminus B_{r_0'}))\label{v annular reg NP}\end{align} Let $0<r_0'<r_0<r_1<r_1'<1$, from which we define $\Phi$ according to (1). Then, the function $w$ is defined via the Bogovskii operator (for example, see \cite[Chapter III]{galdi2011}). It satisfies \begin{align}
            \begin{cases}
                \operatorname{div}w=-\nabla\Phi\cdot v,
                \\ w\big\rvert_{\partial(B_{r_1}\setminus B_{r_0})}=0 .
            \end{cases}
        \end{align}
        Using \cite[Lemma III.3.1 and Remark III.3.3]{galdi2011}, $w$ inherits spatial and temporal derivatives of $v$ through the Bogovskii operator. This allows one to deduce that $(V,\Pi)$ is smooth. We see that they satisfy the Navier-Stokes equations with forcing 
        \begin{align}
            F:=&\left(\partial_t\Phi -\Delta\Phi\right)v - 2\nabla\Phi\cdot\nabla v+\left(\Phi^2-\Phi\right)v\cdot\nabla v \\&+\partial_tw-\Delta w+\Phi v\cdot\nabla w+ w\cdot\nabla\left(\Phi v\right) +w\cdot\nabla w +(\nabla\Phi)p. \nonumber
        \end{align}
        All terms other than the pressure term may be treated as in \cite{TruncationMethod} by using \eqref{v annular reg NP} and are contained in $L^\infty_t W^{k,\infty}_x((t_0,0)\times \R^3)$.
        We now outline how to deduce the improved time regularity for the forcing. 
        In our setting, the pressure has more time regularity  than in the setting of \cite{curvedbndry}. Indeed, here $p=\mathcal{R}_i\mathcal{R}_j(v_iv_j)$, so the fact that $v$ is a Leray-Hopf weak solution implies (by Calderón-Zygmund estimates) that $p\in L^6_tL^\frac{9}{8}_x(\R^3\times(-1,0))$. This is greater than the time regularity $L^\frac{3}{2}_t$ assumed for suitable weak solutions as in \cite{curvedbndry}.
        Now, as $-\Delta p=\operatorname{div}(v\cdot\nabla v)$ and $p\in L^6_tL^1_x(Q_1)$, we may deduce from \eqref{v annular reg NP} and elliptic regularity theory that $(\nabla\Phi )p \in L^6(t_0,0;W^{k,\infty}(\R^3))$, hence Lemma \ref{Truncation procedure} (3) holds.     

\begin{Theorem}\label{SlightlysupercriticalThmproof}
    There exists a universal constant $\theta\in (0,1)$ such that the following holds. Suppose that $v:\R^3\times (0,\infty)\to\R^3$ is a Leray-Hopf weak solution of the Navier-Stokes equations on $\R^3\times (0,\infty)$ which first loses smoothness at time $T_*>0$ with $(x_*,T_*)$ being a singular point. Then for all $\delta>0,$
    \begin{align}
        \limsup_{t\uparrow T_*}\int_{B_\delta(x_*)}\dfrac{|v(x,t)|^3}{\left(\log\log\log\left(\left(\log\left(e^{e^{3e^e}}+|v(x,t)|\right)\right)^\frac{1}{3}\right)\right)^\theta}\;dx = \infty.
    \end{align}
\end{Theorem}
We follow the proof of \cite[Theorem 2]{Mildcriticality} - we will apply the truncation procedure Lemma \ref{Truncation procedure}, and apply the slightly supercritical regularity criteria for the \textit{forced} Navier-Stokes equations instead. 
\begin{proof}
    Let $\theta \in (0,1)$ be a fixed universal constant that will be determined later, and let $\delta>0$. By translating, we may assume $x_*=0$. For contradiction, assume that there exists a time $0<t_*<T_*$ so that 
    \begin{align}
        \sup_{t_*<t<T_*}\int_{B_\delta}\dfrac{|v(x,t)|^3}{\left(\log\log\log\left(\left(\log\left(e^{e^{3e^e}}+|v(x,t)|\right)\right)^\frac{1}{3}\right)\right)^\theta}\;dx &\leq L < \infty.
        \end{align}
    Without loss of generality, we may assume that $L\geq 1$ is as large as required later, and we may also assume that $0<\delta<\min(1,\sqrt{T_*-t_*})$.
    We will show that $v$ is bounded all the way up to time $t=T_*$, hence blow-up cannot occur then.
    \\\underline{\textbf{Step 1: Estimating the $L^{3-\mu}(B_\delta)$ norm of $v$.}} This step is similar to \cite{Mildcriticality}.
    \\ Let $\lambda\in (1,\infty)$ and $\mu \in (0,\tfrac{1}{2}]$ be free constants to be determined later. Just as in \cite{Mildcriticality}, decompose $v$ by heights:
    \begin{align}
        v(x,t) &:= v_{>\lambda}(x,t) + v_{\leq \lambda}(x,t)
        := v(x,t)\chi_{\left\lbrace x\in B_\delta : |v(x,t)|>\lambda\right\rbrace} +v(x,t)\chi_{\left\lbrace x\in B_\delta : |v(x,t)|\leq\lambda\right\rbrace}. 
    \end{align}
    We split the Orlicz-type integral into two corresponding parts:
    \begin{align}
        \int_{B_\delta}\dfrac{|v(x,t)|^3}{\left(\log\log\log\left(\left(\log\left(e^{e^{3e^e}}+|v(x,t)|\right)\right)^\frac{1}{3}\right)\right)^\theta}\;dx =: I_{>\lambda} + I_{\leq \lambda},
    \end{align}
    where 
    \begin{align}
        I_{>\lambda} := \int_{\left\lbrace x\in B_\delta : |v(x,t)|>\lambda\right\rbrace}\dfrac{|v(x,t)|^3}{\left(\log\log\log\left(\left(\log\left(e^{e^{3e^e}}+|v(x,t)|\right)\right)^\frac{1}{3}\right)\right)^\theta}\;dx.
    \end{align}
    First, we may estimate $I_{>\lambda}$ exactly as in \cite[(34)-(37)]{Mildcriticality} to obtain, for all $t\in(t_*,T_*),$
    \begin{align}
        \int_{\left\lbrace x\in B_\delta : |v(x,t)|>\lambda\right\rbrace} |v(x,t)|^{3-\mu}\;dx \leq \frac{L\log(\lambda)}{\lambda^\mu}, \label{upperheight estim}
    \end{align}
    provided that 
    \begin{align}
        \lambda\geq \max(C(\theta),2,e^\frac{1}{\mu}).\label{lambda large}
    \end{align}
    Next, since $\delta \in (0,1)$, $\mu \in (0,\tfrac{1}{2}]$ and $L\geq 1$, we can apply H\"older's inequality and $I_{\geq\lambda}\leq L$ to obtain that for all $t\in(t_*,T_*)$:
    \begin{align}
        \norm{v_{\leq \lambda}(\cdot,t)}{L^{3-\mu}(B_\delta)}\leq \frac{4\pi}{3}\delta^\frac{\mu}{3-\mu}\norm{v_{\leq \lambda}(\cdot,t)}{L^3(B_\delta)} &\leq  \frac{4\pi}{3}\delta^\frac{\mu}{3-\mu}L^\frac{1}{3} \left(\log\log\log\left(\left(\log\left(e^{e^{3e^e}}+\lambda\right)\right)^\frac{1}{3}\right)\right)^\frac{\theta}{3} 
        \\ &\leq \frac{4\pi}{3}L\left(\log\log\log\left(\left(\log\left(e^{e^{3e^e}}+\lambda\right)\right)^\frac{1}{3}\right)\right)^\frac{\theta}{3} .\label{lowerhieght estim}
    \end{align}
    For $\lambda\geq \max(C(\theta),2,e^\frac{1}{\mu})$, $\mu \in (0,\tfrac{1}{2}]$, we may combine (\ref{upperheight estim}) and (\ref{lowerhieght estim}) to get:
    \begin{align}
        \norm{v}{L^\infty_t L^{3-\mu}_x((t_*,T_*)\times B_\delta)}&\leq \max\left(\left(\frac{\log(\lambda)L}{\lambda^\mu}\right)^\frac{1}{2},\left(\frac{\log(\lambda)L}{\lambda^\mu}\right)^\frac{1}{3}\right) + \frac{4\pi}{3}L\left(\log\log\log\left(\left(\log\left(e^{e^{3e^e}}+\lambda\right)\right)^\frac{1}{3}\right)\right)^\frac{\theta}{3} . \nonumber
    \end{align}
    \underline{\textbf{Step 2: set up to apply the mild criticality breaking Lemma.}} This step is similar to \cite{ Mildcriticality}.
    \\ Define $\theta:= \frac{1}{618}$ (see in Remark \ref{Small mu remark}) and define $\mu:= \frac{1}{(\log(\lambda))^\frac{1}{2}}$ for $\lambda$ sufficiency large. In particular, $\lambda\geq \max(e, C(\theta))$ (where $C(\theta)$ comes from \eqref{lambda large}) implies that   $\lambda \geq e^\frac{1}{\mu}$ and hence $\lambda$ satisfies \eqref{lambda large}.
    Using the observation that $\lim_{\lambda\to\infty}\frac{\log(\lambda)}{\lambda^\mu}=0,$ we see that there exists a sufficiently large $\lambda_0(L,\delta)\geq 0$ such that for all $\lambda>\lambda_0(L,\delta),$
    \begin{align}
        \norm{v}{L^\infty_t L^{3-\mu}_x((t_*,T_*)\times B_\delta)}&\leq \max\left(\left(\frac{\log(\lambda)L}{\lambda^\mu}\right)^\frac{1}{2},\left(\frac{\log(\lambda)L}{\lambda^\mu}\right)^\frac{1}{3}\right) \nonumber+ \frac{4\pi}{3}L\left(\log\log\log\left(\left(\log\left(e^{e^{3e^e}}+\lambda\right)\right)^\frac{1}{3}\right)\right)^\frac{1}{3\cdot 618}
        \\ &\leq \frac{1}{2}\delta^\frac{1}{5}\left(\log\log\log\left(\left(\log\left(e^{e^{3e^e}}+\lambda\right)\right)^\frac{1}{3}\right)\right)^\frac{1}{618} =:N_\lambda.
    \end{align}
    Now, if we ensure that $\lambda_0$ is sufficiently large, we have:
    $$\left(\log\left(e^{e^{3e^e}}+\lambda \right)\right)^\frac{1}{3}\leq (\log(\lambda))^\frac{1}{2}= \frac{1}{\mu}.$$
    This gives us that 
    \begin{align}
        0<\mu \leq \frac{1}{\left(\log\left(e^{e^{3e^e}}+\lambda \right)\right)^\frac{1}{3}} =  \frac{1}{\exp\exp\exp\left(\left(2\delta^{-\frac{1}{5}}N_\lambda\right)^{618}\right)}.\label{small mu achieved}
    \end{align}
    In the following, we notice that as $\lambda \to \infty,$ we have $N_\lambda\to \infty$ and $\mu =\mu(N_\lambda;\lambda)\to 0.$
    \\\\ \underline{\textbf{Step 3: apply the mild criticality breaking.}} Here, we implement the truncation procedure Lemma \ref{Truncation procedure}, then proceed as in \cite{ Mildcriticality} to apply the mild criticality breaking Proposition \ref{General forced subcritical information transfer}.
   \\ Recall that $0<\delta< \min(\sqrt{T_*-t_*},1)$. Set $\hat{v}(x,t):= \delta v(\delta x,\delta^2 t+T_*)$ for $(x,t)\in \R^3\times (-1,0)$, which is a smooth suitable Leray-Hopf weak solution to the Navier-Stokes equations on $\R^3\times (-1,0)$ and satisfies:
    \begin{gather}
        \norm{\hat{v}}{L^\infty_t L^{3-\mu}_x(Q_1)}\leq \delta^{-\frac{\mu}{3-\mu}}\norm{v}{L^\infty_t L^{3-\mu}_x(Q_\delta(0,T_*))}\leq \delta^{-\frac{1}{5}}N_\lambda =: \frac{\hat{N}_\lambda}{2}, \label{reslcaedsupercritassump1}
    \end{gather}
    Now we may apply the truncation procedure Lemma \ref{Truncation procedure} to the solution $\hat{v}$ to form a smooth suitable weak solution $(V,\Pi)$ on $\R^3\times(-1,0)$ to the Navier-Stokes equations with truncation forcing $F$, satisfying properties (1)-(3) of Lemma \ref{Truncation procedure}. 
    By part (2) and (3) respectively, $w\in L^\infty_t L^{3-\mu}_x(\R^3\times(t_0,0))$ and $F\in L^2_tH^1_x\cap L^6_{t,x}(\R^3\times (t_0,0))$. Therefore, ensuring that  $\lambda>\lambda_0(L,\delta, w,V,F)$ is sufficiently large, we have 
    \begin{align}
         \norm{w}{L^\infty_t L^{3-\mu}_x(\R^3\times(-\frac{1}{2},0))} \leq \frac{\hat{N}_\lambda}{2},  \quad \text{and}\quad  \norm{F}{L^2_tH^1_x\cap L^6_{t,x}(\R^3\times (-\frac{1}{2},0))}\leq \frac{\hat{N}_\lambda}{2}. \label{initialV1}
    \end{align}
    Using part 1 of Lemma \ref{Truncation procedure} and  (\ref{reslcaedsupercritassump1})-\eqref{initialV1}, we also have:
    \begin{align}
        \norm{V}{L^\infty_t L^{3-\mu}_x(\R^3\times(-\frac{1}{2},0))} \leq \norm{\Phi \hat{v}}{L^\infty_t L^{3-\mu}_x(Q_1)} + \norm{w}{L^\infty_t L^{3-\mu}_x(\R^3\times(-\frac{1}{2},0))} \leq \frac{\hat{N}_\lambda}{2}+\frac{\hat{N}_\lambda}{2}=\hat{N}_\lambda.
    \end{align} 
    Additionally, since $V$ is smooth on $\R^3\times (-1,0)$, we also have that for $\lambda>\lambda_0(L,\delta, w,V,F)$ sufficiently large:
    \begin{align}
         \norm{V(\cdot,-\tfrac{1}{2})}{L^4(\R^3)}\leq  \frac{\hat{N}_\lambda}{2}.
    \end{align}
    Since (\ref{small mu achieved}) holds true, we may finally apply\footnote{Translating in time, we apply the Proposition to $\tilde{V}(x,t):= V(x,t-\tfrac{1}{2})$ with where we take ``$T=\tfrac{1}{2}$ and $A=N=\hat{N}_\lambda$", to find that $\tilde{V}\in L^\infty_tL^4_x(\R^3\times(0,\tfrac{1}{2}))$.} Proposition \ref{General forced subcritical information transfer} in light of Remark \ref{Small mu remark} from which it follows that $V\in L^\infty_tL^4_x(\R^3\times(-\tfrac{1}{2},0))$.
    Using part (1) of Lemma \ref{Truncation procedure}, we then see that $\hat{v}\in L^\infty_tL^4_x(B_{r_0}\times (-\tfrac{1}{2},0)).$
    Returning back to $v$, we see that $v\in L^\infty_tL^4_x(B_{\delta r_0}\times (T_*-\tfrac{\delta^2}{2},T_*))$. Therefore $v$ satisfies the Prodi-Serrin condition so must be bounded in a backward parabolic neighbourhood of $(x_*,T_*)$ (e.g. see \cite{Serrin62}), which is a contradiction.
\end{proof} 

\section{Application of the quantitative estimate II: {$L^3$} blow-up rate for the Boussinesq equation}\label{Bouss Aplic}

\begin{Theorem}
    Suppose that $(v,\theta, p)$ is a smooth solution to the Boussinesq equations on $\R^3\times (0,T^*)$ corresponding to initial data $v_0\in L^2\cap L^\infty$ and $\theta_0\in L^2\cap L^\infty$ such that
    \begin{enumerate}
        \item $v$ is a Leray-Hopf weak solution to the Navier Stokes equations on $\R^3\times (0,T^*)$ with forcing $f:=\theta \textbf{e}_3$ in the sense of Definition \ref{LerayHopfWeak}.
        \item $\theta$ is smooth with $\theta \in L^2_t\dot{H}^1_x\cap C_tL^2_x\cap L^6_{t,x}(\R^3\times (0,T^*))$.
        \item $T^*>0$ is the first time such that $\lim\limits_{t\uparrow T^*}\norm{v(\cdot,t)}{L^\infty(\R^3)}=\infty.$
    \end{enumerate}
    Then we have the following quantitative blow-up rate for the $L^3$ norm of $v$:
    \begin{align}
        \limsup\limits_{t\uparrow T^*}\dfrac{\norm{v(\cdot,t)}{L^3(\R^3)}}{\left(\log\log\log\left(\dfrac{1}{(T^*-t)^\frac{1}{4}}\right)\right)^\frac{1}{614}} = \infty.
    \end{align}
\end{Theorem}
Similar arguments may be seen in \cite[Theorem 1.4]{Tao21} and \cite[Theorem 2]{Localizedblowup}. Here, we use an a-priori $L^p$ estimate for the temperature, which gives us a bound on the forcing term. This then allows us to consider the Boussinesq equations as forced Navier-Stokes equations, and we may apply the forced quantitative regularity estimate Proposition \ref{Rescaledquantestimtime} in the same way as \cite[Theorem 2]{Localizedblowup}.
\begin{proof}
    Assume that the opposite holds; then there exists $N\in (1,\infty)$ such that 
    \begin{align} \label{Contropositive to quantitative Bous blow-up}
        \sup\limits_{0<s<t}\norm{v(\cdot,s)}{L^3(\R^3)} \leq N\left(\log\log\log\left(\dfrac{1}{(T^*-t)^\frac{1}{4}}\right)\right)^\frac{1}{614} \qquad \forall t\in (0,T^*).
    \end{align}
    Next, define $F:= \theta \mathbf{e}_3.$
    Testing the equation for $\theta$ with $\theta |\theta|^{p-2}$, integrating by parts and using that $v$ is divergence-free gives us the following $L^p$ energy estimate for $\theta$ for $p>1:$
    \begin{align}
        \norm{\theta}{L^\infty_t L^p_x(\R^3\times(0,T^*))}^p + p(p-1)\int_0^{T^*}\int_{\R^3}|\nabla\theta|^2|\theta|^{p-2}\;dxdt \leq \norm{\theta_0}{L^p(\R^3)}^p.
    \end{align}
    In particular, for $p=2$ and $p=6$, we obtain
    \begin{align} \label{estimate for Boussinesq forcing}
        \norm{F}{L^2_tH^1_x \cap L^6_{t,x}(\R^3\times(0,T^*))} \leq \max\left((T^*)^\frac{1}{2}\norm{\theta_0}{L^2(\R^3)}, \norm{\theta_0}{L^2(\R^3)}, (T^*)^\frac{1}{6}\norm{\theta_0}{L^6(\R^3)} \right) =: \mathcal{F}.
    \end{align}
    Now, fix $\varepsilon_0>0$ such that 
    \begin{align} \label{epsilon small compared to N}
        \dfrac{N}{\left(\log\log\log\left(\varepsilon_0^{-\frac{1}{4}}\right)\right)^{\frac{1}{613}-\frac{1}{614}}} \leq \frac{1}{2}
    \end{align}
    and 
    \begin{align} \label{M suff large by making epsilon suff small}
        \log\log\log\left(\varepsilon_0^{-\frac{1}{4}}\right)^\frac{1}{613} \geq 2\max\left(M_0,\mathcal{F}\right).
    \end{align}
    Here, $M_0$ is a sufficiently large universal constant.
    Then combining lines (\ref{Contropositive to quantitative Bous blow-up})  and (\ref{epsilon small compared to N}) gives us that
    \begin{align}
     \sup\limits_{0<s<t}\norm{v(\cdot,s)}{L^3(\R^3)} +\norm{F}{L^2_tH^1_x \cap L^6_{t,x}(\R^3\times(0,T^*))}  \leq \left(\log\log\log\left(\dfrac{1}{(T^*-t)^\frac{1}{4}}\right)\right)^\frac{1}{613} =: M  \qquad \forall t \in (T^*-\varepsilon_0,T^*).
    \end{align}
    In particular, in this construction, we have $M\geq M_0$.
    Without loss of generality, we may assume that $T^*\leq 1$ (else, replace $0$ by $T^*-1$).
    Then we may apply our quantitative estimate (Proposition \ref{Rescaledquantestimtime}) for each $t\in (T^*-\varepsilon_0,T^*)$. We have:
    \begin{align}
        \norm{v(\cdot,t)}{L^\infty(\R^3)}\leq t^{-\frac{1}{2}}\exp\exp\exp \left(M^{613}\right) = \dfrac{1}{t^\frac{1}{2}(T^*-t)^\frac{1}{4}}.
    \end{align}
    \\In particular, we see that $v\in L^2_tL^\infty_x(\R^3\times (T^*-\varepsilon_0,T^*))$
    and hence satisfies the Prodi-Serrin condition (e.g. see \cite[Theorem 1.1]{Bous-Serrintype}), so $v\in L^\infty(\R^3\times (T^*-\varepsilon_0,T^*))$, contradicting that $\lim\limits_{t\to T^*}\norm{v(\cdot,t)}{L^\infty(\R^3)}=\infty$.
\end{proof}
\begin{Remark}
    Suppose that $(v, p)$ is a smooth Leray-Hopf weak solution to the Navier-Stokes equations with smooth forcing $f$ on $\R^3\times (0,T^*)$ that first blows up at time $T^*>0$, in the sense that there exists a point $x^*$ such that $v\notin L^\infty(Q_r(x^*,T^*))$ for all $r>0$. Following the exact same arguments from \eqref{estimate for Boussinesq forcing} onwards, if we additionally have $f\in L^2_tH^1_x \cap L^6_{t,x}(\R^3\times(0,T^*)) $, then we have the following quantitative blow-up rate for the $L^3$ norm of $v$:
    \begin{align}
        \limsup\limits_{t\uparrow T^*}\dfrac{\norm{v(\cdot,t)}{L^3(\R^3)}}{\left(\log\log\log\left(\dfrac{1}{(T^*-t)^\frac{1}{4}}\right)\right)^\frac{1}{614}} = \infty.
    \end{align}
\end{Remark}
\subsubsection*{Acknowledgements} TB is supported by an EPSRC New Investigator Award UKRI096 `Dynamics and regularity criteria for nonlinear incompressible partial differential equations'. HP is supported by Raoul \& Catherine Hughes (Alumni funds) and the University Research Studentship award EH-MA1333.
\subsubsection*{Data Availability} Data sharing is not applicable to this article as no datasets were generated or analysed during the current study.
\subsubsection*{Declarations} 
\subsubsection*{Conflict of interest} The authors state that there is no conflict of interest.
\bibliographystyle{alpha}
\bibliography{Bibliography}

\end{document}